\documentclass[twoside]{article}

\usepackage[margin=1in]{geometry}
\newcommand{\rene}[1]{  \ifthenelse{\boolean{showcomments}}
{\todo[inline,color=cyan]{Ren\'e: #1}}{}}

\usepackage{url}            
\usepackage{booktabs}       
\usepackage{amsfonts, bm, amssymb}       
\usepackage{nicefrac}       
\usepackage{microtype}      
\usepackage{tikz} 
\usepackage{float}
\usepackage{pifont}

\usepackage{array}
\usepackage{diagbox}
\usepackage{mathtools}
\usepackage{multirow, makecell}
\usepackage{rotating}
\usepackage{graphicx}
\usepackage{setspace}
\usepackage{algorithm, algorithmic}
\usepackage{mathdots}
\usepackage{caption}
\usepackage{subcaption}
\usepackage{float}
\usepackage{enumerate}
\usepackage{enumitem} 
\usepackage{lipsum}
\usepackage{amsmath, amsthm}
\usepackage{mdframed}
\usepackage{thmtools}

\usepackage[bottom]{footmisc}
\usepackage[colorlinks,linkcolor={blue!75!black},urlcolor=red,citecolor={blue!75!black}]{hyperref}
\usepackage{footnotebackref}

\usepackage{threeparttable}
\usepackage{color, colortbl}
\usepackage{arydshln}

\usepackage[sort]{natbib}
\allowdisplaybreaks[4]
\hypersetup{
    colorlinks=true,
    linkcolor=blue,      
    urlcolor=blue
}

\usepackage[colorinlistoftodos,bordercolor=orange,backgroundcolor=orange!20,linecolor=orange,textsize=scriptsize]{todonotes}

\definecolor{OrangeRed}{HTML}{ED135A}
\definecolor{greenContr}{HTML}{57e338}
\newcommand{\xmark}{\ding{55}}
\usepackage{enumitem}

\definecolor{shadecolor}{gray}{0.9}
\declaretheoremstyle[
headfont=\normalfont\bfseries,
notefont=\mdseries, notebraces={(}{)},
bodyfont=\normalfont,
postheadspace=0.5em,
spaceabove=1pt,
mdframed={
  skipabove=8pt,
  skipbelow=8pt,
  hidealllines=true,
  backgroundcolor={shadecolor},
  innerleftmargin=4pt,
  innerrightmargin=4pt}
]{shaded}

\newcommand{\R}{\mathbb{R}}
\newcommand{\Exp}[1]{{\mathbb{E} #1 }}
\declaretheorem[style=shaded,within=section]{definition}
\declaretheorem[style=shaded,sibling=definition]{theorem}
\declaretheorem[style=shaded,sibling=definition]{proposition}

\declaretheorem[style=shaded,sibling=definition]{corollary}

\declaretheorem[style=shaded,sibling=definition]{lemma}

\usepackage{soul}
\def\rene#1{{\color{blue}#1}}

\newcommand\norm[1]{\left\lVert#1\right\rVert}

\newcommand{\Expe}[1]{{\mathbb{E}\left[#1\right] }}  
\newcommand{\Expep}[1]{{\mathbb{E}\left[#1 \Big| \mathcal{F}^k\right]}}
\newcommand{\ceil}[1]{\left\lceil {#1} \right\rceil}
\newcolumntype{?}[1]{!{\vrule width #1}}

\usepackage[accepted]{aistats2024}
\runningtitle{Stochastic Extragradient with Random Reshuffling: Improved Convergence for Variational Inequalities}

\usepackage{graphicx}

\begin{document}

\twocolumn[

\aistatstitle{Stochastic Extragradient with Random Reshuffling:\\ Improved Convergence for Variational Inequalities}

\aistatsauthor{Konstantinos Emmanouilidis \And Ren\'e Vidal \And  Nicolas Loizou}

\aistatsaddress{
\begin{tabular}{c}
    CS \& MINDS \\Johns Hopkins University
\end{tabular} \And \begin{tabular}{c}
    ESE, Radiology \& IDEAS\\
    University of Pennsylvania
\end{tabular}\And \begin{tabular}{c}
    AMS \& MINDS \\Johns Hopkins University
\end{tabular} }
] 

\begin{abstract}
  The Stochastic Extragradient (SEG)
method is one of the most popular algorithms for solving finite-sum min-max
optimization and variational inequality 
problems (VIPs) appearing in various
machine learning tasks. However, existing
convergence analyses of SEG focus on its
with-replacement variants, while practical
implementations of the method randomly
reshuffle components and sequentially
use them. Unlike the well-studied with-replacement variants, SEG with Random
Reshuffling (SEG-RR) lacks established
theoretical guarantees. In this work, we
provide a convergence analysis of SEG-RR
for three classes of VIPs: (i) strongly
monotone, (ii) affine, and (iii) monotone. We derive conditions under
which SEG-RR achieves a faster convergence rate than
the uniform with-replacement sampling
SEG. In the monotone setting, our analysis of SEG-RR guarantees convergence to an arbitrary accuracy without large batch sizes, a strong requirement needed in the classical with-replacement SEG. As a byproduct of our results, we provide convergence guarantees for
Shuffle Once SEG (shuffles the data only
at the beginning of the algorithm) and the
Incremental Extragradient (does not shuffle the
data). We supplement our analysis with experiments validating empirically the superior performance of SEG-RR over the classical with-replacement sampling SEG.
\end{abstract}

\section{Introduction}
Minimax optimization and, more generally, variational inequality problems (VIPs) have received much attention in recent years, especially in the machine learning community. Several machine learning tasks, including Generative Adversarial Networks (GANs) \citep{goodfellow2014generativeadversarialnetworks, pmlr-v70-arjovsky17a}, adversarial training of neural networks \citep{madry2018towards, wang2021adversarial}, reinforcement learning \citep{brown2020combining, DBLP:conf/iclr/SokotaDKLLMBK23}, and distributionally robust learning \citep{namkoong2016stochastic,yu2022fast} are formulated as finite-sum min-max optimization problems,
\vspace{0mm}
\begin{equation}
\label{eq: MinMax}
\min_{x \in\R^{d_1}} \max_{y \in\R^{d_2}} f(x,y)=\frac{1}{n} \sum_{i=1}^n f_i(x,y) \, ,
\end{equation}
with the goal of finding a solution $z^*=(x^*, y^*)^\top$ such that $f(x^*, y) \leq f(x^*, y^*) \leq f(x, y^*), \forall x \in \R^{d_1}, y \in \R^{d_2}$.\!

In this work, we focus on a more abstract formulation of problem \eqref{eq: MinMax}, and we analyze algorithms for solving the following unconstrained finite-sum variational inequality problem (VIP):
find $ z^* \in \mathbb{R}^d$ such that
\begin{equation}\label{VI-problem}
 F(z^*) = \frac{1}{n} \sum_{i = 1}^n F_i(z^*) = 0,
\end{equation}
where $F_i : \mathbb{R}^d \to \mathbb{R}^d, \forall i \in [n], d = d_1 + d_2$. We denote with $\mathcal{Z}_* \subset \mathbb{R}^d$ the solution set of \eqref{VI-problem}. 

Problem \eqref{VI-problem} is quite general and covers a wide range of possible
problem formulations. For example, when the operator $F(x)$ is the gradient of a convex function $f(x)$, then problem \eqref{VI-problem} is equivalent to the minimization of the function $f(x)$. In addition, if the min-max optimization problem~\eqref{eq: MinMax}
has convex-concave continuously differentiable $f$, then using the first-order optimality conditions it can be cast as a special case of \eqref{VI-problem} with $z= (x^\top, y^\top)^\top \in \R^{d}$ and $F(z) = (\nabla_{x}f(x,y)^\top, -\nabla_{y}f(x,y)^\top)^\top$.

In the typical large-scale regime of machine learning applications ($n$ in problem \eqref{VI-problem} is large), stochastic iterative algorithms are preferred mainly because of their cheap per-iteration cost. In that setting,
we only \begin{figure}[t]
\centering 
    \includegraphics[width=\linewidth]{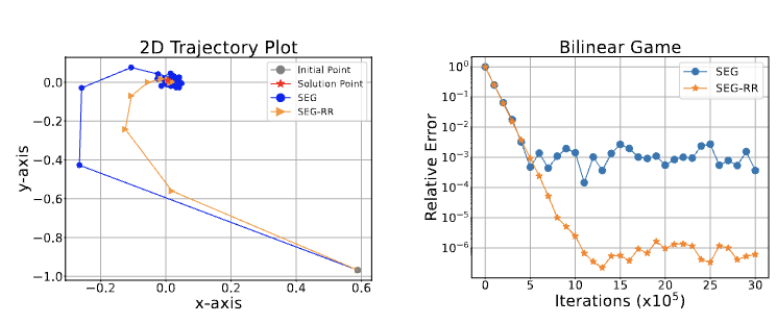}
  \caption{Bilinear Game. Left plot: 2D trajectory plot. Right plot: Relative error $\frac{\|z^k-z^*\|^2}{\|z^0-z^*\|^2}$ as a function of the number of iterations.}
  \label{Fig: 2d-plot}
  \vspace{-0.1cm}
\end{figure}
assume to have access to a stochastic estimate of the operator $F$. Several papers have been devoted to the understanding and convergence analysis of stochastic variants of popular algorithms like the gradient method~\citep{chen1997convergence}, extragradient method~\citep{korpelevich1976extragradient, gorbunov2022extragradient}, and optimistic method~\citep{popov1980modification,gorbunov2022last}. Some recent works in the area include \cite{loizou2021stochastic,beznosikov2022stochastic} on the analysis of stochastic gradient descent ascent (SGDA), \cite{gorbunov2022stochastic, hsieh2020explore, mishchenko2020revisiting} for SEG and \cite{hsieh2019convergence,choudhury2023single} for stochastic past extragradient methods. 

Most existing analyses of stochastic algorithms for solving \eqref{VI-problem} focus on algorithms that use with-replacement sampling in their update rule. Specifically, a component $F_i$ (or a minibatch) of the finite-sum structure of \eqref{VI-problem} is selected uniformly\footnote{Different with-replacement samplings can be used. Here, we use uniform distribution for ease of exposition.} at random in each step. However, most practical implementations of algorithms for solving finite-sum min-max problems and VIPs use without-replacement sampling, creating a gap between practical and theoretically understood approaches. 

In the well-studied problem of solving finite-sum minimization problems (i.e., $\min_x \frac{1}{n} \sum  f_i(x)$), practitioners prefer running popular algorithms that use without-replacement sampling. This is due to the remarkable ease of use and the better empirical performance compared to with-replacement variants~\citep{Bottou2012StochasticGD}. Unfortunately, the fact that the selected samples in an epoch of a without-replacement sampling algorithm are not independent of each other makes the analysis of the method more challenging. However, in the last few years, several works in the optimization literature were able to prove a faster convergence rate of SGD without-replacement under different scenarios~\citep{mishchenko2020random,ahn2020sgd, safran2020good, gurbuzbalaban2021random,nguyen2021unified, cai2023empirical}.

Despite the extensive use of without-replacement sampling, perhaps surprisingly, not many works focus on providing convergence guarantees for without-replacement sampling algorithms for solving min-max optimization problems and VIPs. 
\citet{das2022sampling} provide convergence guarantees for SGDA and the proximal point method (PPM) with without-replacement sampling for solving smooth and strongly convex-strongly concave problems satisfying a two-sided Polyak-Łojasiewicz inequality and show faster convergence, while  \citet{cho2023sgda} provide theoretical guarantees for SGDA with shuffling for solving structured non-monotone minimax problems. However, it is well known that SGDA fails to converge in simple monotone min-max problems (e.g., bilinear), while the PPM serves only as an implicit method.

The Stochastic Extragradient (SEG) method is one of the most popular algorithms for tackling finite-sum VIPs.
The algorithm consists of two steps: a) an \emph{extrapolation step} that computes a gradient update at the current iterate, and b) an \emph{update step} that updates the current iterate using the value of the vector field at the extrapolation point. SEG comes in different forms~\citep{gorbunov2022stochastic}. One of the most common choices is same-sample SEG (S-SEG) given in the following update rule:
\begin{equation}
	z^{k+1} = z^k - \gamma_{1} F_i\left(z^k - \gamma_{2} F_i(z^k)\right), \tag{S-SEG} \label{eq:S_SEG}
\end{equation}
where in each iteration, the same component $i \in [n]$ is sampled uniformly at random and used for the extrapolation (computation of $z^k - \gamma_{2} F_i(z^k)$) and update (computation of $z^{k+1}$) steps.  
\begin{figure*}[t]
\label{fig pseudocodes}
\centering
\begin{minipage}{0.32\textwidth}
\begin{algorithm}[H]
\setstretch{0.89}
\caption{SEG-RR}
\label{algo SEG}
\begin{algorithmic}[1]
\small
\STATE {\bf Given:} $z_0$, \text{step sizes }$\gamma_1, \gamma_2$,\\
        number of epochs $K$
\STATE {\bf Initialize:} $z_0^{0} = z_0$
\FORALL{$k = 0, ..., K-1$}
    \STATE \textcolor{orange}{Sample uniformly at random}\\
    \textcolor{orange}{a permutation $\pi^k$ of $[n]$}
    \FORALL{$i = 0, ..., n-1$}
        \STATE $\Bar{z}_{i}^k = z_{i}^k - \gamma_2 F_{\pi^k_i}(z^k_i)$
        \STATE $z_{i+1}^k = z_{i}^k - \gamma_1 F_{\pi^k_i}(\Bar{z}^k_i)$
    \ENDFOR
    \STATE $z_0^{k+1} = z^k_{n}$
\ENDFOR
\end{algorithmic}
\end{algorithm}
\end{minipage}
\hfill
\begin{minipage}{0.32\textwidth}
\begin{algorithm}[H]
\caption{SEG-SO}
\small
\label{algo EGSO}
\begin{algorithmic}[1]
\STATE {\bf Given:} $z_0$, \text{step sizes} $\gamma_1, \gamma_2$,\\
        number of epochs $K$
\STATE {\bf Initialize:} $z_0^{0} = z_0$
    \STATE \textcolor{blue}{Sample uniformly at random}\\
    \textcolor{blue}{a permutation $\pi$ of $[n]$}
\FORALL{$k = 0, ..., K-1$}
    \FORALL{$i = 0, ..., n-1$}
        \STATE $\Bar{z}_{i}^k = z_{i}^k - \gamma_2 F_{\pi_i}(z^k_i)$
        \STATE $z_{i+1}^k = z_{i}^k - \gamma_1 F_{\pi_i}(\Bar{z}^k_i)$
    \ENDFOR
    \STATE $z_0^{k+1} = z^k_{n}$
\ENDFOR
\end{algorithmic}
\end{algorithm}
\end{minipage}
\hfill
\begin{minipage}{0.32\textwidth}
\begin{algorithm}[H]
\caption{IEG}
\small
\label{algo EG IG}
\begin{algorithmic}[1]
\STATE {\bf Given:} $z_0$, \text{step sizes }$\gamma_1, \gamma_2$,\\
        number of epochs $K$
\STATE {\bf Initialize:} $z_0^{0} = z_0$
\STATE \textcolor{OrangeRed}{$\pi = [n]$ } \textcolor{OrangeRed}{// maintain the order of the dataset}
\FORALL{$k = 0, ..., K-1$}
    \FORALL{$i = 0, ..., n-1$}
        \STATE $\Bar{z}_{i}^k = z_{i}^k - \gamma_2 F_{\pi_i}(z^k_i)$
        \STATE $z_{i+1}^k = z_{i}^k - \gamma_1 F_{\pi_i}(\Bar{z}^k_i)$
    \ENDFOR
    \STATE $z_0^{k+1} = z^k_{n}$
\ENDFOR
\end{algorithmic}
\end{algorithm}
\end{minipage}
\caption{Three variants of without-replacement sampling SEG (SEG-RR, SEG-SO, IEG)}
\end{figure*}
Existing works focusing on the convergence guarantees of SEG studied only with-replacement sampling strategies, similar to \ref{eq:S_SEG}. However, most practical implementations of SEG use without-replacement sampling. 

A popular in practice but theoretically elusive update rule, belonging to the class of without-replacement sampling SEG, is \textit{SEG with Random Reshuffling} given in \ref{eq:SEG-RR} (see also Algorithm~\ref{algo SEG}). This is the method we pay most attention to in this work, as reflected in the title. In each epoch $k$, \ref{eq:SEG-RR} samples uniformly at random a permutation $\pi^k = \{\pi_0^k,\pi_1^k, \dots, \pi_{n-1}^k\}$ of $[n] \coloneqq \{1, 2, \dots, n\}$, and proceeds with $n$ iterates of the form:
\begin{equation}
   z^{k}_{i+1} = z^k_i - \gamma_1 F_{\pi^k_i} \left(z^k_i - \gamma_2 F_{\pi^k_i}(z^k_i) \right), \tag{SEG-RR} 
   \label{eq:SEG-RR}
\end{equation} 
where $\gamma_1>0$ and $\gamma_2>0$ are the step sizes in the update and extrapolation steps of the method, respectively. We then set $z_0^{k+1} = z^k_{n}$ and repeat the process for a total of $K$ epochs. In \ref{eq:SEG-RR} (Alg.~\ref{algo SEG}), a new permutation/shuffling is generated at the beginning of each epoch, which justifies the algorithm's name.

As a proof of concept, in Figure~\ref{Fig: 2d-plot}, we compare the above two variants of SEG: \ref{eq:S_SEG} (with-replacement) and \ref{eq:SEG-RR} (without-replacement) on solving a simple two-dimension bilinear problem of the form~\eqref{eq: MinMax}, where we choose $x$ and $y$ to be scalars.
As we can see in Fig.~\ref{Fig: 2d-plot}, \ref{eq:SEG-RR} converges to a smaller neighborhood of the min-max solution. Interestingly and perhaps surprisingly, in the left plot of Fig.~\ref{Fig: 2d-plot}, where we look at the trajectory of the two methods, the variant with random reshuffling \eqref{eq:SEG-RR} reduces the rotation around the solution, which might explain its preference in practical implementations over the uniform sampling variant \eqref{eq:S_SEG}. This motivates us to study further the convergence guarantees of \ref{eq:SEG-RR} in different classes of problems. 
Our work aims to bridge the gap between the theoretical analysis and practical implementation of SEG by studying the following question:
\begin{quote}
\centering
\textit{Can Random Reshuffling lead to improved theoretical and practical convergence for SEG in finite-sum VIPs?}
\end{quote}

\subsection{Preliminaries}
In this work, we assume that the operators $F_i$ of problem \eqref{VI-problem} are $L_i$-Lipschitz. This implies that the operator $F$ is also Lipschitz, and we will indicate with $L$ its value. Throughout this work, we focus on three classes of operators $F$: (i) strongly monotone, (ii) affine, and (iii) monotone.
Let us provide below the main definitions. 

\begin{definition}[$L-$Lipschitz]
An operator $F: \R^d \rightarrow \R^d$ is $L-$Lipschitz if there is $L>0$:
\begin{eqnarray}
\|F(z_1)-F(z_2)\| \leq L \|z_1-z_2\|, \quad \forall z_1, z_2 \in \R^d \label{DefLipschitz}
\end{eqnarray}
\end{definition}
We denote with $L_{max} = \max_{i \in [n]} L_i$, the maximum Lipschitz constant of the $F_i$ operators in problem~\eqref{VI-problem}

\begin{definition}[Strongly monotone / monotone operator]
\label{DefSM}
We say that an operator $F$ is $\mu-$strongly monotone if there exist $\mu>0$ such that $\forall z_1,z_2 \in \R^d$, 
$\left\langle F(z_1)-F(z_2),  z_1-z_2\right\rangle \geq \mu \|z_2-z_2\|^2$. If $\mu = 0$, then $\forall z_1,z_2 \in \R^d, \left\langle F(z_1)-F(z_2),  z_1-z_2\right\rangle \geq 0$, and we say that the operator is monotone.
\end{definition}

Lastly, we also focus on the class of affine operators, a subclass of monotone VIPs, that can be seen as a generalization of bilinear min-max problems. 

\begin{definition}[Affine]
    An operator $F: \mathbb{R}^d \xrightarrow[]{} \mathbb{R}^d$ is affine if it there exist $Q \in \mathbb{R}^{d \times d}, b \in \mathbb{R}^d$ such that $F(z) = Q z + b$.
\end{definition}

We denote with $\lambda_{min}^{+}(Q)$ the minimum non-zero eigenvalue of an affine and monotone operator $F$ in problem~\eqref{VI-problem}.

\renewcommand{\arraystretch}{1.5} 

\begin{table*}[ht]
\centering 
\resizebox{\textwidth}{!}{
\begin{tabular}{|>{}c|c|c|c|c|}
\hline
\textbf{Algorithm} & \textbf{Citation} & \textbf{Strongly Monotone} & \textbf{Affine \& Monotone} & \textbf{Monotone} \\ \hline
SGDA & \cite{loizou2021stochastic} & $\widetilde{\mathcal{O}}\left(\frac{1}{nK}\right)$ & \xmark & \xmark \\ \hline
SGDA-RR & \cite{das2022sampling} & $\widetilde{\mathcal{O}}\left(\frac{1}{nK^2}\right)$ & \xmark & \xmark \\ \hline
S-SEG & \cite{gorbunov2022stochastic, hsieh2020explore} & $\widetilde{\mathcal{O}}\left(\frac{1}{nK}\right)$ & $\widetilde{\mathcal{O}}\left(\frac{1}{nK}\right)$ & \xmark\textcolor{blue}{$^{(*)}$} \\ \hline
\rowcolor{green!40}
SEG-RR & \textbf{This paper} & $\widetilde{\mathcal{O}}\left(\frac{1}{nK^2}\right)$ & $\widetilde{\mathcal{O}}\left(\frac{1}{nK^2}\right)$ & $\mathcal{O}\left(\frac{1}{\sqrt{nK}}\right)$ \\ \hline
\end{tabular}
}
\caption{Iteration Complexity of SEG with uniform with-replacement sampling (\ref{eq:S_SEG}) and \ref{eq:SEG-RR} after a certain number of epochs $K$ that depends on the problem parameters. The $\Tilde{\mathcal{O}}(\cdot)$ notation suppresses constant and logarithmic factors.  $\textcolor{blue}{(*)}$: In the monotone case, \ref{eq:S_SEG} with constant step sizes requires large batch sizes to achieve a given target accuracy $\epsilon > 0$.}
\label{table:complexity}
\end{table*}

\paragraph{On bounded variance.} 
In the convergence analysis of \ref{eq:SEG-RR}, we do not assume bounded variance of the stochastic oracles $F_i$, i.e. there exists $c>0$ such that $\Exp[\|\nabla F_i(z)-\nabla F(z)\|^2] \leq c, \forall z \in \mathbb{R}^d$, or growth conditions, i.e. there exist $c_1, c_2 > 0: \Exp [\| \nabla F_i (z)\|^2] \leq c_1 \|\nabla F(z)\|^2 + c_2$, $\forall z \in \mathbb{R}^d$. These conditions are typically assumed in the theoretical analysis of stochastic methods for solving finite-sum VIPs of the form~\eqref{VI-problem}, as they simplify the proofs~\citep{mishchenko2020revisiting, lin2020finite, lin2020gradient, juditsky2011solving}. However, these assumptions are true only for a restrictive set of problems, and for large common classes of problems (e.g., unconstrained strongly monotone VIPs), they might not be even satisfied~\citep{loizou2021stochastic}. 
Instead, we follow a recent line of work that uses the Lipschitz assumption to provide closed-form expressions for the upper bound on the variance \citep{loizou2021stochastic,gorbunov2022stochastic,choudhury2023single}. 
More specifically, in Appendix~\ref{app preparatory_lemmas} we prove that if each $F_i$ is $L_i$-Lipschitz, then the following bound on the variance holds:
$\frac{1}{n}\sum_{i=0}^{n-1} \norm{ F_{i}(z) - F(z)}^2 \leq A \|z - z^*\|^2 + 2 \sigma_*^2$ 
where $A=\frac{2}{n} \sum\limits_{i=0}^{n-1} L_i^2 $ and $\sigma_*^2 =\frac{1}{n} \sum\limits_{i=0}^{n-1}\left\|F_i\left(z^*\right)\right\|^2$.
This new upper bound allows one to avoid the necessity of introducing any extra assumptions on the variance of the stochastic operators in the proofs. 

\subsection{Main Contributions}
Our main contributions are summarized below. See also Table~\ref{table:complexity} for a comparison of iteration complexities of our results with closely related works.
\begin{enumerate}
    \item \textbf{Strongly monotone or affine VIPs.}
 We prove the first convergence guarantees for \ref{eq:SEG-RR} for solving strongly monotone and affine VIPs. We show a linear convergence to a neighborhood of $z^*$ when constant step sizes $\gamma_1$ and $\gamma_2$ are used, and we explain why a double stepsize selection is needed. In particular, in our theorems, we require the extrapolation stepsize to be larger than the update stepsize ($\gamma_2 > \gamma_1$), which aligns with recent results on the convergence of \ref{eq:S_SEG} \citep{gorbunov2022stochastic, hsieh2020explore}. In both strongly monotone and affine regimes, we prove improved convergence of random reshuffling over uniform with-replacement sampling by showing that after a certain number of epochs $K$, \ref{eq:SEG-RR} achieves an iteration complexity of $\Tilde{\mathcal{O}}\left(\frac{1}{nK^2}\right)$ outperforming the $\Tilde{\mathcal{O}}\left(\frac{1}{nK}\right)$ iteration complexity of \ref{eq:S_SEG}. In the strongly monotone regime, this coincides with the benefit that SGDA-RR has over SGDA with uniform sampling proved in prior works. However, SGDA and SGDA-RR fail to converge to simple problems captured under the affine setting (e.g., bilinear minimax problems). 

\item \textbf{Monotone VIPs: Convergence without large batch sizes.} In the monotone case, we prove a sublinear convergence of the weighted average norm $\sum_{k=0}^{K-1} w_k\Expe{\| F(z_0^k)\|^2}$ to a neighborhood around the solution. In particular, we prove that \ref{eq:SEG-RR} can reduce the neighborhood and reach any target accuracy $\epsilon$ by choosing appropriately the step sizes $\gamma_1$ and $\gamma_2$ of the method, establishing in this way an iteration complexity of $\mathcal{O}\left(\frac{1}{\sqrt{nK}}\right)$ in the monotone setting.
This comes in stark contrast with the well-known results on the convergence of \ref{eq:S_SEG} for monotone problems, which require the use of large batch sizes, when constant step sizes are used in order to be able to reduce the neighborhood of convergence and reach any specific target accuracy $\epsilon > 0$. 

\item \textbf{Further Convergence Guarantees: Other without-replacement samplings and novel stepsize selection.} As a byproduct of our analysis, we also provide convergence guarantees for two other popular without-replacements sampling variants of SEG, the Shuffle Once SEG (SEG-SO) (see Alg.~\ref{algo EGSO}), which shuffles the data only at the beginning of the algorithm, and the Incremental Extragradient (IEG) (see Alg.~\ref{algo EG IG}), which does not shuffle the data and processes them in the order that they appear in the dataset. 
For solving strongly monotone and affine VIPs, we also provide convergence guarantees under different stepsize rules. In particular, using a carefully constructed switching stepsize-rule, we prove a $O(1/k)$ rate to the exact solution. The suggested switching stepsize rule describes when one should switch from a constant to a
decreasing stepsize regime, and it is the first time used in the analysis of algorithms utilizing without-replacement samplings. The details for these results are included in Appendix~\ref{Appendix_Further}. 

\item \textbf{Numerical Experiments.} We show the benefits of \ref{eq:SEG-RR} by performing numerical experiments on finite-sum strongly-monotone quadratic and bilinear minimax problems, as well as on Wasserstein GANs
for learning the mean of a multivariate Gaussian distribution. Our
numerical findings corroborate our theoretical results.
\end{enumerate}

\section{Convergence Analysis}

Let us now present our main theoretical results. We start by presenting a sketch of the proof techniques used in our theorems and explaining the difference/main challenge compared to the classical analysis of SEG. We, then, focus on the convergence guarantees for \ref{eq:SEG-RR} in three different classes of VIPs: strongly monotone, affine, and monotone. 

\subsection{Overview of Proof Techniques}
The main challenge in the proof of \ref{eq:SEG-RR} compared to the one of \ref{eq:S_SEG} is that the stochastic oracles $F_{\pi^k_i}$ are no longer unbiased estimators of the deterministic operator $F$.
Our proof is based on the key insight from previous works on random reshuffling in
minimization problems \citep{haochen2019random, ahn2020sgd, gurbuzbalaban2021random} that, for small enough step sizes, the epoch iterates $z_0^k$ of a stochastic algorithm using without-replacement sampling approximately follow the trajectory of the same full-batch algorithm.

Building on the aforementioned idea, we manage to upper bound the distance to $z^*$, $\| z^{k+1}_0 - z^*\|^2$, by three terms:
\begin{eqnarray}\label{ndaojsnda}
  \|z_0^{k+1} - z^*\|^2 \leq C_1 \underbrace{\left\|z_{0}^k - z^* - \gamma_1 n F(\hat{z}^k_0) \right\|^2}_{T_1} \nonumber \\ 
  + \text{ } C_2 \underbrace{\sum_{i=0}^{n-1}\| F_{\pi^k_i}(z_i^k) - F(z_0^k)\|^2}_{T_2} \nonumber \\
  +\text{ } C_3 \underbrace{\sum_{i=0}^{n-1}\left\|z^k_{i} - z_0^k\right\|^2,}_{T_3} 
\end{eqnarray}
where $\hat{z}^k_0 = z^k_0 - \gamma_2 F(z^k_0)$ and $C_1, C_2, C_3$ are constants depending on the properties of the problem in hand.

The term $T_1$ in \eqref{ndaojsnda} can be interpreted as the distance of one step of the full-batch SEG algorithm, starting from the point $z^k_0$, to the optimum $z^*$. Thus, it serves as a measure of the progress that the algorithm that uses the full operator $F$ makes. The term $T_2$, on the other hand, accounts for the fact that \ref{eq:SEG-RR} has access only to a stochastic oracle $F_{\pi^k_i}$ (not the full-batch operator $F$) per iteration. Using, in addition, the intuition that for small enough step sizes, the iterates $z^k_i$ inside an epoch stay ``close'' to the initial point $z_0^k$, the second term $T_2$ measures the distance between the stochastic oracle $F_{\pi^k_i}$ from the (full-batch) operator $F$ (though at different points). Lastly, the term $T_3$ indicates how ``far'' the points inside an epoch are from $z^k_0$. In our proofs, bounding each one of the three terms $T_1, T_2, T_3$ enables us to bound the distance of the iterate $z_0^{k+1}$ from the optimum $z^*$ and thus derive convergence guarantees for \ref{eq:SEG-RR} for the three different classes of VIPs under study.

\subsection{Strongly Monotone VIPs}
We focus on \ref{eq:SEG-RR} with constant step sizes $\gamma_1$ and $\gamma_2$. We first prove linear convergence to a neighborhood of the solution $z^*$. If, in addition, the total number of epochs $K$ is available, we suggest a constant stepsize selection that depends on $K$, which allows us to prove that after a certain number of epochs $K$, \ref{eq:SEG-RR} achieves an iteration complexity of $\Tilde{\mathcal{O}}\left(\frac{1}{nK^2}\right)$ outperforming the $\Tilde{\mathcal{O}}\left(\frac{1}{nK}\right)$ iteration complexity of \ref{eq:S_SEG}.

\begin{theorem}
\label{thm SC-SC4}
Suppose that the operator $F$ is $\mu$-strongly monotone and each $F_i, \, \forall i\in[n]$ is $L_i-$Lipschitz.
\begin{enumerate}[leftmargin=*]
\setlength{\itemsep}{0pt}
    \item Then the iterates of \ref{eq:SEG-RR} with constant step sizes $\gamma_2 = 2 \gamma_1$, $\gamma_1 \leq \frac{\mu}{10L_{max}^2\sqrt{10n^2+2n+54}}$ satisfy:
\begin{eqnarray}
    \hspace{-0.4cm}\Expe{\|z_0^{k} - z^*\|^2 } \leq\left(1 - \frac{\gamma_1 n\mu}{4} \right)^k {\|z_0- z^*\|^2} \nonumber \\ 
    \hspace{+1.cm} +\frac{96L_{max}^2}{\mu^2}\left[(25+n) \gamma_1^2+\gamma_2^2\right] \sigma_*^2 \label{sc1} 
\end{eqnarray}
    \item Let $K$ be the total number of epochs \ref{eq:SEG-RR} is run. For step sizes satisfying $\gamma_2 = 2 \gamma_1$ and $\gamma_1 = \min\left\{\frac{\mu}{10L_{max}^2\sqrt{10n^2+2n+54}}, \frac{4\log (n^{1/2}K)}{\mu nK}\right\}$, the following holds:
\begin{eqnarray}
   \Expe{\|z_0^{K} - z^*\|^2 } &\leq& \Tilde{\mathcal{O}}\left( e^{-\frac{\mu^2K}{L_{max}^2}} + \frac{1}{{nK^2}}\right) \label{sc2}
\end{eqnarray}
\end{enumerate}
\end{theorem}

Theorem \ref{thm SC-SC4} indicates in inequality \eqref{sc1} that for constant step sizes the \ref{eq:SEG-RR} algorithm converges linearly to a neighborhood of the solution $z^*$. The neighborhood of convergence is proportional to the step sizes and the variance $\sigma_*^2$ at the optimum point. In particular, Theorem \ref{thm SC-SC4} indicates that the neighborhood around the solution $z^*$ diminishes in relation to the step sizes of the algorithm as $\mathcal{O}(\gamma_1^2 + \gamma_2^2)$. In comparison, \ref{eq:S_SEG} converges linearly to a neighborhood around the solution, with the neighborhood decreasing as $\mathcal{O}(\gamma_1 + \gamma_2)$ (Theorem 3.1 of \citet{gorbunov2022stochastic}). Thus, although both algorithms achieve a linear convergence rate to a neighborhood around the solution, the without-replacement sampling variant will converge for the same step sizes to a smaller neighborhood of $z^*$. 

In addition, with the total number of epochs $K$ available, inequality \eqref{sc2} of Theorem \ref{thm-bil-1} establishes the iteration complexity of \ref{eq:SEG-RR} for achieving an error $\Expe{\|z_0^{K} - z^*\|^2 }\leq \epsilon$. More specifically, after a certain number of epochs satisfying $K \geq \kappa^2 \log(nK^2)$, where $\kappa = \frac{L_{max}}{\mu}$ is the condition number, the second term dominates in the iteration complexity and thus $\Expe{\|z_0^{K} - z^*\|^2 } = \Tilde{\mathcal{O}}\left(\frac{1}{nK^2}\right)$. In contrast, the iteration complexity of \ref{eq:S_SEG} in \citet{gorbunov2022stochastic} is $\Tilde{\mathcal{O}}\left( e^{-\frac{\mu nK}{L_{max}}} + \frac{1}{{nK}}\right)$ and thus after the same number of epochs $K \geq \kappa^2 \log(nK^2)$ the distance from the solution is $\Expe{\|z_0^{K} - z^*\|^2 } = \Tilde{\mathcal{O}}\left(\frac{1}{nK}\right)$. In this case, \ref{eq:SEG-RR} will require less number of epochs (equivalently iterations) to achieve an accuracy $\epsilon$. The difference in the iteration complexity of \ref{eq:SEG-RR} and \ref{eq:S_SEG} showcases the benefit of random reshuffling over uniform with-replacement sampling. 

We note, also, that in the strongly monotone setting the iteration complexity of \ref{eq:S_SEG} with step sizes $\gamma_1 = \frac{1}{6L_{max}}, \gamma_2 = 4\gamma_1$, as shown in \citet{gorbunov2022stochastic}, depends on the condition number as $\mathcal{O}\left(e^{-\kappa}\right)$. Despite this is a better dependence than the one in Theorem \ref{thm SC-SC4}, we highlight that for the step sizes of Theorem \ref{thm SC-SC4} both \ref{eq:SEG-RR} and \ref{eq:S_SEG} have the same $\mathcal{O}(e^{-\kappa^2})$ dependence. Hence, for the step sizes of Theorem \ref{thm SC-SC4}, \ref{eq:SEG-RR} and \ref{eq:S_SEG} will converge with the same rates to the corresponding neighborhoods of solution. Proving convergence of random reshuffling with larger step sizes and better dependence on the condition number is still an open problem. 

We, lastly, compare our results on the convergence of \ref{eq:SEG-RR} with the SGDA-RR algorithm, which is the other frequently used method for solving strongly monotone problems. The iteration complexity of SGDA-RR for step sizes $\gamma_1 = \mathcal{O}\left(\frac{\log(n^{\frac{1}{2}}K^2)}{nK}\right)$, where $K$ is the total number of epochs the algorithm is run, is $\Tilde{\mathcal{O}}\left(e^{-\frac{\mu^2 K}{5L^2}} + \frac{1}{nK^2}\right)$, as established in \citet{das2022sampling}. In our Theorem \ref{thm SC-SC4}, we establish the same iteration complexity (up to constant factors) with SGDA-RR. However, SEG-RR is able to solve VIPs beyond the strongly monotone regime, which is not the case of SGDA-RR, and we examine that below.
\subsection{Affine VIP Operators}
We, now, consider the setting where the variational inequality operator is affine and has the following form:
\begin{eqnarray}
    F(z) = \frac{1}{n} \sum\limits_{i=0}^{n-1} Q_i z + b_i \label{F_w_in_bilinear_games}
\end{eqnarray}
where $F(z)$ has a finite-sum structure with each $F_i(z) = Q_i z + b_i$. 
This setting serves as a generalization of any bilinear min-max optimization problem. For more details on the connection of bilinear games with affine variational inequalities, we refer the interested reader to Appendix \ref{app VI-defs}.

Similarly to the strongly monotone regime, we focus on the convergence of \ref{eq:SEG-RR} with constant step sizes. We prove linear convergence to a neighborhood of the solution $z^*$. If, in addition, the total number of epochs $K$ is available, we show that after a certain number of epochs $K$, \ref{eq:SEG-RR} achieves an iteration complexity of $\Tilde{\mathcal{O}}\left(\frac{1}{nK^2}\right)$. In this setting, since there might be multiple solutions $z^*$, we use as measure of convergence the $\text{dist}(z_0^k, \mathcal{Z}_*)=\min_{z^* \in \mathcal{Z}_*} \| z^0_k - z^*\|^2$, which is the distance of the iterate $z_0^k$ from the solution set $\mathcal{Z}_*$.

\begin{theorem} \label{thm-bil-1}
Suppose that each $F_i, \, \forall i\in [n],$ is monotone, affine and $L_i-$Lipschitz. 
\begin{enumerate}[leftmargin=*]
\setlength{\itemsep}{0pt}
    \item Then the iterates of \ref{eq:SEG-RR} with step sizes\\ $\gamma_2 = 4 \gamma_1$,
        $\gamma_1 \leq \frac{\lambda_{min}^{+}(Q)}{2\sqrt{120} nL^2_{max}}$ satisfy:
        \begin{eqnarray}
        \hspace{-0.25cm} \Expe{\text{dist}(z_0^{k}, \mathcal{Z}_*)}\hspace{-0.25cm} &\leq&\hspace{-0.25cm}\left(1\hspace{-0.07cm} -\hspace{-0.07cm} \frac{\gamma_1n\lambda_{min}^{+}(Q)}{2}\right)^k\hspace{-0.23cm} \text{dist}(z_0, \mathcal{Z}_*)\nonumber\\
             &+& \hspace{-0.29cm}\frac{4L_{max}\left[4n(n+25)\gamma_1^2+ \gamma_2^2\right]\sigma_*^2}{{\lambda_{min}^{+}(Q)}^2n^2} \hspace{+0.49cm} \label{res: affine1}
        \end{eqnarray}
        \item Let $K$ be the total number of epochs the\\ \ref{eq:SEG-RR} is run. For step sizes satisfying $\gamma_2 = 4 \gamma_1$,\\
$\gamma_1 = \min\left\{\frac{\lambda_{min}^{+}(Q)}{2\sqrt{120} nL^2_{max}}, \frac{2\log (n^{1/2}K)}{\lambda_{min}^{+}(Q) nK}\right\}$, it holds:
        \begin{eqnarray}
           \hspace{-0.1cm} \Expe{\text{dist}(z_0^{K}, \mathcal{Z}_*)} \leq \Tilde{\mathcal{O}}\left( e^{\frac{K{\lambda_{min}^{+}}^2(Q)}{4\sqrt{120}L^2_{max}}} + \frac{1}{{nK^2}}\right)
           \label{res: affine2}
        \end{eqnarray}
\end{enumerate}
\end{theorem}

Theorem \ref{thm-bil-1} indicates in \eqref{res: affine1} that \ref{eq:SEG-RR} achieves a linear convergence to a neighborhood of the solution $z^*$, which is proportional to the step sizes and the variance $\sigma_*^2$ at the optimum point. We highlight that the neighborhood of convergence decreases as $\mathcal{O}(\gamma_1^2+\gamma_2^2)$. In contrast, \citet{hsieh2020explore} establish for \ref{eq:S_SEG} a linear rate to a neihgbourhood that decreases as $\mathcal{O}(\gamma_1 + \gamma_2)$. For the step sizes suggested in Theorem \ref{thm-bil-1}, both \ref{eq:S_SEG} and \ref{eq:SEG-RR} converge with a linear rate, however, \ref{eq:SEG-RR} converges to a smaller neighborhood around the solution $z^*$.  

The second point in Theorem \ref{thm-bil-1}, given in \eqref{res: affine2}, establishes the iteration complexity of \ref{eq:SEG-RR} for achieving an error $\Expe{\|z_0^{K} - z^*\|^2 }\leq \epsilon$, assuming knowledge of the total number of epochs $K$. More specifically, after a certain number of epochs $K$ satisfying $K \geq \frac{4\sqrt{120}\lambda_{min}^2(Q)}{L_{max}^2} \log(nK^2)$, the second term dominates in the iteration complexity of \ref{eq:SEG-RR} (inequality \eqref{res: affine2}) and thus $\Expe{\|z_0^{K} - z^*\|^2 } = \Tilde{\mathcal{O}}\left(\frac{1}{nK^2}\right)$. In contrast, after the same number of epochs, the iteration complexity of \ref{eq:S_SEG} with constant step sizes $\gamma_1$ and $\gamma_2$ of \citet{hsieh2020explore} is equal to $\Tilde{\mathcal{O}}\left(\frac{1}{{nK}}\right)$.

\subsection{Monotone Operators}
In this part, we focus on the setting where the operator $F(z)$ in the VIP \eqref{VI-problem} is monotone. In this case,
 we prove a sublinear convergence of the weighted average norm $\sum_{k=0}^{K-1} w_k\Expe{\| F(z_0^k)\|^2}$ to a neighborhood around the solution $z^*$. In addition, for step sizes depending on the total number of epochs $K$, we prove that \ref{eq:SEG-RR} can reduce the neighborhood and reach any target accuracy $\epsilon > 0$, establishing in this way an iteration complexity of $\mathcal{O}\left(\frac{1}{\sqrt{nK}}\right)$ after a certain number of epochs.
 As a comparison, \ref{eq:S_SEG} can guarantee convergence to the same specific target accuracy $\epsilon > 0$, only if it is run with large batch sizes.

\begin{theorem}\label{theorem monotone my}
   Suppose that the operator $F$ is monotone and each $F_i, \, \forall i\in[n],$ is $L_i-$Lipschitz. Let $K$ be the total number of epochs \ref{eq:SEG-RR} is run and $z^* \in \mathcal{Z}_*$.
   \begin{enumerate}
       \item The iterates of \ref{eq:SEG-RR} with step sizes\\ $\gamma_1 \leq \gamma_{1, max}, \gamma_2 \leq \frac{1}{L_{max}}$ satisfy:
        \begin{eqnarray}
            \hspace{-0.3cm}\sum_{k=0}^{K-1} w_k\Expe{\| F(z_0^k)\|^2} \hspace{-0.25cm}&&\hspace{-0.25cm}\leq \frac{24 \|z_0 -z^*\|^2}{\gamma_2^2n^2K} \nonumber \\
            &&\hspace{-3.0cm} + \frac{4D_1\gamma_1 \left[2\gamma_1^2 (24n^2-23n+1) + \gamma_2^2\right] \sigma_*^2}{\gamma_2^2} \hspace{+0.4cm} \label{1st-pt-monotone_my} 
        \end{eqnarray}
        \item For step sizes $\gamma_2 = \frac{1}{L_{max} n^{\frac{3}{4}} K^{\frac{1}{4}}}$ and $ \gamma_1 = \gamma_{1, max}$, the iterates of \ref{eq:SEG-RR} satisfy 
        \begin{eqnarray}
           \hspace{-0.3cm} \sum_{k=0}^{K-1} w_k\Expe{\| F(z_0^k)\|^2} \hspace{-0.27cm}&\leq& \hspace{-0.27cm} \frac{24L_{max}^2 \|z_0 -z^*\|^2}{\sqrt{nK}} \nonumber \\
        &+& \hspace{-0.35cm}\frac{D_1L_{max}^2(24n^2-23n+1)\sigma_*^2}{D_2^3 n^{\frac{3}{2}} K^{\frac{5}{2}}}\nonumber \\
        &+& \hspace{-0.25cm} \frac{2D_1\sigma_*^2}{D_2nK} \hspace{-0.4cm} \label{2nd-pt-monotone_my}
        \end{eqnarray}
   \end{enumerate}
   where $w_{k} = \frac{G_{k+1}}{\sum_{j=0}^{K-1}G_{j+1}}, G_j = \left(\frac{1}{\alpha}\right)^j$ and\\
    $\gamma_{1, max} = \min\left\{\frac{\gamma_2^2n}{4\left(\frac{n}{L_{max}} + 1\right)^2}, \frac{1}{3\sqrt{2}nL_{max}}, \frac{1}{2nK D_2}\right\}$,
    \\
    $\alpha = 1 + 2 \gamma_1 D_2, C = \frac{3(1 + L^2+L_{max})}{L^2_{max}}, D_1 = 24CL_{max}^2,$\\
    $D_2 = nL^2+L_{max}+10n^2C(11L^2 + 15A)$.
\end{theorem}
\vspace{-0.3cm}
Theorem \ref{theorem monotone my} indicates a sublinear convergence for \ref{eq:SEG-RR}. The convergence is in an average sense, i.e. the weighted average operator
norm converges to a neighborhood around a solution $z^*$, which is proportional to the step sizes and the variance $\sigma_*^2$ at the optimum. 
Thus, for smaller step sizes $\gamma_1$ and $\gamma_2$ we expect that the algorithm will converge to a smaller neighborhood around $z^*$.

In contrast, in the convergence analysis of \ref{eq:S_SEG}, the neighborhood around the solution cannot be reduced by selecting only the step sizes.  More specifically, \citet{gorbunov2022stochastic} prove the following upper bound (Corollary E.4 for $\gamma_2 = 4 \gamma_1$):
\begin{eqnarray}
   \Expe{\sum_{k=0}^{K-1}\hspace{-0.05cm} w_k\| F(z_0^k)\|^2} &\leq& \frac{\|z_0- z^*\|^2}{2\gamma_1\gamma_2 nK} + 6 \sigma_*^2 \label{Gr-bound}
\end{eqnarray}
where $w_k$
is a different weighted average of the iterates with weights that depend on the step sizes and the parameters of the problem. The second term on the right-hand side of \eqref{Gr-bound} apparently does not depend on the step sizes of the algorithm. Thus, one cannot reduce the neighborhood of convergence around the solution $z^*$ arbitrarily, even by selecting step sizes that depend on the total number of epochs the algorithm is run.
\vspace{-0.1cm}

A minibatch of size $\mathcal{O}(nK)$ is required according to \cite{gorbunov2022stochastic} in order for \ref{eq:S_SEG} to reduce the variance around the optimum and achieve an arbitrary accuracy $\epsilon > 0$. In contrast, \ref{eq:SEG-RR} can achieve an arbitrary accuracy without the necessity of large batch sizes by selecting step sizes that depend on the total number of epochs, as shown in \eqref{2nd-pt-monotone_my}. In particular, 
after a certain number of epochs $K \geq \mathcal{O}\left(\sqrt{n}\right)$ (see equation \eqref{com-result-monot} in Appendix \ref{sec: iter_comp_seg_rr}), the first term dominates in the right-hand side of inequality \eqref{2nd-pt-monotone_my} and thus \ref{eq:SEG-RR} achieves an $\mathcal{O}\left(\frac{1}{\sqrt{nK}}\right)$ accuracy, arbitrarily close to a solution. This indicates an intrinsic difference in the batch sizes required in the two methods, \ref{eq:S_SEG} and \ref{eq:SEG-RR}, to converge arbitrarily close to the exact solution $z^*$.
\vspace{-0.1cm}

\section{Numerical Experiments}
In this section, we show the benefits of SEG-RR by performing numerical experiments\footnote{The code for reproducing our experimental results is available at \url{https://github.com/emmanouilidisk/Stochastic-ExtraGradient-with-RR}.} in strongly monotone quadratic and bilinear minimax problems, as well as on Wasserstein GANs for learning the mean of a multivariate Gaussian distribution. 

In particular, we compare SEG-RR, SEG-SO, and IEG with the uniform with-replacement sampling \ref{eq:S_SEG} (denoted as \textit{SEG} in the plots). For each experiment, we report the average of 5 runs and plot the relative error $\log(\frac{\| z^k_0 - z^*\|^2}{\|z_0 - z^*\|^2})$ over the iterations the algorithm is run. 

In the strongly monotone setting, similarly to \cite{loizou2021stochastic,gorbunov2022stochastic, choudhury2023single}, we consider a quadratic strongly convex strongly concave minimax problem that has the following form:
\begin{align*}
\min _{x \in \mathbb{R}^{d}} \max _{y \in \mathbb{R}^{d}} \frac{1}{n} \sum_{i=1}^n &\frac{x^\top A_i x}{2} +\hspace{-0.05cm} x^\top B_i y - \hspace{-0.05cm}\frac{y^\top C_i y}{2}\hspace{-0.05cm} + \hspace{-0.05cm}a_i^\top x \hspace{-0.05cm}-\hspace{-0.05cm} c_i^\top y
\end{align*}
while in the affine regime, we focus on the following two-player bilinear zero-sum game:
\begin{eqnarray}
    \min_{x \in \mathbb{R}^{d}} \max_{y \in \mathbb{R}^{d}} \frac{1}{n}\sum\limits_{i=1}^n x^\top B_i y + a_i^\top x - c_i^\top y \nonumber
\end{eqnarray}

\begin{figure}[t]
\centering 
    \includegraphics[width=\linewidth]{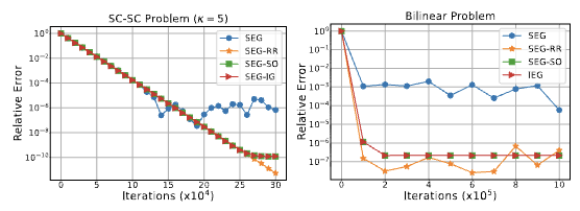}
  \caption{The left plot corresponds to a strongly monotone problem, while the right plot corresponds to a bilinear game. \ref{eq:SEG-RR} with the theoretical step sizes converges to a smaller relative error compared to the other variants of SEG.}
  \label{Fig:performance exp}
\end{figure}

We provide details regarding the way that the matrices $A_i, B_i, C_i$ and the vectors $a_i, c_i$ are sampled in the above problems along with a full description of our experimental setup in Appendix \ref{app experiments}.

\paragraph{Theoretical step sizes.}
In the first experiment, we focus on validating Theorems \ref{thm SC-SC4} and \ref{thm-bil-1} by running \ref{eq:SEG-RR} using the step sizes proposed in our analysis. In Figure~\ref{Fig:performance exp}, we observe that both in the strongly monotone and the bilinear case \ref{eq:SEG-RR} with constant step sizes converges linearly to a neighborhood around the minimax solution $z^*$, verifying our theoretical results. 

In addition, Figure \ref{Fig:performance exp} shows that the three without replacement strategies \ref{eq:SEG-RR}, SEG-SO, and IEG outperform the uniform with-replacement sampling counterpart of SEG for the same number of epochs/iterations. In our experiments, we also observe that \ref{eq:SEG-RR} reaches the same neighborhood of convergence (if not smaller) compared to SEG-SO and IEG. We have run experiments, also, for problems with different Lipschitz constants and have observed similar behavior of convergence for \ref{eq:SEG-RR}. The additional experiments for different Lipschitz parameters can be found in Appendix \ref{app add experiments}.

\paragraph{Beyond Theory: Larger step sizes.} In the second set of experiments, we investigate the behavior of \ref{eq:SEG-RR} with larger step sizes than the ones that our theory predicts. That is, we use larger step sizes proposed in previous analyses of \ref{eq:S_SEG} and compare \ref{eq:SEG-RR} and \ref{eq:S_SEG} using these step sizes selection. In particular, for strongly monotone problems, we run experiments for the step sizes proposed in the analysis of \ref{eq:S_SEG} from \citet{gorbunov2022stochastic} where $\gamma_1 = \frac{1}{6L_{max}}$ and $\gamma_2 = 4\gamma_1$ while for bilinear games, we use the step sizes $\gamma_1 = \frac{0.1}{(t+19)^{r_\eta}}, \gamma_2 = \frac{1}{(t+19)^{r_{\gamma}}}$ where $r_{\gamma} = 0, r_{\eta} =0.7$ suggested in the analysis of SEG for bilinear games in \citet{hsieh2020explore}.

\begin{figure}[t]
\centering 
\includegraphics[width=\linewidth]{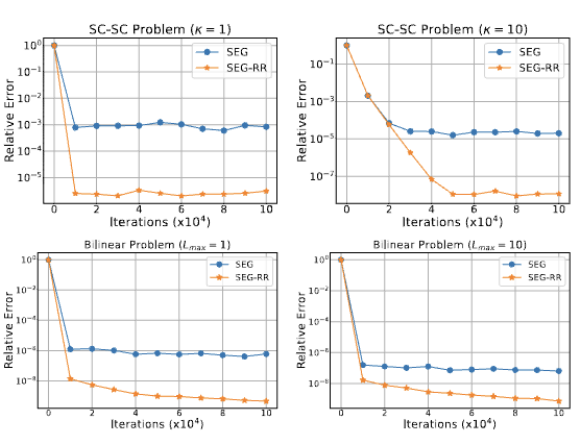}  
\caption{First-row: SC-SC problem. Second-row: Bilinear Game. \ref{eq:SEG-RR} outperforms SEG in problems with different condition numbers (step size used in SC-SC problem as in 
\citet{gorbunov2022stochastic}, while step size used in Bilinear Game as in \citet{hsieh2020explore}).
}
\label{Figure: multiple cond and steps}
\end{figure}

In Figure \ref{Figure: multiple cond and steps}, we observe that \ref{eq:SEG-RR} achieves convergence to a smaller neighborhood than \ref{eq:S_SEG} for both strongly monotone and bilinear problems. We have, also, conducted additional experiments for more step size and problems with different Lipschitz parameters. We refer the interested reader to Appendix \ref{app add experiments} for a dedicated section. 

In all of the experiments, \ref{eq:SEG-RR} achieves at least as good (if not better) performance than \ref{eq:S_SEG}, advocating for the use of random reshuffling in practical scenarios, even with step size larger than the ones in our theoretical convergence guarantees.

\paragraph{Wasserstein GANs.}
In our last experiment, we train a Wasserstein GAN (WGAN)~\citep{pmlr-v70-arjovsky17a} for learning the mean of a Multivariate Gaussian distribution. In this scenario, the optimization objective of the WGAN has the following form:
\begin{equation}
    \inf_{\theta} \sup_{w} \mathbb{E}_{x\sim N(\mu, \Sigma)}\left[ \langle w, x \rangle \right] - \mathbb{E}_{z\sim N(0,\Sigma)}\left[\langle w, z + \theta \rangle\right]\nonumber 
\end{equation}

In this setting, the discriminator is a linear function $D(x; w) = \langle w, x \rangle$ of the parameter $w \in \mathbb{R}^d$, where the input data point is denoted by $x$. On the other hand, the generator takes as input a random noise vector $z \sim N(0, \frac{1}{10} I)$ in $\mathbb{R}^d$ and outputs the vector $G(z; \theta) = z + \theta$, which is a linear function of the parameter $\theta\in \mathbb{R}^d$. The goal of the generator is to find the mean $\mu$ of the underlying true distribution $\mathcal{N}(\mu, \Sigma)$, where $\mu = [3, 4]^T$ and $\Sigma = \frac{1}{10} I$.

For the comparison of \ref{eq:S_SEG} and \ref{eq:SEG-RR}, we train the WGAN with each one of the two methods. We use the same constant step size for both algorithms with $\gamma_{2}= 4 \gamma_{1}, \gamma_1 =0.01$ being the extrapolation and update step size respectively in both the generator and the discriminator. Figure \ref{Fig: wgan} shows clearly that the generator trained with \ref{eq:SEG-RR} is able to converge closer to the optimal weights than the generator trained with \ref{eq:S_SEG}.

\begin{figure}[t]
\centering 
  \includegraphics[width=80mm]{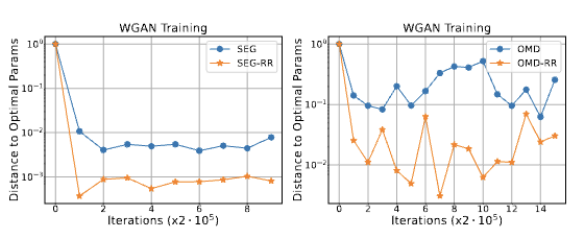}
  \caption{Left: WGAN trained with \ref{eq:SEG-RR} or \ref{eq:S_SEG} (denoted as SEG). Right plot: WGAN trained with OMD-RR or OMD. Random reshuffling helps the generator converge closer to the mean $\mu = [3, 4]^T$ of the Gaussian than with-replacement sampling for either the SEG or OMD algorithm.}
  \label{Fig: wgan}
\end{figure}

Lastly, as in \citet{daskalakis2018training}, we train a WGAN with the use of Optimistic Mirror Descent (OMD). Aiming to see the effect of random reshuffling even for this algorithm, we train the WGAN using (i) uniform with-replacement sampling OMD and (ii) OMD with random reshuffling (OMD-RR). We let the step size of the generator and the discriminator be $\gamma_G=0.02, \gamma_D=0.01$ respectively. In Figure \ref{Fig: wgan}, we observe that random reshuffling allows the OMD algorithm to achieve a smaller distance from the generator's optimal parameters, indicating the benefits of using random reshuffling on top of more popular algorithms.

\section{Conclusion}
We analyze \ref{eq:SEG-RR} for strongly monotone, affine, and monotone VIPs. We show that SEG equipped with without-replacement samplings can outperform the iteration complexity of \ref{eq:S_SEG} after a certain number of epochs. Additionally, in the monotone case, we prove that without-replacement samplings allow the algorithm to converge to an arbitrary accuracy $\epsilon > 0$ without the necessity of having large batch sizes. We aspire that our proof techniques will be a starting point for further results in the field of without-replacement samplings for solving VIPs. In this scope, extending the convergence analysis of \ref{eq:SEG-RR} to structured non-monotone settings, establishing convergence guarantees for the Stochastic Past ExtraGradient (SPEG) \citep{choudhury2023single} with random reshuffling, and developing the random reshuffling literature for distributed VIPs \citep{zhang2023communication,beznosikov2022distributed} are exciting open research questions that remain to be addressed in the future.

\section*{Acknowledgements}
Konstantinos Emmanouilidis acknowledges support from a MINDS Fellowship funded by grant NSF 1934979 ``HDR TRIPODS: Institute for the Foundations of Graph and Deep Learning.'' Ren\'e Vidal acknowledges the support of the NSF-Simons grant 2031985 ``Research Collaborations on the Mathematical and Scientific Foundations of Deep Learning.'' Nicolas Loizou acknowledges support from CISCO Research.

\bibliography{main}

\section*{Checklist}
\begin{enumerate}
 \item For all models and algorithms presented, check if you include:
 \begin{enumerate}
   \item A clear description of the mathematical setting, assumptions, algorithm, and/or model. [Yes]
   \item An analysis of the properties and complexity (time, space, sample size) of any algorithm. [Yes]
   \item (Optional) Anonymized source code, with specification of all dependencies, including external libraries. [Not Applicable]
 \end{enumerate}
 \item For any theoretical claim, check if you include:
 \begin{enumerate}
   \item Statements of the full set of assumptions of all theoretical results. [Yes]
   \item Complete proofs of all theoretical results. [Yes, they are included in the Supplemental Material.]
   \item Clear explanations of any assumptions. [Yes]     
 \end{enumerate}

 \item For all figures and tables that present empirical results, check if you include:
 \begin{enumerate}
   \item The code, data, and instructions needed to reproduce the main experimental results (either in the supplemental material or as a URL). [Yes]
   \item All the training details (e.g., data splits, hyperparameters, how they were chosen). [Yes]
         \item A clear definition of the specific measure or statistics and error bars (e.g., with respect to the random seed after running experiments multiple times). [Yes]
         \item A description of the computing infrastructure used. (e.g., type of GPUs, internal cluster, or cloud provider). [Not Applicable]
 \end{enumerate}

 \item If you are using existing assets (e.g., code, data, models) or curating/releasing new assets, check if you include:
 \begin{enumerate}
   \item Citations of the creator If your work uses existing assets. [Yes]
   \item The license information of the assets, if applicable. [Not Applicable]
   \item New assets either in the supplemental material or as a URL, if applicable. [Not Applicable]
   \item Information about consent from data providers/curators. [Not Applicable]
   \item Discussion of sensible content if applicable, e.g., personally identifiable information or offensive content. [Not Applicable]
 \end{enumerate}

 \item If you used crowdsourcing or conducted research with human subjects, check if you include:
 \begin{enumerate}
   \item The full text of instructions given to participants and screenshots. [Not Applicable]
   \item Descriptions of potential participant risks, with links to Institutional Review Board (IRB) approvals if applicable. [Not Applicable]
   \item The estimated hourly wage paid to participants and the total amount spent on participant compensation. [Not Applicable]
 \end{enumerate}

 \end{enumerate}

\onecolumn
\appendix
\aistatstitle{\vspace{-0.3cm}Stochastic Extragradient with Random Reshuffling:\\ Improved Convergence for Variational Inequalities\\ Supplementary Material\vspace{-0.3cm}}
\vspace{-3cm}

The Supplementary Material is organized as follows: In Section \ref{app preparatory_lemmas}, we provide some preparatory lemmas and propositions. Section \ref{app theorems for SEG-RR} presents the proofs of our main theorems for \ref{eq:SEG-RR}. Section \ref{Appendix_Further} provides further convergence guarantees for \ref{eq:SEG-RR} with decreasing/switching step size. We, also, explain how the convergence of SEG-SO and IEG is obtained as a corollary of \ref{eq:SEG-RR} analysis. In Section \ref{app experiments}, we describe in detail our experimental setup and provide additional experiments. 
{\small\tableofcontents}
\newpage

\section{Preparatory Lemmas \& Propositions} \label{app preparatory_lemmas}
We start by providing the basic notation and some useful inequalities we use in our proofs, as well as essential preliminaries on variational inequalities. In subsection \ref{sec: random_reshuffling_prop}, a proposition about random reshuffling is provided that is critical in the analysis of stochastic algorithms equipped with without-replacement sampling. In subsection \ref{app: bound for const assumpt}, we state a proposition for bounding the variance of the stochastic oracles $F_i$, and in subsection \ref{sec:preparatory}, we conclude this section with lemmas that will be necessary in the proofs of our main theorems.

\subsection{Notation}
We start by introducing the notation that will be useful for stating formally our main results. Let $[n] = \{1, ..., n\}$ and $\mathbb{S}_n$ be the symmetric group of $[n]$. We denote with $\pi^k$ the permutation of the random reshuffling algorithm at epoch $k$ and with $\pi^k_i$ the $i$-th element of the permutation $\pi^k$, for $0 \leq i \leq n-1$. The $i$-th iterate of the algorithm at the $k$-th epoch will be indicated by $z^k_i$. The expectation over the uniform distribution of all permutations $\mathcal{D} = \mathcal{U}(\mathbb{S}_n)$ condition on the natural filtration $\mathcal{F}_k$ of $z_0^{k}$ is denoted by $\Expep{\cdot}$. The expectation taking into account all the stochasticity of the algorithm is denoted by $\Expe{\cdot}$. 

We, also, denote the extrapolation and update step of the \ref{eq:SEG-RR} algorithm with
\begin{eqnarray}
    \Bar{z}_{i-1}^k = z_{i-1}^k - \gamma_2 F_{\pi_{i-1}^k}(z_{i-1}^k) \label{SEG_extrapolation_step}\\
    z_i^{k} = z_{i-1}^k - \gamma_1 F_{\pi_{i-1}^k}(\Bar{z}_{i-1}^k) \label{SEG_update_rule}
\end{eqnarray}
as well as an additional variable useful in our proofs with
\begin{eqnarray}
    \hat{z}_i^k = z_i^k - \gamma_2 F(z_i^k) \label{hat_z}
\end{eqnarray}
\subsection{Useful Inequalities}
In this section, we provide inequalities that will be useful in our proofs
\begin{eqnarray}
		\left\|\sum_{i=1}^n x_i\right\|^2 &\leq& n\sum_{i=1}^n\left\|x_i\right\|^2 \label{ineq1} \\
		\left\|a - b\right\|^2 &\geq& \frac{1}{2} \left\|a \right\|^2 - \left\|b\right\|^2 \label{ineqa_b} \\
		\langle a, b \rangle& =& \frac{1}{2} \left[\left\|a \right\|^2 + \left\|b\right\|^2 - \| a-b\|^2\right] \label{ineq inner product} \\
		e^{-x} &\geq& 1 - x, \forall x \geq 0 \label{ineq exponential}
\end{eqnarray}
Using Jensen inequality for $f(x) = \|x\|^2$ yields $\forall t \in (0,1)$ the below inequality:
\begin{eqnarray} 
			\|a+b\|^2 = \left\|\frac{t}{t} a+ \frac{1-t}{1-t}b\right\|^2 &\leq& t\left\| \frac{a}{t}\right\|^2 +(1-t) \left\|\frac{b}{1-t} \right\|^2 = \frac{1}{t}\| a \|^2 + \frac{1}{1-t}\|b\|^2  \nonumber \\ 
    \iff \|a+b\|^2&\leq& \frac{1}{t}\| a \|^2 + \frac{1}{1-t}\|b\|^2  
        \label{ineq1a}
\end{eqnarray}
Substituting $t = 1 - \frac{1}{2}\gamma_1 n\mu \in (0,1)$ in inequality \eqref{ineq1a}, we have that
\begin{eqnarray} 
			\|a+b\|^2 \leq \frac{1}{1 - \frac{1}{2}\gamma_1 n \mu}\| a \|^2 + \frac{2}{\gamma_1 n \mu}\|b\|^2\label{Youngwitht=1-gamma}
\end{eqnarray}
Similarly, substituting $t = 1 -\gamma_1 n(\lambda_{min}^+(Q) -\gamma_2 L_{max}^2)\in (0,1)$ in inequality \eqref{ineq1a} we get
\begin{eqnarray} 
			\|a+b\|^2 \leq \frac{1}{1 -\gamma_1 n(\lambda_{min}^+(Q) -\gamma_2 L_{max}^2)}\| a \|^2 + \frac{1}{\gamma_1 n(\lambda_{min}^+(Q) -\gamma_2 L_{max}^2)}\|b\|^2\label{Youngforbil}
\end{eqnarray}
\subsection{Min-Max Optimization and Variational Inequalities}
\label{app VI-defs}

In the following, we establish the connection between min-max optimization problems and VIPs. We focus on bilinear min-max optimization problems and explain how they can be cast as a special case of affine VIPs. Similar connections can be established for strongly convex-strongly concave and convex-concave min-max optimization problems.  

Given a bilinear game of the following form
\begin{eqnarray}
  \min_{x \in \mathbb{R}^{d_x}} \max_{y \in \mathbb{R}^{d_y}} f(x,y) := \frac{1}{n} \sum\limits_{i=0}^{n-1} x^\top B_i y + a_i^\top x - c_i^\top y, \label{bilinear game}
\end{eqnarray}
letting $Q_i = \left[\begin{matrix}
        0 & B_i \\
        -B_i^T & 0 \\
\end{matrix}\right], b_i = \left[\begin{matrix}
        a_i \\
        c_i \\
\end{matrix}\right]$ one can observe that the problem corresponds to a variational inequality with an affine operator $F(z) = \frac{1}{n} \sum\limits_{i=0}^{n-1} Q_i z + b_i$. 
Similarly, when the minimax problem is (strongly) convex-(strongly) concave, then the associated variational inequality operator $F(z)$ is (strongly) monotone. Thus, minimax optimization problems are a special case of the problems encapsulated under the more general framework of variational inequalities.

\subsection{Proposition about Random Reshuffling}
\label{sec: random_reshuffling_prop}
We, next, state a proposition about random reshuffling that will turn out to be helpful in deriving the lemmas of Section \ref{sec:preparatory}.

\begin{proposition}[\citet{mishchenko2020random}] \label{prop:random_reshuffling}
Let $\{X_1,\dotsc, X_{n}\}\in \mathbb{R}^d$ be a population of $n$ random vectors, $\mu \triangleq \frac{1}{n}\sum_{i=1}^{n} X_i$ the population average and $\sigma^2 \triangleq \frac{1}{n}\sum_{i=1}^{n} \norm{X_i-\mu}^2$ the population variance.\\ Take a sample $\{X_{\pi_0}, ..., X_{\pi_{d-1}} \}$ of $d\in [n]$ random vectors from $\{X_1,\dotsc, X_{n}\}$ uniformly at random without replacement and let $\Bar{X}=\frac{1}{d}\sum_{i=0}^{d-1} X_{\pi(i)}$ be the sample average and $\operatorname{Var}(\Bar{X})$ the sample variance. 

Then, we have that:
	\begin{align}
		\mathbb{E}_{\pi \in \mathcal{S}}\left[\norm{\Bar{X} - \mu }^2\right]= \mathbb{E}_{\pi \in \mathcal{S}}\left[\norm{\frac{1}{d}\sum_{i=0}^{d-1} X_{\pi_i} - \frac{1}{n}\sum_{i=1}^{n} X_i}^2\right] = \frac{n-d}{d(n-1)}\sigma^2. \label{eq:sampling_wo_replacement}
	\end{align}
 where the expectation is taken with respect to the set $\mathcal{S}$, which is the set of permutations of length $d$ of $[n]$.
\end{proposition}

\begin{proof}
    We first establish the identity $\mathrm{cov}(X_{\pi_i}, X_{\pi_j})=-\frac{\sigma^2}{n-1}, \forall i \neq j$ as follows:
    \begin{eqnarray*}
        \mathrm{cov}(X_{\pi_i}, X_{\pi_j})
        &=& \frac{1}{n(n-1)} \sum_{l=1}^{n} \sum_{m=1,m\neq l}^{n} \Expe{X_l - \mu, X_m - \mu} \\
        &=& \frac{1}{n(n-1)}\sum_{l=1}^{n}\sum_{m=1}^{n}\Expe{X_l - \mu, X_m - \mu} - \frac{1}{n(n-1)}\sum_{l=1}^{n} \norm{X_l - \mu}^2 \\
        &=& \frac{1}{n(n-1)}\sum_{l=1}^{n} \Expe{X_l - \mu, \sum_{m=1}^{n}(X_m - \mu)} - \frac{\sigma^2}{n-1} =-\frac{\sigma^2}{n-1}.
    \end{eqnarray*}
   We, now, turn to the formula for sample variance:
    \begin{align*}
        \operatorname{Var}(\Bar{X})
        = \mathbb{E}_{\pi \in \mathcal{S}}\left[\norm{\Bar{X} - \mu }^2\right] &= \frac{1}{d^2} \sum_{i=0}^{d-1}\sum_{j=0}^{d-1} \mathrm{cov}(X_{\pi_i}, X_{\pi_j}) \\
    &= \frac{1}{d^2}\left[ \sum_{i=0}^{d-1}\sum_{j=i}^{d-1} \mathrm{cov}(X_{\pi_i}, X_{\pi_j}) + \sum_{i=0}^{d-1}\sum_{j=0,j\neq i}^{d-1} \mathrm{cov}(X_{\pi_i}, X_{\pi_j}) \right] \\
    &= \frac{1}{d^2}\left[ \sum_{i=0}^{d-1}\sum_{j=i}^{d-1} \mathrm{Var}(X_{\pi_i}) + \sum_{i=0}^{d-1}\sum_{j=0,j\neq i}^{d-1} \mathrm{cov}(X_{\pi_i}, X_{\pi_j}) \right] 
    \end{align*}
    We, thus, continue our arithmetic manipulations 
    \begin{eqnarray}
        \mathbb{E}_{\pi \in \mathcal{S}}\left[\norm{\Bar{X} - \mu }^2\right] &=\frac{1}{d^2}\left(d\cdot\sigma^2- d(d-1)\frac{\sigma^2}{n-1}\right)
    = \frac{n-d}{d(n-1)}\sigma^2
    \end{eqnarray}
    to conclude with the promised equation in the statement of this proposition.
\end{proof}

\subsection{Variance of Stochastic Oracles}

We provide a proposition for bounding the variance of the stochastic oracles $F_i$. As mentioned in the main paper, our approach follows a recent line of work \citep{loizou2021stochastic,gorbunov2022stochastic,choudhury2023single, loizou2020stochastic, gower2019sgd, gower2021sgd, khaled2023unified} that uses the Lipschitz assumption to provide closed-form expressions for the upper bound on the variance.

\begin{proposition} 
\label{app: bound for const assumpt}
    If each $F_i$ is $L_i-$Lipschitz, then $\forall z \in \mathbb{R}^d$ the following holds
    $$\hspace{-0.1cm} \frac{1}{n}\sum_{i=1}^{n} \norm{ F_{i}(z) - F(z)}^2 \leq A \|z - z^*\|^2 +2 \sigma_*^2$$
where $A=\frac{2}{n} \sum\limits_{i=1}^{n} L_i^2 $ and $\sigma_*^2 =\frac{1}{n} \sum\limits_{i=1}^{n}\left\|F_i\left(z^*\right)\right\|^2$.
\end{proposition}
\begin{proof}
    \begin{eqnarray}  
    \frac{1}{n} \sum_{i=1}^{n}\left\|F_i(z)-F(z)\right\|^2 &\stackrel{}{=}& \frac{1}{n} \sum_{i=1}^{n} \left\|F_i(z) \right\|^2 -\frac{2}{n} \sum_{i=1}^{n}\langle F_i(z), F(z)\rangle + \frac{1}{n} \sum_{i=1}^{n} \left\|F(z)\right\|^2 \nonumber \\
    &=& \frac{1}{n} \sum_{i=1}^{n} \left\|F_i(z) \right\|^2 -2 \|F(z)\|^2 + \left\|F(z)\right\|^2\nonumber \\
    &\leq& \frac{1}{n} \sum_{i=1}^{n} \left\|F_i(z) \right\|^2 \nonumber \\ 
    &\stackrel{\eqref{ineq1}}{\leq}& \frac{2}{n} \sum_{i=1}^{n} \left\|F_i(z) - F_i(z^*)\right\|^2 + \frac{2}{n} \sum_{i=1}^{n} \left\|F_i(z^*)\right\|^2\nonumber \\ 
    &=& A \|z - z^*\|^2 + 2 \sigma_*^2 \nonumber \end{eqnarray}
    where $A=\frac{2}{n} \sum_{i=1}^{n} L_i^2$ and $\sigma_*^2 = \frac{1}{n} \sum\limits_{i=1}^{n} \|F_i(z^*)\|^2$.
\end{proof}

\subsection{Useful Lemmas}\label{sec:preparatory}
In this section, we provide some necessary preparatory lemmas that will be crucial for proving our main results. 
We start with a lemma that bounds the distance between the operator $F$ and $\frac{1}{d} \sum_{j=0}^{d-1} F_{\pi_j}$ with $d \in [n]$. \vspace{-0.2cm}
\begin{lemma}
\label{prop: bound prr epoch deviations}
Suppose that each $F_i, i \in [n]$ is $L_i-$Lipschitz. For any fixed $d \in [n]$ and $z \in \R^{d}$, the following inequality holds: 
\begin{eqnarray}
d^2\mathbb{E}_{\pi \in \mathcal{S}} \left[\norm{\frac{1}{d} \sum_{j=0}^{d-1} F_{\pi_j}(z)- F(z)}^2 \right] &\le& \frac{d(n-d)}{n-1} \left(A \|z - z^*\|^2 +2 \sigma_*^2\right) \quad \quad \label{bound2}   
\end{eqnarray} 
\vspace{-.2cm}
 where $\mathcal{S}$ is the set of all permutations of the set $[n]$ of length $d$, 
 $A =\frac{2}{n} \sum_{i=1}^{n} L_i^2$ and $\sigma_*^2 = \frac{1}{n} \sum\limits_{i=1}^{n} \|F_i(z^*)\|^2$.
\end{lemma}
\begin{proof}
First, we substitute in Proposition \ref{prop:random_reshuffling} $X_i \leftarrow F_i(z_0^k), i\in[n]$ and fix an integer $d\in [n]$. Next, we draw a permutation with $d$ elements uniformly at random from the set of all permutations of $[n]$ with $d$ elements, i.e. $\pi \sim \mathcal{U}(\mathcal{S})$. In other words, let $X_{\pi_0}, \dotsc X_{\pi_{d-1}}$ be sampled uniformly without replacement from $\{X_1,\dotsc, X_{n}\}$. Then, we have that the quantities $\Bar{X}, \mu$ from Proposition \ref{prop:random_reshuffling} are equal to: 
\begin{eqnarray}
  \Bar{X} = \frac{1}{d}\sum_{j=0}^{d-1} F_{\pi_j}(z), \quad \mu = \frac{1}{n} \sum\limits_{j=1}^{n}F_j(z) =F(z) 
\end{eqnarray}
From Proposition \ref{prop:random_reshuffling}, thus, we get that:
\begin{eqnarray}
\mathbb{E}_{\pi \in \mathcal{S}} \left[\norm{\Bar{X} - \mu}^2\right] = \mathbb{E}_{\pi \in \mathcal{S}} \left[\norm{\frac{1}{d}\sum_{j=0}^{d-1} F_{\pi_j}(z) - F(z)}^2 \right]= \frac{n-d}{d(n-1)}\frac{1}{n} \sum_{j=1}^{n}\left\|F_j(z)-F(z)\right\|^2 \label{in_to_substitute_var}
\end{eqnarray}
Using Proposition \ref{app: bound for const assumpt}, we next bound the sum on the right hand-sight (RHS) as follows:
\begin{eqnarray}
d^2 \mathbb{E}_{\pi \in \mathcal{S}} \left[\norm{\frac{1}{d}\sum_{j=0}^{d-1} F_{\pi_j}(z) - F(z)}^2\right] &\leq& \frac{d(n-d)}{n-1} \left(A \|z - z^*\|^2 + 2\sigma_*^2 \right) \nonumber
\end{eqnarray}
 \end{proof}
 
We, next, provide a lemma bounding the average distance of the iterates inside the $k-$th epoch from the initial point $z_0^k$ in the epoch. 

\begin{lemma}\label{Lemma: squared-norms-epoch-iterates}
    Assume that each $F_i, i\in [n],$ is $L_i-$Lipschitz and the step size of \ref{eq:SEG-RR} satisfy $\gamma_1\leq \frac{1}{3\sqrt{2n(n-1)}L_{max}}$, $ \gamma_2 \leq \frac{1}{\sqrt{n (n-1)} L_{max}}$.
    The iterates of the \ref{eq:SEG-RR} algorithm satisfy the following bound 
    \begin{eqnarray}
        \Expep{\frac{1}{n} \sum\limits_{j=0}^{n-1} \norm{z_{j}^k - z_0^k}^2}&\stackrel{}{\le}& \left[10n^2L^2 + A (25+n)\right]\gamma_1^2 \norm{z_0^k - z^*}^2 + 2(n+25)\gamma_1^2\sigma_*^2 \label{EG-bound}
    \end{eqnarray}
\end{lemma}
\begin{proof}
Using the update rule of \ref{eq:SEG-RR} in \eqref{SEG_update_rule}, we have that:
\begin{eqnarray}
  z_i^{k} = z_{i-1}^k - \gamma_1 F_{\pi_{i-1}^k}(\Bar{z}_{i-1}^k) 
		= z_{0}^k - \gamma_1 \sum\limits_{j=0}^{i-1} ( F_{\pi_j^k}(\Bar{z}_{j}^k) -
			F_{\pi_j^k}(\hat{z}_{0}^k)\big) \nonumber 
\end{eqnarray} 
We, thus, have that:
\begin{eqnarray}
		\norm{z_i^k - z_0^k}^2 &\stackrel{}{=}& \gamma_1^2 i^2 \norm{\frac{1}{i} \sum\limits_{j=0}^{i-1} F_{\pi_j^k}(\Bar{z}_{j}^k)}^2 \nonumber \\ 
        & \stackrel{}{=}& \gamma_1^2 i^2 \norm{\frac{1}{i} \sum\limits_{j=0}^{i-1} \left[F_{\pi_j^k}(\Bar{z}_{j}^k) -F_{\pi_j^k} (z_0^k)+F_{\pi_j^k} (z_0^k) \right] - F(z_0^k) + F(z_0^k)}^2\nonumber\\ 
    &\stackrel{\eqref{ineq1}}{\leq}& 3\gamma_1^2 i \sum\limits_{j=0}^{i-1} \norm{F_{\pi_j^k}  (\Bar{z}_{j}^k)-F_{\pi_j^k} (z_0^k)}^2  + 3\gamma_1^2 i^2 \norm{\frac{1}{i}\sum\limits_{j=0}^{i-1} F_{\pi_j^k} (z_0^k)-F (z_0^k)}^2 + 3\gamma_1^2 i^2 \norm{F(z_0^k)}^2\nonumber
\end{eqnarray}
Using the Lipschitz property of $F_{\pi_j^k}(z)$, we continue as follows
\begin{eqnarray}
    \norm{z_i^k - z_0^k}^2 &\le& 3\gamma_1^2 L_{max}^2 i\sum\limits_{j=0}^{i-1} \norm{\Bar{z}_{j}^k-z_0^k}^2 + 3\gamma_1^2 i^2 \norm{\frac{1}{i}\sum\limits_{j=0}^{i-1} F_{\pi_j^k} (z_0^k)-F (z_0^k)}^2 + 3\gamma_1^2 i^2 \norm{F(z_0^k)}^2\nonumber 
\end{eqnarray}
Substituting the update rule \eqref{SEG_extrapolation_step} of the extrapolation step of \ref{eq:SEG-RR}, we get:
\begin{eqnarray}
    \norm{z_i^k - z_0^k}^2 
    \hspace{-.2cm}&\stackrel{\eqref{SEG_extrapolation_step}}{\leq}& \hspace{-.2cm}3\gamma_1^2 L_{max}^2 i \sum\limits_{j=0}^{i-1} \norm{z_{j}^k - \gamma_2 F_{\pi_j^k}(z_j^k)-z_0^k}^2 +3\gamma_1^2 i^2 \norm{\frac{1}{i}\sum\limits_{j=0}^{i-1} F_{\pi_j^k} (z_0^k)-F (z_0^k)}^2 \hspace{-.2cm}+ 3\gamma_1^2 i^2 \norm{F(z_0^k)}^2\nonumber\\
    \hspace{-.2cm}&\stackrel{\eqref{ineq1}}{\leq}& \hspace{-.2cm}6\gamma_1^2 L_{max}^2 i  \sum\limits_{j=0}^{i-1} \norm{z_{j}^k-z_0^k}^2 + 6\gamma_1^2 \gamma_2^2 L_{max}^2 i  \sum\limits_{j=0}^{i-1} \norm{F_{\pi_j^k}(z_{j}^k)}^2 \nonumber \\
    && +3\gamma_1^2 i^2 \norm{\frac{1}{i}\sum\limits_{j=0}^{i-1} F_{\pi_j^k} (z_0^k)-F (z_0^k)}^2 + 3\gamma_1^2 i^2 \norm{F(z_0^k)}^2\nonumber
\end{eqnarray}

Continuing with further algebraic manipulations, we obtain
        \begin{eqnarray}
        \norm{z_i^k - z_0^k}^2 
        \hspace{-.2cm}&\leq&\hspace{-.2cm} 6\gamma_1^2 L_{max}^2 i  \sum\limits_{j=0}^{i-1} \norm{z_{j}^k-z_0^k}^2 \ + 6\gamma_1^2 \gamma_2^2 L_{max}^2 i  \sum\limits_{j=0}^{i-1} \norm{F_{\pi_j^k}(z_{j}^k) -F_{\pi_j^k}(z_{0}^k) + F_{\pi_j^k}(z_{0}^k)}^2\nonumber \\ 
        \hspace{-.2cm}& \quad&\hspace{-.2cm} +3\gamma_1^2 i^2 \norm{\frac{1}{i}\sum\limits_{j=0}^{i-1} F_{\pi_j^k} (z_0^k)-F (z_0^k)}^2 + 3\gamma_1^2 i^2 \norm{F(z_0^k)}^2\nonumber\\
        \hspace{-.2cm}&\stackrel{\eqref{ineq1}}{\leq}&\hspace{-.2cm}
        6\gamma_1^2 L_{max}^2 i  \sum\limits_{j=0}^{i-1} \norm{z_{j}^k-z_0^k}^2+ 12\gamma_1^2 \gamma_2^2 L_{max}^2 i  \sum\limits_{j=0}^{i-1} \norm{F_{\pi_j^k}(z_{j}^k) -F_{\pi_j^k}(z_{0}^k)}^2\nonumber \\ 
        \hspace{-.2cm}&&\hspace{-.2cm} + 12\gamma_1^2 \gamma_2^2 L_{max}^2 i  \sum\limits_{j=0}^{i-1} \norm{F_{\pi_j^k}(z_{0}^k)}^2 +3\gamma_1^2 i^2 \norm{\frac{1}{i}\sum\limits_{j=0}^{i-1} F_{\pi_j^k} (z_0^k)-F (z_0^k)}^2 \hspace{-.2cm}+ 3\gamma_1^2 i^2 \norm{F(z_0^k)}^2\nonumber
        \end{eqnarray}

Using the Lipschitz property of $F_{\pi_j^k}(z)$ results to 
\begin{eqnarray}
    \norm{z_i^k - z_0^k}^2 \hspace{-.2cm}&\stackrel{}{\leq}&\hspace{-.2cm} 6\gamma_1^2 L_{max}^2 i  \sum\limits_{j=0}^{i-1} \norm{z_{j}^k-z_0^k}^2 + 12\gamma_1^2 \gamma_2^2 L_{max}^4 i  \sum\limits_{j=0}^{i-1} \norm{z_{j}^k - z_0^k}^2  \nonumber \\ 
    \hspace{-.2cm}&&\hspace{-.2cm} + 12\gamma_1^2 \gamma_2^2 L_{max}^2 i  \sum\limits_{j=0}^{i-1} \norm{F_{\pi_j^k}(z_{0}^k)}^2+3\gamma_1^2 i^2 \norm{\frac{1}{i}\sum\limits_{j=0}^{i-1} F_{\pi_j^k} (z_0^k)-F (z_0^k)}^2 + 3\gamma_1^2 i^2 \norm{F(z_0^k)}^2\nonumber \\
    \hspace{-.2cm}&\leq&\hspace{-.2cm} 6\gamma_1^2 L_{max}^2 i  \sum\limits_{j=0}^{n-1} \norm{z_{j}^k-z_0^k}^2 + 12\gamma_1^2 \gamma_2^2 L_{max}^4 i  \sum\limits_{j=0}^{n-1} \norm{z_{j}^k - z_0^k}^2  \nonumber \\ 
    \hspace{-.2cm}&&\hspace{-.2cm} + 12\gamma_1^2 \gamma_2^2 L_{max}^2 i  \sum\limits_{j=0}^{i-1} \norm{F_{\pi_j^k}(z_{0}^k)}^2+3\gamma_1^2 i^2 \norm{\frac{1}{i}\sum\limits_{j=0}^{i-1} F_{\pi_j^k} (z_0^k)-F (z_0^k)}^2 + 3\gamma_1^2 i^2 \norm{F(z_0^k)}^2\nonumber 
\end{eqnarray}

    Letting $G_k = \frac{1}{n} \sum\limits_{j=0}^{n-1} \norm{z_{j}^k - z_0^k}^2$ for brevity, we have that:
    \begin{eqnarray}
        \norm{z_i^k - z_0^k}^2 &\stackrel{}{\leq}& 6\gamma_1^2 L_{max}^2i (1+2\gamma_2^2 L_{max}^2)  nG_k +12 \gamma_1^2\gamma_2^2L_{max}^2 i \sum_{j=0}^{i-1} \norm{F_{\pi_j^k} (z_0^k) - F(z_0^k) + F(z_0^k)}^2 \nonumber\\ 
    && +3\gamma_1^2 i^2 \norm{\frac{1}{i}\sum\limits_{j=0}^{i-1} F_{\pi_j^k} (z_0^k)-F (z_0^k)}^2 +3\gamma_1^2 i^2 \norm{F(z_0^k)}^2\nonumber\\
    &\stackrel{\eqref{ineq1}}{\le}& 6\gamma_1^2 L_{max}^2i (1+2\gamma_2^2 L_{max}^2)nG_k +24 \gamma_1^2\gamma_2^2L_{max}^2 i \sum_{j=0}^{n-1} \norm{F_{\pi_j^k} (z_0^k)- F(z_0^k)}^2\nonumber \\ 
    && + 3\gamma_1^2 i^2 \norm{\frac{1}{i}\sum\limits_{j=0}^{i-1} F_{\pi_j^k} (z_0^k)-F (z_0^k)}^2 + 3\gamma_1^2 (i^2+8 \gamma_2^2 i^2L_{max}^2)\norm{F(z_0^k)}^2\nonumber
    \end{eqnarray}
Taking expectation condition on the filtration $\mathcal{F}_k$ we get:
	\begin{eqnarray}
		\Expep{\norm{z_i^k - z_0^k}^2 } 
        & \stackrel{}{\leq}& 6\gamma_1^2 L_{max}^2i (1+2\gamma_2^2 L_{max}^2)  n \Expep{G_k} + 3\gamma_1^2 (i^2+8 \gamma_2^2 i^2L_{max}^2)\norm{F(z_0^k)}^2 \nonumber \\ 
        && +24 \gamma_1^2\gamma_2^2L_{max}^2 i \Expep{\sum_{j=0}^{i-1} \norm{F_{\pi_j^k} (z_0^k)- F(z_0^k)}^2}\nonumber \\ 
        && + 3\gamma_1^2 i^2 \Expep{\norm{\frac{1}{i}\sum\limits_{j=0}^{i-1} F_{\pi_j^k} (z_0^k)-F (z_0^k)}^2} \label{eq:expected zgap22}
	\end{eqnarray}
        We next bound the last two terms of \eqref{eq:expected zgap22}. 
        Using Proposition \eqref{app: bound for const assumpt}, we have that:
        \begin{eqnarray}
        \frac{1}{n} \sum_{j=0}^{n-1} \norm{F_{\pi_j^k} (z_0^k)- F(z_0^k)}^2 \leq A\norm{z_0^k - z^*}^2 + 2\sigma_*^2 \nonumber \\ 
        \iff \sum_{j=0}^{n-1} \norm{F_{\pi_j^k} (z_0^k)- F(z_0^k)}^2 \leq An \norm{z_0^k - z^*}^2 + 2n \sigma_*^2 \nonumber 
        \end{eqnarray}
        Taking conditional expectation on both sides of the inequality, results to
        \begin{eqnarray}
            \Expep{\sum_{j=0}^{n-1} \norm{F_{\pi_j^k} (z_0^k)- F(z_0^k)}^2} \leq A n \norm{z_0^k - z^*}^2 + 2n \sigma_*^2\label{bound_term_3}
        \end{eqnarray}
        Substituting inequality \eqref{bound_term_3} in \eqref{eq:expected zgap22}, we get:
	\begin{eqnarray}
            \Expep{ \norm{z_i^k - z_0^k}^2 } &\stackrel{\eqref{bound_term_3}}{\leq}& 6\gamma_1^2 L_{max}^2i (1+2\gamma_2^2 L_{max}^2)  n \Expep{G_k} + 3\gamma_1^2 (i^2+8 \gamma_2^2 i^2L_{max}^2)\norm{F(z_0^k)}^2 \nonumber \\ 
        && +24 \gamma_1^2\gamma_2^2L_{max}^2 i \left(An \norm{z_0^k - z^*}^2 + 2n \sigma_*^2\right)\nonumber \\ 
        && + 3\gamma_1^2 i^2 \Expep{\norm{\frac{1}{i}\sum\limits_{j=0}^{i-1} F_{\pi_j^k} (z_0^k)-F (z_0^k)}^2} \nonumber \\
        &=& 6\gamma_1^2 L_{max}^2i (1+2\gamma_2^2 L_{max}^2)  n \Expep{G_k} + 3\gamma_1^2 (i^2+8 \gamma_2^2 i^2L_{max}^2)\norm{F(z_0^k)}^2 \nonumber \\ 
        && +24 \gamma_1^2\gamma_2^2L_{max}^2 A ni \norm{z_0^k - z^*}^2 + 48 \gamma_1^2\gamma_2^2L_{max}^2ni \sigma_*^2\nonumber \\ 
        && + 3\gamma_1^2 i^2 \Expep{\norm{\frac{1}{i}\sum\limits_{j=0}^{i-1} F_{\pi_j^k} (z_0^k)-F (z_0^k)}^2}\label{eq:4} 
        \end{eqnarray}
        
        Using Lemma \ref{prop: bound prr epoch deviations} with $d \leftarrow i$, we can bound the last term in \eqref{eq:4} and get:
        \begin{eqnarray}
    		\Expep{ \norm{z_i^k - z_0^k}^2 } &\stackrel{\eqref{prop: bound prr epoch deviations}}{\le}&
        6\gamma_1^2 L_{max}^2i (1+2\gamma_2^2 L_{max}^2) n\Expep{G_k} +3\gamma_1^2 \left(1+8 \gamma_2^2 L_{max}^2\right) i^2 \norm{F(z_0^k)}^2 \nonumber \\ 
        && + 3 \gamma_1^2A \left[\frac{i(n-i)}{n-1} + 8\gamma_2^2ni L_{max}^2\right] \norm{z_0^k - z^*}^2 \nonumber \\
        && +6\gamma_1^2\sigma_*^2\left[\frac{i(n-i)}{n-1} +8 \gamma_2^2ni L_{max}^2\right] \label{gap iter of w}
    \end{eqnarray}
    Summing over $0\le i \le n-1$ and multiplying with $\frac{1}{n}$, we have that: 
	\begin{eqnarray}
		\Expep{G_k }= \frac{1}{n} \sum_{i=0}^{n-1} \Expep{ \norm{z_{i}^k - z_0^k}^2} &\le& 3\gamma_1^2 L_{max}^2 (1+2\gamma_2^2 L_{max}^2) n(n-1) \Expep{ G_k} +\gamma_1^2 D \norm{F(z_0^k)}^2 \nonumber \\ 
        && + \gamma_1^2A \left[\frac{n+1}{2} + 12 \gamma_2^2 L_{max}^2n(n-1)\right] \norm{z_0^k - z^*}^2 \nonumber \\
        && + \gamma_1^2\sigma_*^2\left[(n+1) +24 \gamma_2^2L_{max}^2n(n-1)\right] \label{inbef}
	\end{eqnarray}
	where we used the facts \begin{eqnarray}
	\frac{1}{n} \sum_{i=0}^{n-1} i = \frac{n-1}{2}, \quad \frac{1}{n}\sum_{i=0}^{n-1} i^2  = \frac{(n-1)(2n-1)}{6},\quad \frac{1}{n}\sum_{i=0}^{n-1} \frac{i(n-i)}{n-1} = \frac{n+1}{6}    \nonumber
	\end{eqnarray}
        and let, also, $D = \left[\frac{(1+8\gamma_2^2 L_{max}^2)(n-1)(2n-1)}{2}\right]$ for brevity.\\
        \\ 
        Rearranging the terms in inequality \eqref{inbef}, we obtain:
        \begin{eqnarray}
           [1 - 3n(n-1) \gamma_1^2(1+2\gamma_2^2 L_{max}^2) L_{max}^2]\Expe{ G_k} \hspace{-.2cm}&\leq&\hspace{-.2cm} \gamma_1^2 D \norm{F(z_0^k)}^2  + \gamma_1^2A \left[\frac{n+1}{2} + 12 \gamma_2^2 L_{max}^2n(n-1)\right] \norm{z_0^k - z^*}^2 \nonumber \\
        && +\gamma_1^2\sigma_*^2\left[(n+1) +24 \gamma_2^2L_{max}^2n(n-1)\right] \nonumber 
        \end{eqnarray}
		Letting $\gamma_2 \leq \frac{1}{L_{max}}$, we get:
  \begin{eqnarray}
    [1 - 9n(n-1) \gamma_1^2 L_{max}^2]\Expe{ G_k} &\leq& \gamma_1^2 D \norm{F(z_0^k)}^2 \nonumber \\ 
        && + \gamma_1^2A \left[\frac{n+1}{2} + 12 \gamma_2^2 L_{max}^2n(n-1)\right] \norm{z_0^k - z^*}^2 \nonumber \\
        && +\gamma_1^2\sigma_*^2\left[(n+1) +24 \gamma_2^2L_{max}^2n(n-1)\right] \nonumber   
  \end{eqnarray}
  
  By letting $ D_1 = [1 - 9n(n-1)\gamma_1^2 L_{max}^2]$ and selecting the stepsize $\gamma_1 < \frac{1}{3\sqrt{n(n-1)}L_{max}}$, we have that  $D_1 > 0$ and thus we get that:
  \begin{eqnarray}
        \Expep{ G_k}&\stackrel{}{\le}& \frac{D\gamma_1^2}{D_1} \norm{F(z_0^k)}^2 \nonumber \\ 
        && + \frac{A\gamma_1^2}{D_1} \left[\frac{n+1}{2} + 12 \gamma_2^2 L_{max}^2n(n-1)\right] \norm{z_0^k - z^*}^2 \nonumber \\
        && +\frac{\gamma_1^2}{D_1} \left[(n+1) +24 \gamma_2^2L_{max}^2n(n-1)\right] \sigma_*^2 \label{bound_on_Gk}
	\end{eqnarray}
 Lastly, substituting the definition of $G_k$ we get that:
 \begin{eqnarray}
 \Expep{\frac{1}{n} \sum\limits_{j=0}^{n-1} \norm{z_{j}^k - z_0^k}^2}&\leq&\frac{D\gamma_1^2}{D_1} \norm{F(z_0^k)}^2 + \frac{A\gamma_1^2}{D_1} \left[\frac{n+1}{2} + 12 \gamma_2^2 L_{max}^2n(n-1)\right] \norm{z_0^k - z^*}^2 \nonumber \\
        && +\frac{\gamma_1^2}{D_1} \left[(n+1) +24 \gamma_2^2L_{max}^2n(n-1)\right] \sigma_*^2 \label{bound_on_G_k2}
 \end{eqnarray}
 Selecting $\gamma_1 \leq \frac{1}{3L_{max}\sqrt{2n(n-1)}}$ we have that:
\begin{eqnarray}
  D_1 = 1 - 9n(n-1)L_{max}^2 \gamma_1^2 \geq \frac{1}{2} \iff \frac{1}{2D_1} \leq 1 \label{bound-for-D_1-2nd-term}
\end{eqnarray}
We, also, have that for $\gamma_2 \leq \frac{1}{\sqrt{n (n-1)} L_{max}}$ we can upper bound $D$ as follows 
\begin{eqnarray}
  D = \frac{(1+8\gamma_2^2 L_{max}^2)(n-1)(2n-1)}{2} \stackrel{}{\leq} \frac{5(n-1)(2n-1)}{2} \leq 5 n^2 \label{bound-for-D-new-2nd-t}
\end{eqnarray}
Substituting the bounds \eqref{bound-for-D_1-2nd-term}, \eqref{bound-for-D-new-2nd-t} to \eqref{bound_on_G_k2}, we obtain
\begin{eqnarray}
    \Expep{\frac{1}{n} \sum\limits_{j=0}^{n-1} \norm{z_{j}^k - z_0^k}^2}&\stackrel{\eqref{bound-for-D_1-2nd-term}, \eqref{bound-for-D-new-2nd-t}}{\leq}& 10n^2 \gamma_1^2 \norm{F(z_0^k)}^2  + 2A\gamma_1^2 \left[\frac{n+1}{2} + 12 \gamma_2^2 L_{max}^2n(n-1)\right] \norm{z_0^k - z^*}^2 \nonumber \\
        && + 2\gamma_1^2 \left[(n+1) +24 \gamma_2^2L_{max}^2n(n-1)\right] \sigma_*^2 \label{Lemma: per epoch iterates without gamma2 constr} \\
        &\stackrel{\gamma_2 \leq \frac{1}{\sqrt{n(n-1)}L_{max}}}{\leq}& 10n^2\gamma_1^2 \norm{F(z_0^k)}^2 + A\gamma_1^2 (25+n) \norm{z_0^k - z^*}^2 + 2\gamma_1^2 (n+25)\sigma_*^2 \nonumber \\
        &\stackrel{\eqref{DefLipschitz}}{\leq}& \left[10n^2L^2 + A (25+n)\right]\gamma_1^2 \norm{z_0^k - z^*}^2 + 2(n+25)\gamma_1^2\sigma_*^2\nonumber
\end{eqnarray}
\end{proof}

In the next Lemma, we bound a term that appears in the proofs of our theorems.

\label{app lemma-for-SC-SC}
\begin{lemma} \label{Lemma: Strongly-monotone-case1}
Assume each $F_i, i \in [n],$ is $L_i-$Lipschitz. If the extrapolation stepsize of \ref{eq:SEG-RR} satisfies $\gamma_2 \leq \frac{1}{L_{max}}$, 
then the following bound holds:
\begin{eqnarray}
\Expep{\sum_{i=0}^{n-1}\left\|F_{\pi_i^k}(\Bar{z}_{i}^k) - F_{\pi_i^k}(\hat{z}_{0}^k)\right\|^2} &\leq& 6L_{max}^2 \Expep {\sum_{i=0}^{n-1} \left\|z^k_{i}-z^k_0\right\|^2}\nonumber \\ && +3L_{max}^2 \gamma_2^2 \Expep{\sum_{i=0}^{n-1} \left\|F_{\pi_i^k}(z^k_{0}) -F(z^k_{0}) \right\|^2}  \nonumber
\end{eqnarray}
\end{lemma}
\begin{proof}
Using the Lipschitz property of $F_i, \forall i \in [n], k \geq 0$, we have that
\begin{eqnarray}
	\sum_{i=0}^{n-1}\left\|F_{\pi_i^k}(\Bar{z}_{i}^k) - F_{\pi_i^k}(\hat{z}_{0}^k)\right\|^2 &\leq\,& L_{max}^2 \sum_{i=0}^{n-1} \left\|\Bar{z}^k_{i}-\hat{z}^k_0\right\|^2 \nonumber\\
 &\stackrel{\eqref{SEG_extrapolation_step}}{=}& L_{max}^2 \sum_{i=0}^{n-1}\left\|z^k_{i}- \gamma_2 F_{\pi_i^k} (z_{i}^k) -z^k_0 +\gamma_2 F(z_0^k) +\gamma_2 F_{\pi_i^k}(z_0^k)  - \gamma_2 F_{\pi_i^k}(z_0^k) \right\|^2\nonumber\\
 &\stackrel{\eqref{ineq1}}{\leq}& 3 L_{max}^2\sum_{i=0}^{n-1} \left\|z^k_{i}-z^k_0\right\|^2 + 3L_{max}^2 \gamma_2^2 \sum_{i=0}^{n-1} \left\|F_{\pi_i^k}(z^k_{i}) -F_{\pi_i^k}(z^k_{0}) \right\|^2  \nonumber \\&& +3L_{max}^2 \gamma_2^2 \sum_{i=0}^{n-1} \left\|F_{\pi_i^k}(z^k_{0}) -F(z^k_{0}) \right\|^2 \nonumber\\
 &\stackrel{\eqref{DefLipschitz}}{\leq}& 3L_{max}^2 (1+\gamma_2^2L_{max}^2) \sum_{i=0}^{n-1} \left\|z^k_{i}-z^k_0\right\|^2 +3L_{max}^2 \gamma_2^2 \sum_{i=0}^{n-1} \left\|F_{\pi_i^k}(z^k_{0}) -F(z^k_{0}) \right\|^2  \nonumber 
\end{eqnarray}
Taking expectation condition on the filtration $\mathcal{F}_k$ and using Proposition \ref{app: bound for const assumpt}, we have
\begin{eqnarray}
\Expep{\sum_{i=0}^{n-1}\left\|F_{\pi_i^k}(\Bar{z}_{i}^k) - F_{\pi_i^k}(\hat{z}_{0}^k)\right\|^2 } &\leq\,& 3L_{max}^2 (1+\gamma_2^2L_{max}^2)\Expep {\sum_{i=0}^{n-1} \left\|z^k_{i}-z^k_0\right\|^2}\nonumber \\ && +3L_{max}^2 \gamma_2^2 \Expep{\sum_{i=0}^{n-1} \left\|F_{\pi_i^k}(z^k_{0}) -F(z^k_{0}) \right\|^2}  \nonumber
\end{eqnarray}
For $\gamma_2 \leq \frac{1}{L_{max}}$, we get:
\begin{eqnarray}
\Expep{\sum_{i=0}^{n-1}\left\|F_{\pi_i^k}(\Bar{z}_{i}^k) - F_{\pi_i^k}(\hat{z}_{0}^k)\right\|^2 }\hspace{-0.25cm} &\leq& \hspace{-0.25cm} 6L_{max}^2 \Expep {\sum_{i=0}^{n-1} \left\|z^k_{i}-z^k_0\right\|^2}\hspace{-0.05cm}+\hspace{-0.05cm}3L_{max}^2 \gamma_2^2 \Expep{\sum_{i=0}^{n-1} \left\|F_{\pi_i^k}(z^k_{0}) -F(z^k_{0}) \right\|^2}  \nonumber
\end{eqnarray}
\end{proof}
\newpage

\section{Proofs for SEG-RR}
\label{app theorems for SEG-RR}
\subsection{Proofs for Strongly Monotone Case}
\subsubsection{Lemma for Iterates in Strongly Monotone Case}
\vspace{+0.1cm}
\begin{lemma} \label{Lemma: SC-SC2}
For SEG-RR if $F$ is $\mu-$strongly monotone and Assumption \ref{DefLipschitz} holds, we have the following bound:
\begin{eqnarray}
   \left\|z_{0}^k - z^* - \gamma_1 nF(\hat{z}_0^k)\right\|^2 &\stackrel{}{\leq}& \left(1 - \frac{1}{2} \gamma_1 n \mu\right)^2 \|z_{0}^k - z^*\|^2 + U \|z_{0}^k - z^*\|^2\nonumber
\end{eqnarray}
    where $\hat{z}_0^k = z_0^k - \gamma_2 F(z_0^k)$ and $U = \left\{ \gamma_1^2 n^2 (2L^2 - \frac{\mu^2}{4}) +2\gamma_1 \gamma_2 nL^2 \left[-1  +\gamma_2 (\gamma_1 nL^2 + \gamma_2L^2 +\mu) \right]\right\}$.
\end{lemma}
\begin{proof}
We have that:
\begin{eqnarray}
			\left\|z_{0}^k - z^* - \gamma_1 n F(\hat{z}_0^k)\right\|^2 
			&=& \|z_{0}^k - z^*\|^2 - 2 \gamma_1 n\langle z_{0}^k - z^*, F(\hat{z}_0^k)\rangle + \gamma_1^2 n^2\| F(\hat{z}_0^k)\|^2 \nonumber\\
			&\stackrel{}{=}& \|z_{0}^k - z^*\|^2 
 - 2 \gamma_1 n\langle z_{0}^k - \gamma_2 F(z_0^k) - z^*, F(\hat{z}_0^k)\rangle \nonumber \\ && + \gamma_1^2 n^2\| F(\hat{z}_0^k)\|^2  
    - 2 \gamma_1\gamma_2  n\langle F(z_0^k), F(\hat{z}_0^k) \rangle\nonumber \\
    &\stackrel{\eqref{hat_z}}{=}& \|z_{0}^k - z^*\|^2 
 - 2 \gamma_1 n\langle\hat{z}_{0}^k - z^*, F(\hat{z}_0^k)\rangle + \gamma_1^2 n^2\| F(\hat{z}_0^k)\|^2 \nonumber \\ && 
    - 2 \gamma_1\gamma_2  n\langle F(z_0^k), F(\hat{z}_0^k) \rangle\nonumber \\
    &\stackrel{\eqref{DefSM}}{\leq}& \|z_{0}^k - z^*\|^2 
 -2\gamma_1 n\mu \|\hat{z}_{0}^k - z^*\|^2 + \gamma_1^2 n^2\| F(\hat{z}_0^k)\|^2 \nonumber \\ && 
    - 2 \gamma_1\gamma_2  n\langle F(z_0^k), F(\hat{z}_0^k) \rangle\nonumber \\
    &\stackrel{}{=}& \|z_{0}^k - z^*\|^2 
 -2\gamma_1 n\mu \|\hat{z}_{0}^k - z^*\|^2 + \gamma_1^2 n^2\| F(\hat{z}_0^k) - F(z^k_0) + F(z^k_0) \|^2 \nonumber \\ && 
    - 2 \gamma_1\gamma_2  n\langle F(z_0^k), F(\hat{z}_0^k) \rangle\nonumber \\
    &\stackrel{\eqref{ineq1}}{\leq}& \|z_{0}^k - z^*\|^2 
 -2\gamma_1 n\mu \|\hat{z}_{0}^k - z^*\|^2 + 2\gamma_1^2 n^2\| F(\hat{z}_0^k) - F(z^k_0)\|^2 \nonumber \\ &&  + 2\gamma_1^2 n^2\|F(z^k_0) \|^2 - 2 \gamma_1\gamma_2  n\langle F(z_0^k), F(\hat{z}_0^k) \rangle\nonumber 
 \end{eqnarray}
 Using the Lipschitz property of $F$, we get 
 \begin{eqnarray}
    \left\|z_{0}^k - z^* - \gamma_1 n F(\hat{z}_0^k)\right\|^2  &\stackrel{\eqref{DefLipschitz}}{\leq}& \|z_{0}^k - z^*\|^2 
 -2\gamma_1 n \mu \| \hat{z}^k_0 -z^*\|^2 + 2 \gamma_1^2 n^2 L^2 \|\hat{z}_{0}^k - z_{0}^k\|^2\nonumber \\ && 
     + 2\gamma_1^2 n^2\|F(z^k_0) \|^2  - 2 \gamma_1\gamma_2  n\langle F(z_0^k), F(\hat{z}_0^k) \rangle\nonumber \\ 
    &\stackrel{\eqref{ineq inner product}}{\leq}& \|z_{0}^k - z^*\|^2 -2\gamma_1 n \mu \| \hat{z}^k_0 -z^*\|^2 + 2 \gamma_1^2 n^2 L^2 \| \hat{z}^k_0 -z^k_0\|^2 \nonumber \\ && 
 + 2\gamma_1 n (\gamma_1 n - \gamma_2) \|F(z^k_0) \|^2  +  2 \gamma_1\gamma_2  n \|F(z_0^k) - F(\hat{z}_0^k)\|^2 \nonumber \\
&\stackrel{\eqref{DefLipschitz}}{\leq}& \|z_{0}^k - z^*\|^2 -2\gamma_1 n \mu \| \hat{z}^k_0 -z^*\|^2 + 2 \gamma_1 n L^2 (\gamma_1 n + \gamma_2) \| \hat{z}^k_0 -z^k_0\|^2 \nonumber \\ && 
 + 2\gamma_1 n (\gamma_1 n - \gamma_2) \|F(z^k_0) \|^2  \nonumber
 \end{eqnarray}
 Substituting the definition of $\hat{z}^k_0$ we have 
 \begin{eqnarray}
     \left\|z_{0}^k - z^* - \gamma_1 n F(\hat{z}_0^k)\right\|^2 
 &\stackrel{\eqref{hat_z}}{\leq}& \|z_{0}^k - z^*\|^2 -2\gamma_1 n \mu \| z^k_0  -\gamma_2 F(z^k_0) -z^*\|^2 \nonumber \\ && 
 + 2\gamma_1 n (\gamma_1 n - \gamma_2 +\gamma_2^2L^2 (\gamma_1 n + \gamma_2)) \|F(z^k_0) \|^2\nonumber 
\end{eqnarray}

We, next, make use of inequality \eqref{ineqa_b} to obtain \vspace{-0.2cm}
\begin{eqnarray}
 \left\|z_{0}^k - z^* - \gamma_1 n F(\hat{z}_0^k)\right\|^2 
 &\stackrel{\eqref{ineqa_b}}{\leq}& (1 - \gamma_1 n \mu) \|z_{0}^k - z^*\|^2 \nonumber \\ && 
 + 2\gamma_1 n \left[\gamma_1 n - \gamma_2 +\gamma_2^2L^2 (\gamma_1 n + \gamma_2) +\gamma_2^2 \mu\right] \|F(z^k_0) \|^2\nonumber   \\
 &=& \left(1 - \frac{1}{2} \gamma_1 n \mu\right)^2 \|z_{0}^k - z^*\|^2 - \frac{1}{4}\gamma_1^2 n^2 \mu^2 \|z_{0}^k - z^*\|^2 \nonumber \\ 
 && + 2\gamma_1 n \left[\gamma_1 n - \gamma_2 +\gamma_2^2L^2 (\gamma_1 n + \gamma_2) +\gamma_2^2 \mu\right] \|F(z^k_0) \|^2\nonumber  \\
 &\stackrel{\eqref{DefLipschitz}}{\leq}& \left(1 - \frac{1}{2} \gamma_1 n \mu\right)^2 \|z_{0}^k - z^*\|^2 + U \|z_{0}^k - z^*\|^2\nonumber 
\end{eqnarray}
where $U = \left\{ \gamma_1^2 n^2 (2L^2 - \frac{\mu^2}{4}) +2\gamma_1 \gamma_2 nL^2 \left[-1  +\gamma_2 (\gamma_1 nL^2 + \gamma_2L^2 +\mu) \right]\right\}$.
\end{proof}

\begin{lemma} \label{Lemma step sizes scsc}
    If the step size of the SEG-RR algorithm satisfy $\gamma_1 \leq \frac{\mu}{10L_{max}^2\sqrt{10n^2+2n+54}}$, $\gamma_2 = 2 \gamma_1$ then the following holds:
    \begin{eqnarray}
       \frac{U}{1 - \frac{\gamma_1 n \mu}{2}} + \frac{6nCL_{max}^2\gamma_1}{\mu} &\leq& \frac{\gamma_1 n \mu}{4} \label{upper-b1}
  \end{eqnarray}
where the constants are $C = 2 \left[(25+n)A +10n^2 L^2\right] \gamma_1^2 +A \gamma_2^2$,\\
$U = \left\{ \gamma_1^2 n^2 (2L^2 - \frac{\mu^2}{4}) +2\gamma_1 \gamma_2 nL^2 \left[-1  +\gamma_2 (\gamma_1 nL^2 + \gamma_2L^2 +\mu) \right]\right\}$.
\end{lemma}
\begin{proof}
We have that:
    \begin{eqnarray}
       \frac{U}{1 - \frac{\gamma_1 n \mu}{2}} + \frac{6nCL_{max}^2\gamma_1}{\mu} &\leq& \frac{\gamma_1 n \mu}{4} \nonumber \\
       \iff \frac{\gamma_1^2 n^2 (2L^2 - \frac{\mu^2}{4}) +2\gamma_1 \gamma_2 nL^2 \left[-1  +\gamma_2 (\gamma_1 nL^2 + \gamma_2L^2+\mu) \right]}{1 - \frac{\gamma_1 n \mu}{2}} +\frac{6nCL_{max}^2\gamma_1}{\mu} &\leq& \frac{\gamma_1 n \mu}{4}\nonumber \\
       \iff \frac{\gamma_1 n (2L^2 - \frac{\mu^2}{4}) +2 \gamma_2 L^2 \left[-1  +\gamma_2 (\gamma_1 nL^2 + \gamma_2L^2+\mu) \right]}{1 - \frac{\gamma_1 n \mu}{2}}  + \frac{6CL_{max}^2}{\mu} &\leq& \frac{\mu}{4}\nonumber 
      \end{eqnarray}
      Rearranging the terms we get
      \begin{eqnarray}
          && \gamma_1 n \left(2L^2 - \frac{\mu^2}{8}\right) +2 \gamma_2 L^2 \left[-1  +\gamma_2 (\gamma_1 nL^2 + \gamma_2L^2+\mu) \right] + \frac{6CL_{max}^2}{\mu} \left(1 - \frac{\gamma_1 n \mu}{2}\right) - \frac{\mu}{4} \leq 0 \nonumber\\
          &\stackrel{}{\iff}& \gamma_1 n \left(2L^2 - \frac{\mu^2}{8}\right) +2 \gamma_2 L^2 \left[-1  +\gamma_2 (\gamma_1 nL^2 + \gamma_2L^2+\mu) \right] \nonumber \\
          && + \frac{6\left\{2 \left[(25+n)A +10n^2 L^2\right] \gamma_1^2 +A \gamma_2^2\right\}L_{max}^2}{\mu} \left(1 - \frac{\gamma_1 n \mu}{2}\right) - \frac{\mu}{4} \leq 0 \nonumber \\
          &\stackrel{\gamma_2 = 2 \gamma_1}{\iff}& \gamma_1 n \left(2L^2 - \frac{\mu^2}{8}\right) +4 L^2\gamma_1 \left\{-1  +2\gamma_1[\gamma_1 (n+2) L^2 +\mu] \right\} \nonumber \\ 
          && + \frac{6\left[(54+2n) A +20n^2 L^2\right]\gamma_1^2L_{max}^2}{\mu} \left(1 - \frac{\gamma_1 n \mu}{2}\right) - \frac{\mu}{4} \leq 0 \label{eq-to-continue-from-k}
      \end{eqnarray}
    For $\gamma_1 \leq \frac{1}{3\sqrt{2n(n-1)}L_{max}}$, we have that
      \begin{eqnarray}
          \left\{-1  +2\gamma_1[\gamma_1 (n+2) L^2 +\mu] \right\}&\stackrel{\gamma_1 \leq \frac{1}{3\sqrt{2}nL_{max}}}{\leq}& \left\{-1  +\frac{2(n+2)L^2}{18n(n-1)L_{max}^2} +\frac{\mu}{3\sqrt{2n(n-1)}L_{max}} \right\} \nonumber \\ 
          &\leq& \left[-1  + \frac{1}{9(n-1)L_{max}^3\sqrt{2n(n-1)}} +\frac{\mu}{9n(n-1)L_{max}^2}\right] \nonumber \\ 
          &\leq& -1 + \frac{1}{9} +\frac{1}{9} = - \frac{7}{9} \label{eq1_fak}
      \end{eqnarray}
      Thus, using \eqref{eq1_fak} and the fact that $(1 - \frac{\gamma_1 n \mu}{2}) \leq 1$ in \eqref{eq-to-continue-from-k} it suffices to ensure that 
      \begin{eqnarray}
         && \gamma_1 n \left(2L^2 - \frac{\mu^2}{8}\right) - \frac{28L^2}{9}\gamma_1 + \frac{6\left[(54+2n) A +20n^2 L^2\right]L_{max}^2}{\mu}\gamma_1^2- \frac{\mu}{4} \leq 0 \nonumber \\ 
         && \gamma_1 \left[2 n L^2 - \frac{28L^2}{9}-\frac{n\mu^2}{8}\right]+ \frac{6\left[(54+2n) A +20n^2 L^2\right]L_{max}^2}{\mu}\gamma_1^2- \frac{\mu}{4} \leq 0 \nonumber
      \end{eqnarray}
    Thus, it suffices to ensure that:
    \begin{eqnarray}
       2 n L^2 \gamma_1+ \frac{6\left[(54+2n) A +20n^2 L^2\right]L_{max}^2}{\mu}\gamma_1^2- \frac{\mu}{4} \leq 0 \label{stepsize_equation_sc-sc}
    \end{eqnarray}
    In order, now, to derive a simple expression for the stepsize $\gamma_1$, instead of solving the quadratic inequality \eqref{stepsize_equation_sc-sc}, we choose $\gamma_1$ such that
    \begin{eqnarray}
        2 n L^2 \gamma_1 - \frac{\mu}{8} \leq 0 &\text{and}& + \frac{6\left[(54+2n) A +20n^2 L^2\right]L_{max}^2}{\mu}\gamma_1^2- \frac{\mu}{8} \leq 0 \nonumber \\
        \gamma_1 \leq \frac{\mu}{16 n L^2} &\text{and}& \gamma_1\leq \frac{\mu}{\sqrt{48\left[(54+2n) A +20n^2 L^2\right]L_{max}^2}} \leq 0 \label{before-final-stepsizes}
    \end{eqnarray}
     Using the fact that $A = \frac{2}{n} \sum\limits_{i=0}^{n-1} L_i^2 \leq 2 L_{max}^2$ and $L^2 \leq L_{max}^2$, we observe that it suffices 
     \begin{eqnarray}
         \gamma_1 \leq \frac{\mu}{16 n L^2} &\text{and}& \gamma_1\leq \frac{\mu}{10L_{max}^2\sqrt{10n^2+2n+54}} \leq 0 \nonumber \\
         &\iff& \gamma_1 \leq \frac{\mu}{10L_{max}^2\sqrt{10n^2+2n+54}} \nonumber
     \end{eqnarray}
     Lastly, incorporating the initial constraint that $\gamma_1 \leq \frac{1}{3\sqrt{2n(n-1)}L_{max}}$, 
     it suffices to choose $\gamma_2 = 2 \gamma_1$ and \\
     \begin{eqnarray}
        \gamma_{1} &=& \min \left\{\frac{1}{3\sqrt{2n(n-1)}L_{max}}, \frac{\mu}{10L_{max}^2\sqrt{10n^2+2n+54}}\right\} = \frac{\mu}{10L_{max}^2\sqrt{10n^2+2n+54}} \nonumber
     \end{eqnarray}
     \end{proof}

\subsubsection{Proof of Theorem \ref{thm SC-SC4}}
\label{app thm SC-SC4}
\begin{proof}
Denote with $\Bar{z}^k_i = z_i^k - \gamma_2 F_{\pi_i^k}(z_i^k)$ the extrapolation step of the SEG-RR algorithm from \eqref{SEG_extrapolation_step} and let 
\begin{eqnarray*}
    \hat{z}_i^k = z_i^k - \gamma_2 F(z_i^k)
\end{eqnarray*}
We start with the proof of the 1st point (inequality \eqref{sc1}).\\

\textbf{Proof of Inequality \eqref{sc1}.}
Using the update rule \eqref{SEG_update_rule} of SEG-RR, we have that:
\begin{eqnarray}
		z_0^{k+1} &=&z_{n}^k\nonumber\\ &\stackrel{\eqref{SEG_update_rule}}{=}& z_{n-1}^k - \gamma_1 F_{\pi^k_{n-1}}(\Bar{z}_{n-1}^k) \nonumber\\
		&\stackrel{\eqref{SEG_update_rule}}{=}\;& z_{0}^k -\gamma_1 \sum_{i=0}^{n-1} F_{\pi_i^k}(\Bar{z}_{i}^k)\nonumber\\
		&=\;& z_{0}^k -\gamma_1 n F(\hat{z}^k_0) \label{gewg.gwe}
		\;- \gamma_1 \sum_{i=0}^{n-1} ( F_{\pi_i^k}(\Bar{z}_{i}^k) -
			F_{\pi_i^k}(\hat{z}_{0}^k)\big) \label{unro-ineq}
\end{eqnarray}
where we have expressed an epoch-level update by using \eqref{SEG_update_rule} and in the last step we have added and subtracted the term $\gamma_1 n F(\hat{z}^k_0) = \gamma_1 \sum\limits_{i=0}^{n-1}F_{\pi_i^k}(\hat{z}_{0}^k)$, utilizing the finite sum structure of the operator $F(z) = \frac{1}{n} \sum\limits_{i=0}^{n-1} F_{\pi_i^k}(z)$.
Subtracting $z^\star$ from both sides of \eqref{unro-ineq} and taking the norm, we get:
\begin{eqnarray}
    \|z_0^{k+1} - z^*\|^2
    &=& \left\|z_{0}^k - z^* - \gamma_1 n F(\hat{z}^k_0)-\gamma_1\sum_{i=0}^{n-1} (F_{\pi_i^k}(\Bar{z}^k_{i}) - F_{\pi_i^k}^k(\hat{z}_0^k))\right\|^2\label{norm_to_exp}
\end{eqnarray}
We, next, use Young's inequality \eqref{Youngwitht=1-gamma} with $t = 1 - \frac{\gamma_1 n\mu}{2}\in (0,1)$ in order to expand the norm in the right-hand side (RHS) of \eqref{norm_to_exp} and then simplify the resulting terms. Specifically, we obtain:
    \begin{eqnarray}
        \|z_0^{k+1} - z^*\|^2 &\stackrel{\eqref{Youngwitht=1-gamma}}{\leq}& \frac{\left\|z_{0}^k - z^* - \gamma_1 n F(\hat{z}^k_0) \right\|^2 }{1 - \frac{\gamma_1 n\mu}{2}} +\frac{2}{\gamma_1 n\mu} \left\|\gamma_1 \sum_{i=0}^{n-1}(F_{\pi_i^k}(\Bar{z}^k_{i}) - F_{\pi_i^k}(\hat{z}_0^k))\right\|^2\label{ineq_after-Young}
    \end{eqnarray}
    Taking expectation condition on the filtration $\mathcal{F}^k$ (history of $z_0^k$) and using Lemma \ref{Lemma: Strongly-monotone-case1} to bound the second term in the right-hand side of \eqref{ineq_after-Young} and get 
 \begin{eqnarray}
    \Expep{\|z^{k+1}_0 - z^*\|^2} &\leq& \frac{1}{1 - \frac{\gamma_1 n\mu}{2}}\underbrace{\left\|z_{0}^k - z^* \!-\! \gamma_1 n F(\hat{z}^k_0) \right\|^2 }_{T_1} + \frac{12\gamma_1L_{max}^2}{\mu} \underbrace{\Expep{\sum_{i=0}^{n-1} \left\|z^k_{i}-z^k_0\right\|^2}}_{T_2} \nonumber \\ 
 && +\frac{6\gamma_1\gamma_2^2L_{max}^2}{\mu} \underbrace{\Expep{\sum_{i=0}^{n-1} \left\|F_{\pi_i^k}(z^k_{0}) -F(z^k_{0}) \right\|^2}}_{T_3} \label{ineq_with_3terms}
 \end{eqnarray}
 We, next, use the upper bounds from Lemma~\ref{Lemma: SC-SC2}, \ref{Lemma: squared-norms-epoch-iterates} and Proposition~\ref{app: bound for const assumpt} with $\gamma_2 = 2\gamma_1, \gamma_1 \leq \frac{1}{3\sqrt{2n(n-1)}L_{max}}$ in order to bound the terms $T_1, T_2, T_3$ as follows:
\begin{eqnarray}
   T_1 &\stackrel{}{\leq}& \left(1 - \frac{1}{2} \gamma_1 n \mu\right)^2 \|z_{0}^k - z^*\|^2 +U \|z_{0}^k - z^*\|^2 \quad \quad \label{main bound T-new2}\\
   T_2 &\stackrel{}{\leq}& \left[10n^2L^2 + A (25+n)\right]n\gamma_1^2 \norm{z_0^k - z^*}^2 + 2n(n+25)\gamma_1^2\sigma_*^2 \label{bound_T_2} \\
   T_3 &\stackrel{}{\leq}& A \|z - z^*\|^2 +2 \sigma_*^2 \label{bound_T_3}
\end{eqnarray}
where $U = \left\{ \gamma_1^2 n^2 (2L^2 - \frac{\mu^2}{4}) +2\gamma_1 \gamma_2 nL^2 \left[-1  +\gamma_2 (\gamma_1 nL^2 + \gamma_2L^2 +\mu) \right]\right\}$. 

Substituting the upper bounds \eqref{main bound T-new2}, \eqref{bound_T_2}, \eqref{bound_T_3} into \eqref{ineq_with_3terms} and letting $C = 2 \left[(25+n)A +10n^2 L^2\right] \gamma_1^2 +A \gamma_2^2$ for brevity, we get:
\begin{eqnarray}
    \Expep{\|z^{k+1}_0 - z^*\|^2}
    &\leq& \left(1 - \frac{1}{2} \gamma_1 n \mu + \frac{U}{1 - \frac{\gamma_1 n \mu}{2}} + \frac{6nCL_{max}^2\gamma_1}{\mu}\right) \|z_{0}^k - z^*\|^2  \nonumber \\ 
    && + \frac{24nL_{max}^2\gamma_1}{\mu}\left[(25+n) \gamma_1^2+\gamma_2^2 \right]\sigma_*^2 \quad \quad \label{last-ineq-nnew}
\end{eqnarray}
Choosing the step size $\gamma_2 = 2 \gamma_1, \gamma_1 \leq \frac{\mu}{10L_{max}^2\sqrt{10n^2+2n+54}}$ appropriately and using Lemma \eqref{Lemma step sizes scsc}, we can upper bound the term
    \begin{eqnarray}
       \frac{U}{1 - \frac{\gamma_1 n \mu}{2}} + \frac{6nCL_{max}^2\gamma_1}{\mu} &\leq& \frac{\gamma_1 n \mu}{4} \label{upper-b1-thm1}
  \end{eqnarray}

Substituting \eqref{upper-b1-thm1} into \eqref{last-ineq-nnew}, we obtain
\begin{eqnarray}
    \Expep{\|z^{k+1}_0 - z^*\|^2}&\stackrel{}{\leq}&
    \Bigg(1 - \frac{1}{4}\gamma_1 n\mu \Bigg){\|z_{0}^k - z^*\|^2}+\frac{24nL_{max}^2\gamma_1\left[(25+n) \gamma_1^2+\gamma_2^2\right]}{\mu} \sigma_*^2 \quad \quad \label{sc-sc-res-to}
\end{eqnarray}

Taking expectation on both sides and using the tower property of expectations, we have that:
\begin{eqnarray}
        \Expe{\|z_0^{k+1} - z^*\|^2 }
        &\leq& \Bigg(1 - \frac{1}{4} \gamma_1 n\mu \Bigg)\|z^k_0- z^*\|^2+\frac{24nL_{max}^2\gamma_1 \left[(25+n) \gamma_1^2+\gamma_2^2\right]}{\mu} \sigma_*^2\quad \quad \label{eq_after_expe}\\
     &\leq&\Bigg(1 - \frac{1}{4} \gamma_1 n\mu \Bigg)^{k+1}\|z^0_0- z^*\|^2 +\frac{24nL_{max}^2\gamma_1 \left[(25+n) \gamma_1^2+\gamma_2^2\right]}{\mu} \sum_{i=1}^k (1 - \frac{1}{4} \gamma_1 n\mu)^i  \sigma_*^2 \nonumber\\ 
   &\leq& \Bigg(1 - \frac{1}{4} \gamma_1 n\mu \Bigg)^{k+1}\|z^0_0- z^*\|^2 +\frac{24nL_{max}^2\gamma_1 \left[(25+n) \gamma_1^2+\gamma_2^2\right]}{\mu} \sum_{i=1}^{\infty} (1 - \frac{1}{4} \gamma_1 n\mu)^i \sigma_*^2 \nonumber\\ 
   &\stackrel{}{=}& \left(1 - \frac{\gamma_1 n\mu}{4} \right)^{k+1} {\|z_0- z^*\|^2} +\frac{96L_{max}^2}{\mu^2} \left[(25+n) \gamma_1^2+\gamma_2^2\right] \sigma_*^2 \quad \quad \label{sc-sc result-4}
\end{eqnarray}

\textbf{Proof of Equation \eqref{sc2}.}
From \eqref{sc-sc result-4} we have that 
\begin{eqnarray}
    \Expe{\|z_0^{K} - z^*\|^2 } &\leq& \left(1 - \frac{\gamma_1 n\mu}{4} \right)^K {\|z_0- z^*\|^2} +\frac{96L_{max}^2}{\mu^2} \left[(25+n) \gamma_1^2+\gamma_2^2\right] \sigma_*^2\nonumber\\
    &\stackrel{\gamma_2 = 2 \gamma_1}{\leq}& \left(1 - \frac{\gamma_1 n\mu}{4} \right)^K {\|z_0- z^*\|^2} +\frac{96(29+n)L_{max}^2}{\mu^2} \gamma_1^2\sigma_*^2\nonumber \\
    &\stackrel{\eqref{ineq exponential}}{\leq}& e^{\frac{-\gamma_1 nK\mu}{4}} {\|z_0- z^*\|^2} +\frac{96(29+n)L_{max}^2}{\mu^2} \gamma_1^2\sigma_*^2 \label{intem_}
\end{eqnarray}
We substitute $\gamma_1 = \min\left\{\frac{\mu}{10L_{max}^2\sqrt{10n^2+2n+54}}, \frac{4\log (n^{1/2}K)}{\mu nK}\right\} \leq \frac{4\log (n^{1/2}K)}{\mu nK}$ and bound the second term in the right-hand side (RHS) of \eqref{intem_} as
\begin{eqnarray}
    \frac{96(29+n)L_{max}^2}{\mu^2} \gamma_1^2\sigma_*^2 \leq \frac{96(29+n)L_{max}^2}{\mu^4} \frac{\log^2 (n^{1/2}K)}{n^2K^2}\sigma_*^2 \label{second_term_rhs}
\end{eqnarray}
Substituting \eqref{second_term_rhs} into \eqref{intem_}, we obtain the following:
\begin{equation}
\Expe{\|z_0^{K+1} - z^*\|^2 } \leq e^{-\frac{\gamma_1 nK\mu}{4}} {\|z_0- z^*\|^2} +\frac{96(29+n)L_{max}^2}{\mu^4} \frac{16\log^2 (n^{1/2}K)}{n^2K^2}\sigma_*^2 \label{second_term_edited}
\end{equation}
We now consider the following cases:
\paragraph{Case 1: $\frac{\mu}{10L_{max}^2\sqrt{10n^2+2n+54}}\leq \frac{\log (n^{1/2}K)}{\mu nK}$} In this case we have that $\gamma_1 = \frac{\mu}{10L_{max}^2\sqrt{10n^2+2n+54}}$, which implies that the RHS of \eqref{second_term_edited} is bounded by 
\begin{eqnarray}
    && e^{-\frac{\gamma_1 nK\mu}{4}} {\|z_0- z^*\|^2} +\frac{96(29+n)L_{max}^2}{\mu^4} \frac{16\log^2 (n^{1/2}K)}{n^2K^2}\sigma_*^2 \nonumber \\ 
    &\leq& e^{-\frac{nK\mu^2}{40L_{max}^2\sqrt{10n^2+2n+54}}} {\|z_0- z^*\|^2} +\frac{96(29+n)L_{max}^2}{\mu^4} \frac{16\log^2 (n^{1/2}K)}{n^2K^2}\sigma_*^2 \nonumber\\
    &\leq& e^{-\frac{K\mu^2}{40\sqrt{12}L_{max}^2}} {\|z_0- z^*\|^2} +\frac{96(29+n)L_{max}^2}{\mu^4} \frac{16\log^2 (n^{1/2}K)}{n^2K^2}\sigma_*^2\label{case1}
s\end{eqnarray}

\paragraph{Case 2: $\frac{4\log (n^{1/2}K)}{\mu nK} \leq \frac{\mu}{10L_{max}^2\sqrt{10n^2+2n+54}}$}
In this case we have that $\gamma_1 = \frac{4\log (n^{1/2}K)}{\mu nK}$, which implies that the RHS of \eqref{second_term_edited} is bounded by 
\begin{eqnarray}
    && e^{-\frac{\gamma_1 nK\mu}{4}} {\|z_0- z^*\|^2} +\frac{96(29+n)L_{max}^2}{\mu^4} \frac{16 \log^2 (n^{1/2}K)}{n^2K^2}\sigma_*^2 \nonumber \\ 
    &\leq& \frac{1}{nK^2}{\|z_0- z^*\|^2} +\frac{96(29+n)L_{max}^2}{\mu^4} \frac{16\log^2 (n^{1/2}K)}{n^2K^2}\sigma_*^2 \label{case2}
\end{eqnarray} 
Taking the maximum of the right-hand side of \eqref{case1} and \eqref{case2} and using the inequality $\max\{ a, b\} \leq a + b$, we obtain the desired result which holds for both cases:
\begin{equation}
  \Expe{\|z_0^{K} - z^*\|^2 } \leq e^{-\frac{K\mu^2}{40\sqrt{12}L_{max}^2}} {\|z_0- z^*\|^2} + \frac{1}{nK^2}{\|z_0- z^*\|^2} +2 \frac{96(29+n)L_{max}^2}{\mu^4} \frac{16 \log^2 (n^{1/2}K)}{n^2K^2}\sigma_*^2\nonumber  
\end{equation}
Suppressing constant and logarithmic terms, we get the final result
\begin{eqnarray}
\Expe{\|z_0^{K+1} - z^*\|^2 }&=& \Tilde{\mathcal{O}}\left( e^{-\frac{K\mu^2}{L_{max}^2}} + \frac{1}{{nK^2}}\right)\nonumber
\end{eqnarray}\end{proof}

\subsection{Proofs for Affine Case}
\vspace{+0.1cm}
\subsubsection{Lemma for Iterates in Affine Case}
\vspace{+0.1cm}
\begin{lemma} \label{lemma: bil_2ndTerm_bound}
    Suppose that $F_i, \forall i \in [n-1]$ are affine and Assumption \ref{DefLipschitz} holds. If the step size of SEG-RR Algorithm satisfy $\gamma_2 = 4\gamma_1, \gamma_1 \in \Big(0,\frac{1}{3\sqrt{2n(n-1)}L_{max}}\Big]$ then the following holds
    \begin{eqnarray}
      \Expep{\left\| \sum_{i=0}^{n-1}Q_{\pi_i^k}(I - \gamma_2 Q_{\pi_i^k}) (z^k_i -z^k_0 +z^*) + (I-\gamma_2 Q_{\pi_i^k}) b_{\pi_i^k}\right\|^2} 
     &\stackrel{}{\leq}& 2 n L_{max} D_2 \gamma_1^2 \norm{z_0^k - z^*}^2 \quad \nonumber \\ 
     && + 2L_{max} \left[\gamma_2^2 + 2n\gamma_1^2(n+25)\right] \sigma_*^2 \nonumber
    \end{eqnarray}
    where $D_2 = \left[10n^2L^2 + (n+25) A\right]$.
\end{lemma}
\begin{proof}
    We have that 
    \begin{eqnarray}
    && \left\| \sum_{i=0}^{n-1}Q_{\pi_i^k}(I - \gamma_2 Q_{\pi_i^k}) (z^k_i -z^k_0 +z^*) + (I-\gamma_2 Q_{\pi_i^k}) b_{\pi_i^k}\right\|^2 \nonumber \\
     &=& \left\| \sum_{i=0}^{n-1}Q_{\pi_i^k}(I - \gamma_2 Q_{\pi_i^k}) (z^k_i -z^k_0) +\sum_{i=0}^{n-1}Q_{\pi_i^k}(I - \gamma_2 Q_{\pi_i^k})z^*+ (I-\gamma_2 Q_{\pi_i^k}) b_{\pi_i^k}\right\|^2\nonumber \\
    &=& \left\| \sum_{i=0}^{n-1}Q_{\pi_i^k}(I - \gamma_2 Q_{\pi_i^k}) (z^k_i -z^k_0) +(Qz^* +b)- \gamma_2\sum_{i=0}^{n-1} Q_{\pi_i^k} (Q_{\pi_i^k} z^*+b_{\pi_i^k})\right\|^2\nonumber\\
    &\stackrel{\eqref{F_w_in_bilinear_games}}{=}& \left\| \sum_{i=0}^{n-1}Q_{\pi_i^k}(I - \gamma_2 Q_{\pi_i^k}) (z^k_i -z^k_0) +F(z^*) - \gamma_2\sum_{i=0}^{n-1} Q_{\pi_i^k} F_{\pi_i^k}(z^*) \right\|^2\nonumber 
    \end{eqnarray}
    Using the fact that $F(z^*)=0$, we continue our derivation as follows
    \begin{eqnarray}
    && \left\| \sum_{i=0}^{n-1}Q_{\pi_i^k}(I - \gamma_2 Q_{\pi_i^k}) (z^k_i -z^k_0 +z^*) + (I-\gamma_2 Q_{\pi_i^k}) b_{\pi_i^k}\right\|^2 \nonumber \\
    &\stackrel{}{=}& \left\| \sum_{i=0}^{n-1}Q_{\pi_i^k}(I - \gamma_2 Q_{\pi_i^k}) (z^k_i -z^k_0) - \gamma_2\sum_{i=0}^{n-1} Q_{\pi_i^k} F_{\pi_i^k}(z^*) \right\|^2\nonumber \\ 
    &\stackrel{\eqref{ineq1}}{\leq}& 2 \left\| \sum_{i=0}^{n-1}Q_{\pi_i^k}(I - \gamma_2 Q_{\pi_i^k}) (z^k_i -z^k_0)\right\|^2 + 2\left\| \gamma_2\sum_{i=0}^{n-1} Q_{\pi_i^k} F_{\pi_i^k}(z^*) \right\|^2\nonumber\\ 
    &\stackrel{\eqref{ineq1}}{\leq}& 2 n \sum_{i=0}^{n-1} \left\|Q_{\pi_i^k}(I - \gamma_2 Q_{\pi_i^k}) (z^k_i -z^k_0)\right\|^2 + 2\gamma_2^2 \sum_{i=0}^{n-1}\left\|Q_{\pi_i^k} F_{\pi_i^k}(z^*) \right\|^2 \nonumber\\
    &\leq& 2 n \left(L_{max} - \gamma_2 \lambda_{min}(Q_{\pi_i^k})\right) \sum_{i=0}^{n-1} \left\|z^k_i -z^k_0\right\|^2 + 2L_{max} \gamma_2^2 \sum_{i=0}^{n-1}\left\| F_{\pi_i^k}(z^*) \right\|^2\nonumber\\
    &=& 2 n \left(L_{max} - \gamma_2 \lambda_{min}(Q_{\pi_i^k})\right) \sum_{i=0}^{n-1} \left\|z^k_i -z^k_0\right\|^2 + 2L_{max} \gamma_2^2 \sigma_*^2 \label{add_subtract_terms}
 \end{eqnarray} 
 where $\sigma_*^2 = \frac{1}{n} \sum_{i=1}^{n}\left\| F_{\pi_i^k}(z^*) \right\|^2$ and $\lambda_{min}(Q_{\pi_i^k})= \min\limits_{i \in[n]} \min\limits_{\lambda} \lambda(Q_i)$ is the minimum eigenvalue of all $Q_{\pi_i^k}, i \in [n]$. 
 Taking expectation on both sides condition on the filtration $\mathcal{F}^k$, we get
 \begin{eqnarray}
     && \Expep{\left\| \sum_{i=0}^{n-1}Q_{\pi_i^k}(I - \gamma_2 Q_{\pi_i^k}) (z^k_i -z^k_0 +z^*) + (I-\gamma_2 Q_{\pi_i^k}) b_{\pi_i^k}\right\|^2} \nonumber\\
     &\leq& 2 n \left(L_{max} - \gamma_2 \lambda_{min}(Q_{\pi_i^k})\right) \Expep{\sum_{i=0}^{n-1} \left\|z^k_i -z^k_0\right\|^2} + 2L_{max} \gamma_2^2 \sigma_*^2\nonumber
 \end{eqnarray}
 Using lemma \ref{Lemma: squared-norms-epoch-iterates}, we continue our derivation as follows:
 \begin{eqnarray}
     && \Expep{\left\| \sum_{i=0}^{n-1}Q_{\pi_i^k}(I - \gamma_2 Q_{\pi_i^k}) (z^k_i -z^k_0 +z^*) + (I-\gamma_2 Q_{\pi_i^k}) b_{\pi_i^k}\right\|^2}\nonumber \\
     &\stackrel{\text{Lemma } \ref{Lemma: squared-norms-epoch-iterates}}{\leq}& 2 n \left(L_{max} - \gamma_2 \lambda_{min}(Q_{\pi_i^k})\right) \left\{10n^2\gamma_1^2 \norm{F(z_0^k)}^2 + A\gamma_1^2 (25+n) \norm{z_0^k - z^*}^2\right\}\quad \nonumber \\ 
     && +4 n\left(L_{max} - \gamma_2 \lambda_{min}(Q_{\pi_i^k})\right)(n+25)\gamma_1^2\sigma_*^2 + 2L_{max} \gamma_2^2 \sigma_*^2 \nonumber\\
     &\leq& 2 n L_{max}\left\{10n^2\gamma_1^2 \norm{F(z_0^k)}^2 + A\gamma_1^2 (25+n) \norm{z_0^k - z^*}^2\right\} \nonumber \\ 
     && + 2L_{max} \left[\gamma_2^2 + 2n(n+25) \gamma_1^2\right] \sigma_*^2 \nonumber
 \end{eqnarray}
 Lastly, applying the Lipschitz property of $F$ (Assumption~\ref{DefLipschitz}), we get
 \begin{eqnarray}
     \Expep{\left\| \sum_{i=0}^{n-1}Q_{\pi_i^k}(I - \gamma_2 Q_{\pi_i^k}) (z^k_i -z^k_0 +z^*) + (I-\gamma_2 Q_{\pi_i^k}) b_{\pi_i^k}\right\|^2} 
     &\stackrel{\eqref{DefLipschitz}}{\leq}& 2 n L_{max}\gamma_1^2 \left[10n^2L^2 + (n+25) A\right] \norm{z_0^k - z^*}^2 \quad \nonumber \\ 
     && + 2L_{max} \left[\gamma_2^2 + 2n\gamma_1^2(n+25)\right] \sigma_*^2 \nonumber
 \end{eqnarray}
\end{proof}

\begin{lemma} \label{lemma: stepsizes-bilinear}
    If $\gamma_2 = 4 \gamma_1$, $\gamma_1\in \Big(0, \frac{\lambda_{min}^+(Q)}{2\sqrt{120} nL^2_{max}}],$ then the following hold:
    \begin{eqnarray}
      1 -\gamma_1 n(\lambda_{min}^+(Q) -\gamma_2 L_{max}^2) +\frac{2L_{max}\gamma_1^3}{\lambda_{min}^+(Q)}\left[10n^2L^2 + (n+25) A\right] &\leq& 1 - \frac{\gamma_1 n \lambda_{min}^+(Q)}{2} \nonumber
    \end{eqnarray}
\end{lemma}
\begin{proof}   
Selecting the step size $\gamma_1, \gamma_2 = 4\gamma_1$ such that 
    \begin{eqnarray}
       && 1 -\gamma_1 n(\lambda_{min}^+(Q) -4 \gamma_1 L_{max}^2) +\frac{2L_{max}\gamma_1^3}{\lambda_{min}^+(Q)}\left[10n^2L^2 + (n+25) A\right] \leq 1 - \frac{\gamma_1 n \lambda_{min}^+(Q)}{2} \nonumber \\
        &\iff& \lambda_{min}^+(Q) -4 \gamma_1 L_{max}^2 - \frac{2L_{max}\gamma_1^2}{\lambda_{min}^+(Q)n}\left[10n^2L^2 + (n+25) A\right] \geq \frac{\lambda_{min}^+(Q)}{2} \nonumber \\
       &\iff& \frac{\lambda_{min}^+(Q)}{2} -4 L_{max}^2\gamma_1 - \frac{2L_{max}\gamma_1^2}{\lambda_{min}^+(Q)n}\left[10n^2L^2 + (n+25) A\right] \geq 0 \nonumber 
    \end{eqnarray}
    Using the fact that $A = \frac{2}{n} \sum\limits_{i=0}^{n-1} L^2_i \leq 2 L_{max}^2$, it suffice to select the stepsize $\gamma_1$ such that
    \begin{eqnarray}
       \frac{\lambda_{min}^+(Q)}{2} -4 L_{max}^2\gamma_1 - \frac{2L_{max}^3\gamma_1^2}{\lambda_{min}^+(Q)n}(5n^2 + 2n +50) \geq 0 \label{stepsizes_bilinear_quadratic_equation} 
    \end{eqnarray}
    In order to get simple expressions for the step size instead of solving the quadratic inequality \eqref{stepsizes_bilinear_quadratic_equation} we select $\gamma_1$ such that
    \begin{eqnarray}
      \frac{\lambda_{min}^+(Q)}{4} - 4 L_{max}^2\gamma_1 \geq 0 &\text{and}& \frac{\lambda_{min}^+(Q)}{4} - \frac{2L_{max}^3\gamma_1^2}{\lambda_{min}^+(Q)n}(5n^2 + 2n +50) \geq 0 \nonumber \\
      \iff \gamma_1 \leq \frac{\lambda_{min}^+(Q)}{16 L_{max}^2} &\text{and}& \gamma_1 \leq \frac{\lambda_{min}^+(Q)}{2L_{max}\sqrt{2L_{max}(5n + 2 +\frac{50}{n})}} \nonumber
    \end{eqnarray}
    Thus, the above two constraints are satisfied for stepsize   
    \begin{eqnarray}
        \gamma_1 \leq \frac{\lambda_{min}^+(Q)}{2L^2_{max}\sqrt{2(5n + 52)}} \nonumber
    \end{eqnarray}
    Combining, lastly, the requirement that $\gamma_1 \leq \frac{1}{3L_{max} \sqrt{2 n(n-1)}}$, we have that
    \begin{eqnarray}
        \gamma_{1} &=& \min\left\{ \frac{\lambda_{min}^+(Q)}{2L^2_{max}\sqrt{2(5n + 52)}}, \frac{1}{3L_{max} \sqrt{2 n(n-1)}}\right\} \nonumber \\
        &\leq& \min\left\{ \frac{\lambda_{min}^+(Q)}{2L^2_{max}\sqrt{120n}}, \frac{1}{3\sqrt{2}n L_{max}}\right\} \nonumber\\
        &\leq& \frac{\lambda_{min}^+(Q)}{2\sqrt{120} nL^2_{max}} \nonumber
    \end{eqnarray}
    Thus, in order for the inequality in the statement of the Lemma to hold, 
    it suffices to select the step size $\gamma_2 = 4 \gamma_1 , \gamma_1 \in \left(0, \gamma_{1, max}\right]$, 
    where $\gamma_{1, max} = \frac{\lambda_{min}^+(Q)}{2\sqrt{120} nL^2_{max}}$.
  \end{proof}

\subsubsection{Proof of Theorem \ref{thm-bil-1}}
\label{app thm-bil-1}
\begin{proof}
Let the step size satisfy $\gamma_2 = \alpha \gamma_1, \alpha > 0$. 
We note that, due to the closed form expression \eqref{F_w_in_bilinear_games} of the operator $F$, we have that the following hold
\begin{eqnarray}
    F_{\pi_i^k}(\Bar{z}^k_i) &=&  Q_{\pi_i^k} (I - \alpha\gamma_1 Q_{\pi_i^k}) z_i^k + (I - \alpha\gamma_1  Q_{\pi_i^k})b_{\pi_i^k} \label{F_w_bar_update}\\
    F(\hat{z}) &=& Q (I - \alpha\gamma_1 Q) z + (I - \alpha\gamma_1 Q)b \label{F_w_hat_update} \\
    F(z^*) = 0 &\iff& Qz^* = - b \label{F_w_star_equals_0_bil}
\end{eqnarray}
\textbf{Proof of Inequality \eqref{res: affine1}.} We have that:
\begin{eqnarray}
	   z_0^{k+1} &\coloneqq&z_{n}^k\nonumber\\ &\stackrel{\eqref{SEG_extrapolation_step}}{=}& z_{n-1}^k - \gamma_1 F_{\pi_{n-1}^k}(\Bar{z}_{n-1}^k) \nonumber\\
		&=\;& z_{0}^k -\gamma_1 \sum_{i=0}^{n-1} F_{\pi_i^k}(\Bar{z}_{i}^k) \quad \label{init2} 
\end{eqnarray}
Subtracting $z^\star$ from both sides of \eqref{init2} and taking the norm, we get:
\begin{eqnarray}
\|z_0^{k+1} - z^*\|^2
    &=&\left\| z_0^k - z^* -\gamma_1 \sum_{i=0}^{n-1} F_{\pi_i^k}(\Bar{z}_{i}^k)\right\|^2 \nonumber \\
    &\stackrel{\eqref{F_w_bar_update}}{=}& \left\| z_0^k - z^* -\gamma_1 \sum_{i=0}^{n-1} \left[Q_{\pi_i^k}(I - \gamma_2 Q_{\pi_i^k}) z^k_i + (I-\gamma_2 Q_{\pi_i^k}) b_{\pi_i^k}\right] \right\|^2 \nonumber \\
    &\stackrel{}{=}& \Big\| \left[I - \gamma_1 \sum_{i=0}^{n-1} Q_{\pi_i^k}(I - \gamma_2 Q_{\pi_i^k}) \right](z_0^k - z^*) \nonumber \\ 
    && \quad - \gamma_1 \sum_{i=0}^{n-1} \left[Q_{\pi_i^k}(I - \gamma_2 Q_{\pi_i^k}) (z^k_i -z^k_0 +z^*) + (I-\gamma_2 Q_{\pi_i^k}) b_{\pi_i^k}\right]\Big\|^2 \nonumber
\end{eqnarray}
where in the last step we have added and subtracted the term $\gamma_1 \sum_{i=0}^{n-1} Q_{\pi_i^k}(I - \gamma_2 Q_{\pi_i^k})(z_0^k - z^*)$. 

\newpage
Using Young's inequality \eqref{Youngforbil} with $t = 1 -\gamma_1 n(\lambda_{min}^+(Q) - \gamma_2 L_{max}^2) \in (0,1)$, we obtain:
\begin{eqnarray}
\|z_0^{k+1} - z^*\|^2 &\stackrel{\eqref{Youngforbil}}{\leq}& \frac{\left\|I - \gamma_1 \sum_{i=0}^{n-1} Q_{\pi_i^k}(I - \gamma_2 Q_{\pi_i^k})\right\|^2 \|z_0^k - z^*\|^2}{[1 -\gamma_1 n(\lambda_{min}^+(Q) - \gamma_2 L_{max}^2)]} \nonumber \\
&& +\frac{\left\| \gamma_1 \sum_{i=0}^{n-1} \left[Q_{\pi_i^k}(I - \gamma_2 Q_{\pi_i^k}) (z^k_i -z^k_0 +z^*) + (I-\gamma_2 Q_{\pi_i^k}) b_{\pi_i^k}\right]\right\|^2}{\gamma_1n (\lambda_{min}^+(Q) -\gamma_2 L_{max}^2)}\nonumber \\ 
&\leq& [1 -\gamma_1 n(\lambda_{min}^+(Q) - \gamma_2 L_{max}^2)] \|z_0^k - z^*\|^2\nonumber \\
&& +\frac{\gamma_1 \left\| \sum_{i=0}^{n-1}Q_{\pi_i^k}(I - \gamma_2 Q_{\pi_i^k}) (z^k_i -z^k_0 +z^*) + (I-\gamma_2 Q_{\pi_i^k}) b_{\pi_i^k}\right\|^2}{n(\lambda_{min}^+(Q) -\gamma_2 L_{max}^2)}\nonumber
\end{eqnarray}

Taking expectation condition on the filtration $\mathcal{F}^k$, we have that 
\begin{eqnarray}
    \Expep{\|z_0^{k+1} - z^*\|^2} &\leq& [1 -\gamma_1 n(\lambda_{min}^+(Q) - \gamma_2 L_{max}^2)] \|z_0^k - z^*\|^2\nonumber \\
&& +\frac{\gamma_1 \Expep{\left\| \sum_{i=0}^{n-1}Q_{\pi_i^k}(I - \gamma_2 Q_{\pi_i^k}) (z^k_i -z^k_0 +z^*) + (I-\gamma_2 Q_{\pi_i^k}) b_{\pi_i^k}\right\|^2}}{n(\lambda_{min}^+(Q) -\gamma_2 L_{max}^2)}\label{interm_1_bil}
\end{eqnarray}

We, next, use Lemma \ref{lemma: bil_2ndTerm_bound} to bound the second norm in the RHS of \eqref{interm_1_bil}. Thus, letting $\gamma_1 \leq \frac{1}{3\sqrt{2n(n-1)}L_{max}}$ and using Lemma \ref{lemma: bil_2ndTerm_bound} into \eqref{interm_1_bil} we get 
\begin{eqnarray}
    \Expep{\|z_0^{k+1} - z^*\|^2} &\stackrel{\text{Lemma } \ref{lemma: bil_2ndTerm_bound}}{\leq}& \left[1 -\gamma_1 n(\lambda_{min}^+(Q) - \gamma_2 L_{max}^2) +\frac{2L_{max}(10n^2L^2 + (n+25) A) \gamma_1^3}{(\lambda_{min}^+(Q) -\gamma_2 L_{max}^2)}\right] \|z_0^k - z^*\|^2\nonumber \\ 
    && + \frac{2L_{max}\gamma_1}{n (\lambda_{min}^+(Q) -\gamma_2 L_{max}^2)} \left[\gamma_2^2 + 2n\gamma_1^2(n+25)\right] \sigma_*^2 \label{to_use_substitution_of_F_i_k}
\end{eqnarray}

Selecting the step size $\gamma_2 = 4 \gamma_1$, $\gamma_1\leq \frac{\lambda_{min}^+(Q)}{2\sqrt{120} nL^2_{max}}$ and using Lemma \ref{lemma: stepsizes-bilinear} we have that:
\begin{eqnarray}
     [1 -\gamma_1 n(\lambda_{min}^+(Q) - \gamma_2 L_{max}^2)] +\frac{2L_{max}\gamma_1^3}{(\lambda_{min}^+(Q) -\gamma_2 L_{max}^2)} \left[10n^2L^2 + (n+25) A\right]&\leq& 1 - \frac{\gamma_1 n \lambda_{min}^+(Q)}{2} \nonumber
\end{eqnarray} 

Thus, for the selected step size inequality \eqref{to_use_substitution_of_F_i_k} gives:
\begin{eqnarray}
\Expep{\|z^{k+1}_0 - z^*\|^2}
    &\leq& \left(1 - \frac{1}{2}\gamma_1n\lambda_{min}^+(Q)\right) \|z_{0}^k - z^*\|^2 
    + \frac{2L_{max}\gamma_1}{n\lambda_{min}^+(Q)}\left[\gamma_2^2 + 4n\gamma_1^2(n+25)\right] \sigma_*^2\quad \quad
    \label{billast}
\end{eqnarray}

Taking expectation on both sides and using the tower property of expectation we have that
\begin{eqnarray}
    \Expe{\|z_0^{k+1} - z^*\|^2 } &\leq& \left(1 - \frac{1}{2}\gamma_1n\lambda_{min}^+(Q)\right) \|z_{0}^k - z^*\|^2+\frac{2L_{max}\gamma_1}{n\lambda_{min}^+(Q)}\left[\gamma_2^2 + 4n\gamma_1^2(n+25)\right]\sigma_*^2\label{bil-after-excpe}
\end{eqnarray}
Using the concavity of the min operator and the definition of $\text{dist}(z, \mathcal{Z}_*) = \min\limits_{z^* \in \mathcal{Z}_*} \|z - z^*\|^2$, we obtain
\begin{eqnarray}
    \Expe{\text{dist}(z_0^{k+1}, \mathcal{Z}_*)} &\leq& \min_{z^* \in \mathcal{Z}_*} \Expe{\|z_0^{k+1} - z^*\|^2 } \nonumber \\
    &\leq& \left(1 - \frac{1}{2}\gamma_1n\lambda_{min}^+(Q)\right) \text{dist}(z_0^{k}, \mathcal{Z}_*)+\frac{2L_{max}\gamma_1}{n\lambda_{min}^+(Q)}\left[\gamma_2^2 + 4n\gamma_1^2(n+25)\right]\sigma_*^2
\end{eqnarray}
where $\sigma_*^2$ here denotes $\sigma_*^2 = \min\limits_{z^* \in \mathcal{Z}_*}\frac{1}{n} \sum_{i=0}^{n-1} \|F_i(z^*)\|^2$.\\
Unrolling the recursion, we conclude that
\begin{eqnarray}
     \Expe{\text{dist}(z_0^{k+1}, \mathcal{Z}_*)}
    &\leq& \left(1 - \frac{1}{2}\gamma_1n\lambda_{min}^+(Q)\right)^{k+1} \text{dist}(z_0, \mathcal{Z}_*)
    +\frac{4L_{max}}{\lambda_{min}^2(Q)n^2}\left[\gamma_2^2 + 4n(n+25)\gamma_1^2\right]\sigma_*^2  \label{final result bil1}
\end{eqnarray}
\textbf{Proof of Equation \eqref{res: affine2}.} From inequality \eqref{final result bil1}, the following holds 
\begin{eqnarray}
    \Expe{\text{dist}(z_0^{K}, \mathcal{Z}_*) } &\leq& \left(1 - \frac{1}{2}\gamma_1n\lambda_{min}^+(Q)\right)^K {\text{dist}(z_0, \mathcal{Z}_*)}
    +\frac{4L_{max}}{\lambda_{min}^2(Q)n^2}\left[\gamma_2^2 + 4n(n+25)\gamma_1^2\right]\sigma_*^2 \nonumber\\
    &\stackrel{\gamma_2 = 4\gamma_1}{\leq}& \left(1 - \frac{1}{2}\gamma_1n\lambda_{min}^+(Q)\right)^K {\text{dist}(z_0, \mathcal{Z}_*)}
    +\frac{4L_{max}}{\lambda_{min}^2(Q)n^2}\left(4n^2+25n+16\right) \gamma_1^2\sigma_*^2\nonumber \\
    &\stackrel{\eqref{ineq exponential}}{\leq}& e^{-\frac{\gamma_1nK\lambda_{min}^+(Q)}{2}} {\text{dist}(z_0, \mathcal{Z}_*)} +\frac{4L_{max}\left(4n^2+25n+16\right)}{\lambda_{min}^2(Q)n^2}\gamma_1^2\sigma_*^2 \label{intem2}
\end{eqnarray}
We substitute $\gamma_1 = \min\left\{\frac{\lambda_{min}^+(Q)}{2\sqrt{120} nL^2_{max}}, \frac{2\log (n^{1/2}K)}{\lambda_{min}^+(Q) nK}\right\} \leq \frac{2\log (n^{1/2}K)}{\lambda_{min}^+(Q) nK}$ and bound the second term in the right-hand side (RHS) of \eqref{intem2} as
\begin{eqnarray}
    \frac{4L_{max}\left(4n^2+25n+16\right)}{\lambda_{min}^2(Q)n^2}\gamma_1^2\sigma_*^2 \leq \frac{4L_{max}\left(4n^2+25n+16\right)}{{\lambda_{min}^+(Q)}^4n^2}\frac{4\log^2 (n^{1/2}K)}{n^2K^2}\sigma_*^2\label{second_term_rhs_bil}
\end{eqnarray}
Substituting \eqref{second_term_rhs_bil} into \eqref{intem2}, we obtain the following:
\begin{equation}
\Expe{\text{dist}(z_0^{K}, \mathcal{Z}_*)} \leq e^{-\frac{\gamma_1nK\lambda_{min}^+(Q)}{2}} {\text{dist}(z_0, \mathcal{Z}_*)} +\frac{4L_{max}\left(4n^2+25n+16\right)}{{\lambda_{min}^+(Q)}^4n^2}\frac{4\log^2 (n^{1/2}K)}{n^2K^2}\sigma_*^2\label{second_term_edited_bil}
\end{equation}
We now consider the following cases:
\paragraph{Case 1: $\frac{\lambda_{min}^+(Q)}{2\sqrt{120} nL^2_{max}}\leq \frac{2\log (n^{1/2}K)}{\lambda_{min}^+(Q) nK}$} In this case we have that $\gamma_1 = \frac{\lambda_{min}^+(Q)}{2\sqrt{120} nL^2_{max}}$, which implies that the RHS of \eqref{second_term_edited_bil} is bounded by 
\begin{eqnarray}
    && e^{-\frac{\gamma_1nK\lambda_{min}^+(Q)}{2}} {\text{dist}(z_0, \mathcal{Z}_*)} +\frac{4L_{max}\left(4n^2+25n+16\right)}{{\lambda_{min}^+(Q)}^4n^2}\frac{4\log^2 (n^{1/2}K)}{n^2K^2}\sigma_*^2 \nonumber \\ 
    &\leq& e^{-\frac{K\lambda_{min}^2(Q)}{4\sqrt{120}L^2_{max}}} {\|z_0- z^*\|^2} +\frac{4L_{max}\left(4n^2+25n+16\right)}{{\lambda_{min}^+(Q)}^4n^2}\frac{4\log^2 (n^{1/2}K)}{n^2K^2}\sigma_*^2 \label{case1_bil}
\end{eqnarray}

\paragraph{Case 2: $\frac{2\log (n^{1/2}K)}{\lambda_{min}^+(Q) nK} \leq \frac{\lambda_{min}^+(Q)}{2\sqrt{120} nL^2_{max}}$}
In this case we have that $\gamma_1 = \frac{2\log (n^{1/2}K)}{\lambda_{min}^+(Q) nK}$, which implies that the RHS of \eqref{second_term_edited_bil} is bounded by 
\begin{eqnarray}
    && e^{-\frac{\gamma_1nK\lambda_{min}^+(Q)}{2}} {\text{dist}(z_0, \mathcal{Z}_*)} +\frac{4L_{max}\left(4n^2+25n+16\right)}{{\lambda_{min}^+(Q)}^4n^2}\frac{4\log^2 (n^{1/2}K)}{n^2K^2}\sigma_*^2 \nonumber \\ 
    &\leq& \frac{1}{nK^2} {\|z_0- z^*\|^2} +\frac{4L_{max}\left(4n^2+25n+16\right)}{{\lambda_{min}^+(Q)}^4n^2}\frac{4\log^2 (n^{1/2}K)}{n^2K^2}\sigma_*^2 \label{case2_bil}
\end{eqnarray} 
Taking the maximum of the right-hand side of \eqref{case1_bil} and \eqref{case2_bil} and using the inequality $\max\{ a, b\} \leq a + b$, we obtain the desired result which holds for both cases:
\begin{equation}
  \Expe{\text{dist}(z_0^{K}, \mathcal{Z}_*) } \leq  e^{-\frac{K\lambda_{min}^2(Q)}{4\sqrt{120}L^2_{max}}} {\text{dist}(z_0, \mathcal{Z}_*)} + \frac{1}{nK^2}{\|z_0- z^*\|^2} +2 \frac{4L_{max}\left(4n^2+25n+16\right)}{{\lambda_{min}^+(Q)}^4n^2}\frac{4\log^2 (n^{1/2}K)}{n^2K^2}\sigma_*^2 \nonumber  
\end{equation}
Suppressing constant and logarithmic terms, we get the final result
\begin{eqnarray}
\Expe{\text{\text{dist}}(z_0^{K}, \mathcal{Z}_*)}&=& \Tilde{\mathcal{O}}\left( e^{\frac{K\lambda_{min}^2(Q)}{4\sqrt{120}L^2_{max}}} + \frac{1}{{nK^2}}\right)\nonumber
\end{eqnarray}
\end{proof}

\subsection{Proofs for Monotone Case}
In Section \ref{aoskdnaosdal}, we start by proving three lemmas that will be useful in the proof of the main theorem for the monotone case. Lemma \ref{Lemma: my0} provides a bound on the iterates of \ref{eq:SEG-RR}, Lemma \ref{Lemma: new iterates in epoch} bounds the distance of the iterates inside an epoch and Lemma \ref{Lemma: new my} bounds a term that appears in the proof of Theorem \ref{theorem monotone my}. Next, in Section \ref{theorem monotone app my}, we have the proof of Theorem \ref{theorem monotone my}. Lastly, in Section \ref{sec: iter_comp_seg_rr}, we provide the iteration complexity result of \ref{eq:SEG-RR} that allows us to obtain the desired $\mathcal{O}\left(\frac{1}{\sqrt{nK}}\right)$ after a certain number of epochs.

\subsubsection{Lemmas for Monotone Case}
\label{aoskdnaosdal}

\begin{lemma}\label{Lemma: my0}
    Suppose that $F$ is monotone and each $F_i, \forall i \in [n],$ is $L_i-$Lipschitz. 
    Then, for any $z^* \in \mathcal{Z}_*$ it holds that 
    \begin{eqnarray}
        \|z_0^{k+1} - z^*\|^2 &\leq& (1+2 \gamma_1 n L^2+2 \gamma_1 L_{max}) \|z_0^k-z^*\|^2 \nonumber \\
        && - n\left[2\gamma_2^2n(1 - \gamma_1 n L^2-\gamma_1 L_{max}) - 3\left(\gamma_2 n + 1\right)^2 \gamma_1\right] \|F(z_0^k)\|^2\nonumber \\ 
      && - (2 - \gamma_1 L_{max}) \|z_0^{k+1} -\Tilde{z}^k_0\|^2 + \gamma_1\left[3n\gamma_1^2(1 + L^2)+\gamma_1+\frac{2}{L_{max}}\right]\left\|\sum\limits_{i=0}^{n-1} F_{\pi^k_i}(\Bar{z}_i^k)\right\|^2\nonumber
    \end{eqnarray}
\end{lemma}
\begin{proof}
  From equation \eqref{init2}, we have that 
  \begin{eqnarray}
      z_0^{k+1} &=& z_0^k - \gamma_1 \sum_{i=0}^{n-1} F_{\pi^k_i}(\Bar{z}_i^k)\label{my0},
  \end{eqnarray}
  Let us define $\Tilde{z}^k_0$ to be the following expression: 
  \begin{eqnarray}
      \Tilde{z}^k_0 &\triangleq& z_0^k - \gamma_2 n F(z_0^k) \label{z_tilde def}.
  \end{eqnarray}
  which we will use later in different parts of our proof. 

  Let $z^* \in \mathcal{Z}_*$. Then from the expressions of $z_0^{k+1}$ and $\Tilde{z}^k_0$ in \eqref{my0}, and \eqref{z_tilde def} respectively, we are able to obtain the following:  

  \begin{eqnarray}
    0 & =& \langle z_0^{k+1}-z_0^k + \gamma_1 \sum_{i=0}^{n-1} F_{\pi^k_i}(\Bar{z}_i^k),z^*-z_0^{k+1} \rangle \label{eq_1a}
       \\
    0 & =& \langle \Tilde{z}^k_0 -z_0^k + \gamma_2n F(z_0^k),z_0^{k+1}-\Tilde{z}^k_0\rangle \label{eq_1b}
  \end{eqnarray}

  In addition, since the operator $F$ is monotone it holds that: 
  \begin{eqnarray}
    \gamma_1 n \langle F(z^{k+1}_0), \Tilde{z}^k_0 - z^{k+1}_0\rangle &\leq& \gamma_1 n \langle F(\Tilde{z}^k_0), \Tilde{z}^k_0 - z^{k+1}_0\rangle\label{eq_1c}
  \end{eqnarray}

  By summing these three relations \eqref{eq_1a}, \eqref{eq_1b}, \eqref{eq_1c} together and rearranging the terms, we get
  \begin{eqnarray*}
    \gamma_1 n \langle F(z^{k+1}_0), \Tilde{z}^k_0 - z^{k+1}_0\rangle &\leq&  \langle z_0^{k+1}-z_0^k,z^*-z_0^{k+1}\rangle +\gamma_1 \langle \sum_{i=0}^{n-1} F_{\pi^k_i}(\Bar{z}_i^k), \Tilde{z}^k_0-z_0^{k+1}\rangle + \gamma_1 \langle \sum_{i=0}^{n-1} F_{\pi^k_i}(\Bar{z}_i^k),z^*-\Tilde{z}^k_0\rangle \nonumber \\ 
    && +\langle \Tilde{z}^k_0-z_0^k,
    z_0^{k+1}-\Tilde{z}^k_0\rangle - \gamma_2 n \langle F(z_0^k), \Tilde{z}^k_0-z_0^{k+1}\rangle \nonumber \\
    && + \gamma_1 n \langle F(\Tilde{z}^k_0), \Tilde{z}^k_0 - z^{k+1}_0\rangle
  \end{eqnarray*}
  Now, by multiplying both sides with $2$ and rearranging, we have
  \begin{eqnarray*}
    2\gamma_1 n \langle F(z^{k+1}_0), \Tilde{z}^k_0 - z^{k+1}_0\rangle &\leq& 2 \langle z_0^{k+1}-z_0^k,z^* -z_0^{k+1}\rangle + 2 \langle \Tilde{z}^k_0-z_0^k,
    z_0^{k+1}-\Tilde{z}^k_0\rangle \nonumber \\
    &&+ 2\gamma_1 \langle \sum_{i=0}^{n-1} F_{\pi^k_i}(\Bar{z}_i^k), \Tilde{z}^k_0-z_0^{k+1}\rangle
     -2 \gamma_2 n \langle F(z_0^k), \Tilde{z}^k_0-z_0^{k+1}\rangle \nonumber \\ 
     && -2 \gamma_1 \langle \sum_{i=0}^{n-1} F_{\pi^k_i}(\Bar{z}_i^k),\Tilde{z}^k_0 - z^*\rangle + 2\gamma_1 n \langle F(\Tilde{z}^k_0), \Tilde{z}^k_0 - z^{k+1}_0\rangle
  \end{eqnarray*}
  Using the identity $2\langle a,b\rangle = \| a+b\|^2 - \|a\|^2-\|b\|^2$ for the inner products $\langle z_0^{k+1}-z_0^k,z^*-z_0^{k+1}\rangle$ and $\langle \Tilde{z}^k_0-z_0^k,
    z_0^{k+1}-\Tilde{z}^k_0\rangle$:
  \begin{eqnarray}
    2\gamma_1 n \langle F(z^{k+1}_0), \Tilde{z}^k_0 - z^{k+1}_0\rangle &\leq& - \|z_0^{k+1} - z^*\|^2 - \|z_0^{k+1}-z_0^k\|^2 + \| z^k_0 - z^*\|^2 \nonumber \\
    && - \|z_0^k-\Tilde{z}^k_0\|^2- \|z_0^{k+1} -\Tilde{z}^k_0\|^2 + \|z_0^{k+1}-z_0^k\|^2\nonumber\\
      && +2\gamma_1 \langle \sum_{i=0}^{n-1} F_{\pi^k_i}(\Bar{z}_i^k), \Tilde{z}^k_0-z_0^{k+1}\rangle - 2n \gamma_2 \langle F(z_0^k), \Tilde{z}^k_0-z_0^{k+1}\rangle \nonumber \\
      && -2\gamma_1 \langle \sum_{i=0}^{n-1} F_{\pi^k_i}(\Bar{z}_i^k),\Tilde{z}^k_0-z^*\rangle + 2 \gamma_1 n \langle F(\Tilde{z}^k_0), \Tilde{z}^k_0 - z^{k+1}_0\rangle\nonumber
  \end{eqnarray}
  By bringing the term $-\|z_0^{k+1} - z^*\|^2$ on the left-hand side,
  \begin{eqnarray}
    \|z_0^{k+1} - z^*\|^2 + 2\gamma_1 n \langle F(z^{k+1}_0), \Tilde{z}^k_0 - z^{k+1}_0\rangle &\leq& \|z_0^k-z^*\|^2  - \|z_0^k-\Tilde{z}^k_0\|^2- \|z_0^{k+1} -\Tilde{z}^k_0\|^2\nonumber\\
      && +2\gamma_1 \langle \sum_{i=0}^{n-1} F_{\pi^k_i}(\Bar{z}_i^k), \Tilde{z}^k_0-z_0^{k+1}\rangle - 2n \gamma_2 \langle F(z_0^k), \Tilde{z}^k_0-z_0^{k+1}\rangle \nonumber \\
      && -2\gamma_1 \langle \sum_{i=0}^{n-1} F_{\pi^k_i}(\Bar{z}_i^k),\Tilde{z}^k_0-z^*\rangle + 2 \gamma_1 n \langle F(\Tilde{z}^k_0), \Tilde{z}^k_0 - z^{k+1}_0\rangle \label{my1}
  \end{eqnarray}

  Using inequality \eqref{ineq inner product}, we can lower bound the left hand-side of \eqref{my1} as follows:
  \begin{eqnarray}
      \|z_0^{k+1} - z^*\|^2 + 2\gamma_1 n \langle F(z^{k+1}_0), \Tilde{z}^k_0 - z^{k+1}_0\rangle &\stackrel{\eqref{ineq inner product}}{=}& \|z_0^{k+1} - z^*\|^2 + \gamma_1 n \| F(z^{k+1}_0)\|^2 + \gamma_1 n \|\Tilde{z}^k_0 - z^{k+1}_0\|^2\nonumber \\
      && - \gamma_1 n \|F(z^{k+1}_0) + z^{k+1}_0 -\Tilde{z}^k_0\|^2\nonumber \\
      &\geq& \|z_0^{k+1} - z^*\|^2 + \gamma_1 n \|\Tilde{z}^k_0 - z^{k+1}_0\|^2 \nonumber\\
      && - \gamma_1 n \|F(z^{k+1}_0) + z^{k+1}_0 -\Tilde{z}^k_0\|^2 \label{my2}
  \end{eqnarray}
  By combining \eqref{my1} with \eqref{my2} and rearranging the terms, we obtain 
  \begin{eqnarray}
      \|z_0^{k+1} - z^*\|^2 &\leq& \|z_0^k-z^*\|^2  - \|z_0^k-\Tilde{z}^k_0\|^2-(1+\gamma_1 n) \|z_0^{k+1} -\Tilde{z}^k_0\|^2 + \gamma_1 n \|F(z^{k+1}_0) + z^{k+1}_0 - \Tilde{z}^k_0\|^2\nonumber\\
      && +2\gamma_1 \langle \sum_{i=0}^{n-1} F_{\pi^k_i}(\Bar{z}_i^k), \Tilde{z}^k_0-z_0^{k+1}\rangle - 2n\gamma_2 \langle F(z_0^k), \Tilde{z}^k_0-z_0^{k+1}\rangle \nonumber \\
      && -2\gamma_1 \langle \sum_{i=0}^{n-1} F_{\pi^k_i}(\Bar{z}_i^k),\Tilde{z}^k_0-z^*\rangle+2 \gamma_1 n \langle F(\Tilde{z}^k_0), \Tilde{z}^k_0 - z^{k+1}_0\rangle\label{my3}
  \end{eqnarray}
  Now using $z^k_0 - \Tilde{z}^k_0 \stackrel{\eqref{z_tilde def}}{=} \gamma_2 n F(z^k_0)$, we get
  \begin{eqnarray}
      \|z_0^{k+1} - z^*\|^2 &\leq& \|z_0^k-z^*\|^2  - \gamma_2^2n^2 \|F(z_0^k)\|^2 - (1+\gamma_1 n) \|z_0^{k+1} -\Tilde{z}^k_0\|^2 \nonumber\\
      && +\underbrace{2\gamma_1 \langle \sum_{i=0}^{n-1} F_{\pi^k_i}(\Bar{z}_i^k), \Tilde{z}^k_0-z_0^{k+1}\rangle}_{T_1} - \underbrace{2n\gamma_2 \langle F(z_0^k), \Tilde{z}^k_0-z_0^{k+1}\rangle}_{T_2} \nonumber \\
      && -2\underbrace{\gamma_1 \langle \sum_{i=0}^{n-1} F_{\pi^k_i}(\Bar{z}_i^k),\Tilde{z}^k_0-z^*\rangle}_{T_3} + \underbrace{2 \gamma_1 n \langle F(\Tilde{z}^k_0), \Tilde{z}^k_0 - z^{k+1}_0\rangle}_{T_4} \nonumber \\ 
      && + \underbrace{\gamma_1 n \|F(z^{k+1}_0) + z^{k+1}_0 -\Tilde{z}^k_0\|^2}_{T_5}\label{my3_a}
  \end{eqnarray}
  At this stage, we aim to obtain bounds for the five terms $T_1, T_2, T_3, T_4, T_5$ on the right-hand side of \eqref{my3_a}. 
  
\begin{eqnarray}
  T_1 
  &\leq& \frac{\gamma_1}{L_{max}}\left\|\sum_{i=0}^{n-1} F_{\pi^k_i}(\Bar{z}_i^k)\right\|^2 + \gamma_1 L_{max}\|\Tilde{z}^k_0-z_0^{k+1}\|^2 \label{my4}
\end{eqnarray}

\begin{eqnarray}
    T_2 &\stackrel{\eqref{ineq inner product}}{=}& -n^2 \gamma_2^2 \| F(z_0^k)\|^2 - \|\Tilde{z}^k_0-z_0^{k+1}\|^2 + \|\gamma_2 n F(z_0^k)+\Tilde{z}^k_0-z_0^{k+1}\|^2 \nonumber\\
    &\stackrel{\eqref{z_tilde def}}{=}& -n^2 \gamma_2^2 \| F(z_0^k)\|^2 - \|\Tilde{z}^k_0-z_0^{k+1}\|^2 + \|z^k_0-z_0^{k+1}\|^2 \nonumber \\
    &\stackrel{\eqref{my0}}{=}& -n^2 \gamma_2^2 \| F(z_0^k)\|^2 - \|\Tilde{z}^k_0-z_0^{k+1}\|^2 + \gamma_1^2 \left\|\sum\limits_{i=0}^{n-1} F_{\pi^k_i}(\Bar{z}_i^k)\right\|^2 \label{my5}
\end{eqnarray}

\begin{eqnarray}
 T_3 &\leq& \frac{\gamma_1}{L_{max}}\left\|\sum_{i=0}^{n-1} F_{\pi^k_i}(\Bar{z}_i^k)\right\|^2 + \gamma_1 L_{max}\|\Tilde{z}^k_0-z^*\|^2 \nonumber \\
&\stackrel{\eqref{z_tilde def}}{=}& \frac{\gamma_1}{L_{max}}\left\|\sum_{i=0}^{n-1} F_{\pi^k_i}(\Bar{z}_i^k)\right\|^2 + \gamma_1 L_{max}\|z^k_0-\gamma_2 n F(z_0^k) -z^*\|^2 \nonumber\\
&\stackrel{\eqref{ineq1a}}{\leq}& \frac{\gamma_1}{L_{max}}\left\|\sum_{i=0}^{n-1} F_{\pi^k_i}(\Bar{z}_i^k)\right\|^2 +2 \gamma_1 L_{max}\|z^k_0-z^*\|^2 + 2\gamma_2^2 n^2\gamma_1 L_{max} \|F(z_0^k)\|^2 \label{my6}
\end{eqnarray}

\begin{eqnarray}
     T_4 &\stackrel{\eqref{ineq inner product}}{\leq}& \gamma_1 n \|F(\Tilde{z}^k_0)\|^2 + \gamma_1 n\|\Tilde{z}^k_0 - z^{k+1}_0\|^2 \nonumber\\
     &\stackrel{\eqref{DefLipschitz}}{\leq}& \gamma_1 n L^2 \|\Tilde{z}^k_0 - z^*\|^2 + \gamma_1 n\|\Tilde{z}^k_0 - z^{k+1}_0\|^2\nonumber\\
     &\stackrel{\eqref{z_tilde def}}{=}& \gamma_1 n L^2 \|z^k_0 -\gamma_2 nF(z^k_0) - z^*\|^2 + \gamma_1 n\|\Tilde{z}^k_0 - z^{k+1}_0\|^2\nonumber\\
     &\stackrel{\eqref{ineq1a}}{\leq}& 2 \gamma_1 n L^2 \|z_0^k - z^*\|^2 +2\gamma_1 n^3L^2 \gamma_2^2 \|F(z^k_0)\|^2 + \gamma_1 n\|\Tilde{z}^k_0 - z^{k+1}_0\|^2\label{myy}
\end{eqnarray}

\begin{eqnarray}
   T_5 &\stackrel{\eqref{z_tilde def}}{=}& \gamma_1 n \|F(z^{k+1}_0) - F(z^k_0) + F(z^k_0) +z^{k+1}_0 - z^k_0 + \gamma_2 n F(z^k_0)\|^2 \nonumber \\
   &\stackrel{\eqref{ineq1a}}{\leq}& 3 \gamma_1 n \|F(z^{k+1}_0) - F(z^k_0)\|^2 + 3 \gamma_1 n (\gamma_2 n + 1)^2 \|F(z^k_0)\|^2 +3\gamma_1n \|z^k_0 - z^{k+1}_0\|^2\nonumber\\
   &\stackrel{\eqref{DefLipschitz}}{\leq}& 3n\gamma_1 (1 + L^2) \|z^k_0 - z^{k+1}_0\|^2+3n \gamma_1 (\gamma_2 n + 1)^2 \|F(z^k_0)\|^2\nonumber\\
   &\stackrel{\eqref{my0}}{=}& 3n\gamma_1^3 (1 + L^2) \left\|\sum\limits_{i=0}^{n-1} F_{\pi^k_i}(\Bar{z}_i^k)\right\|^2+3n \gamma_1 (\gamma_2 n + 1)^2 \|F(z^k_0)\|^2 \label{my9}
\end{eqnarray}

Substituting the upper bounds \eqref{my4}, \eqref{my5}, \eqref{my6}, \eqref{myy}, \eqref{my9} into \eqref{my3_a}, we get 
\begin{eqnarray}
\|z_0^{k+1} - z^*\|^2 &\leq& \|z_0^k-z^*\|^2  - \gamma_2^2n^2 \|F(z_0^k)\|^2 - (1+\gamma_1 n) \|z_0^{k+1} -\Tilde{z}^k_0\|^2 \nonumber\\
      && +\frac{\gamma_1}{L_{max}} \left\|\sum_{i=0}^{n-1}F_{\pi^k_i}(\Bar{z}_i^k)\right\|^2 + \gamma_1 L_{max}\|\Tilde{z}^k_0-z_0^{k+1}\|^2\nonumber \\ 
      && -n^2 \gamma_2^2 \| F(z_0^k)\|^2 - \|\Tilde{z}^k_0-z_0^{k+1}\|^2 + \gamma_1^2 \left\|\sum\limits_{i=0}^{n-1} F_{\pi^k_i}(\Bar{z}_i^k)\right\|^2\nonumber \\
      && + \frac{\gamma_1}{L_{max}}\left\|\sum_{i=0}^{n-1} F_{\pi^k_i}(\Bar{z}_i^k)\right\|^2 +2 \gamma_1 L_{max}\|z^k_0-z^*\|^2 + 2\gamma_2^2 n^2\gamma_1 L_{max} \|F(z_0^k)\|^2 \nonumber \\
      && +2 \gamma_1 n L^2 \|z_0^k - z^*\|^2 +2\gamma_1 n^3L^2 \gamma_2^2 \|F(z^k_0)\|^2 + \gamma_1 n\|\Tilde{z}^k_0 - z^{k+1}_0\|^2\nonumber \\
      && +3n\gamma_1^3 (1 + L^2) \left\|\sum\limits_{i=0}^{n-1} F_{\pi^k_i}(\Bar{z}_i^k)\right\|^2+3n \gamma_1 (\gamma_2 n + 1)^2 \|F(z^k_0)\|^2\nonumber \\
      &=& (1+2 \gamma_1 n L^2+2 \gamma_1 L_{max}) \|z_0^k-z^*\|^2 \nonumber\\
      && - n\left[2\gamma_2^2n(1 - \gamma_1 n L^2-\gamma_1 L_{max}) - 3\gamma_1 \left(\gamma_2 n + 1\right)^2\right] \|F(z_0^k)\|^2\nonumber \\ 
      && - (2 - \gamma_1 L_{max}) \|z_0^{k+1} -\Tilde{z}^k_0\|^2 + \gamma_1\left[3n\gamma_1^2(1 + L^2)+\gamma_1+\frac{2}{L_{max}}\right]\left\|\sum\limits_{i=0}^{n-1} F_{\pi^k_i}(\Bar{z}_i^k)\right\|^2\label{last_eq_of_Lemma: my0}
\end{eqnarray}
\end{proof}

\begin{lemma}\label{Lemma: new iterates in epoch}
 Assume that each $F_i, i\in [n],$ is $L_i-$Lipschitz, the step sizes of \ref{eq:SEG-RR} satisfy $\gamma_1\leq \frac{1}{3\sqrt{2n(n-1)}L_{max}}$, $ \gamma_2 \leq \frac{1}{ L_{max}}$ and $z^* \in \mathcal{Z}_*$.
    The iterates of the \ref{eq:SEG-RR} algorithm satisfy the following bound 
    \begin{eqnarray}
        \Expep{\frac{1}{n} \sum\limits_{j=0}^{n-1} \norm{z_{j}^k - z_0^k}^2}&\stackrel{}{\le}& \gamma_1^2 \left[18n^2L^2 + A (24n^2-23n+1)\right]\norm{z_0^k - z^*}^2 + 2\gamma_1^2(24n^2-23n+1)\sigma_*^2 \quad \label{EG-bound123}
    \end{eqnarray}   
\end{lemma}
\begin{proof}
    Following the same steps as in Lemma \ref{Lemma: squared-norms-epoch-iterates}, we get from inequality \eqref{bound_on_G_k2} that for step sizes satisfying $\gamma_1 < \frac{1}{3\sqrt{n(n-1)}L_{max}}, \gamma_2 \leq \frac{1}{L_{max}}$ and for $z^* \in \mathcal{Z}_*,$ the following holds
  \begin{eqnarray}
     \Expep{\frac{1}{n} \sum\limits_{j=0}^{n-1} \norm{z_{j}^k - z_0^k}^2}&\leq&\frac{D\gamma_1^2}{D_1} \norm{F(z_0^k)}^2 + \frac{A\gamma_1^2}{D_1} \left[\frac{n+1}{2} + 12 \gamma_2^2 L_{max}^2n(n-1)\right] \norm{z_0^k - z^*}^2 \nonumber \\
    && +\frac{\gamma_1^2}{D_1} \left[(n+1) +24 \gamma_2^2L_{max}^2n(n-1)\right] \sigma_*^2 \label{bound_on_G_k123}
 \end{eqnarray}
 where $D = \frac{(1+8\gamma_2^2 L_{max}^2)(n-1)(2n-1)}{2}, D_1 = 1 - 9n(n-1)\gamma_1^2 L_{max}^2$. \\
 
 Selecting $\gamma_1 \leq \frac{1}{3L_{max}\sqrt{2n(n-1)}}$ we have that:
\begin{eqnarray}
\frac{1}{D_1} \leq 2 \label{bound-for-D_1-2nd-term1}
\end{eqnarray}
We, also, have that for $\gamma_2 \leq \frac{1}{L_{max}}$ we can upper bound $D$ as follows 
\begin{eqnarray}
  D \leq 9 n^2 \label{bound-for-D-new-2nd-t123}
\end{eqnarray}
Substituting the bounds \eqref{bound-for-D_1-2nd-term1}, \eqref{bound-for-D-new-2nd-t123} to \eqref{bound_on_G_k123}, we obtain
\begin{eqnarray}
    \Expep{\frac{1}{n} \sum\limits_{j=0}^{n-1} \norm{z_{j}^k - z_0^k}^2}&\stackrel{\eqref{bound-for-D_1-2nd-term1}, \eqref{bound-for-D-new-2nd-t123}}{\leq}& 18n^2 \gamma_1^2 \norm{F(z_0^k)}^2  + 2A\gamma_1^2 \left[\frac{n+1}{2} + 12 \gamma_2^2 L_{max}^2n(n-1)\right] \norm{z_0^k - z^*}^2 \nonumber \\
        && + 2\gamma_1^2 \left[(n+1) +24 \gamma_2^2L_{max}^2n(n-1)\right] \sigma_*^2 \label{Lemma: per epoch iterates without gamma2} \\
        &\stackrel{\gamma_2 \leq \frac{1}{L_{max}}}{\leq}& 18n^2\gamma_1^2 \norm{F(z_0^k)}^2 + A\gamma_1^2 (24n^2-23n+1) \norm{z_0^k - z^*}^2 \nonumber \\
        && + 2\gamma_1^2 (24n^2-23n+1)\sigma_*^2 \nonumber \\
        &\stackrel{\eqref{DefLipschitz}}{\leq}& \gamma_1^2 \left[18n^2L^2 + A (24n^2-23n+1)\right]\norm{z_0^k - z^*}^2 + 2\gamma_1^2(24n^2-23n+1)\sigma_*^2\nonumber
\end{eqnarray}
\end{proof}

\begin{lemma}\label{Lemma: new my}
    Assume that each $F_i, i\in [n],$ is $L_i-$Lipschitz, the step sizes of \ref{eq:SEG-RR} satisfy $\gamma_1\leq \frac{1}{3\sqrt{2}nL_{max}}$, $ \gamma_2 \leq \frac{1}{ L_{max}}$ and $z^* \in \mathcal{Z}_*$.
    The iterates of the \ref{eq:SEG-RR} algorithm satisfy the following bound
    \begin{eqnarray}
    \Expep{\gamma_1 C \left\|\sum_{i=0}^{n-1} F_{\pi^k_i}(\Bar{z}_i^k) \right\|^2} &\leq& 20n^2 \gamma_1 C(11L^2 + 15A) \norm{z_0^k - z^*}^2 + 4 n^2 C\gamma_1 \gamma_2^2 L^2 \|F(z_0^k)\|^2 \nonumber \\
    && + 12CL_{max}^2\gamma_1n^2 \left[2\gamma_1^2 (24n^2-23n+1) + \gamma_2^2\right] \sigma_*^2\nonumber
\end{eqnarray}
where $C = \frac{3(1 + L^2+L_{max})}{L^2_{max}}$. 
\end{lemma}
\begin{proof}
   For brevity, let $T = \gamma_1 C \left\|\sum_{i=0}^{n-1} F_{\pi^k_i}(\Bar{z}_i^k) \right\|^2$. Due to the finite-sum property in the definition of the VIP \eqref{VI-problem}, we have that for any permutation $\pi^k$ of $[n]$ it holds that 
   \begin{eqnarray}
       n F(\Hat{z}^k_0) = \sum_{i=0}^{n-1} F_{\pi^k_i}(\Hat{z}^k_0).
   \end{eqnarray}
   Thus, in order to upper bound $T$ we first add and subtract the term $n F(\Hat{z}^k_0) = \sum_{i=0}^{n-1} F_{\pi^k_i}(\Hat{z}^k_0)$ and use inequality \eqref{ineq1} to get
\begin{eqnarray}
    T &=& \gamma_1 C \left\|\sum_{i=0}^{n-1} \left(F_{\pi^k_i}(\Bar{z}_i^k) - F_{\pi^k_i}(\Hat{z}_0^k)\right) + n F(\Hat{z}^k_0)\right\|^2\nonumber \\
    &\stackrel{\eqref{ineq1}}{\leq}& 2\gamma_1C \left\|\sum_{i=0}^{n-1} \left(F_{\pi^k_i}(\Bar{z}_i^k) - F_{\pi^k_i}(\Hat{z}_0^k)\right)\right\|^2  + 2 n^2C\gamma_1 \|F(\Hat{z}^k_0)\|^2 \nonumber \\
    &\stackrel{\eqref{ineq1}}{\leq}& 2Cn\gamma_1 \sum_{i=0}^{n-1}\left\| F_{\pi^k_i}(\Bar{z}_i^k) - F_{\pi^k_i}(\Hat{z}_0^k)\right\|^2 + 2 n^2 C\gamma_1 \|F(\Hat{z}^k_0)\|^2 \nonumber \\
    &\stackrel{\eqref{DefLipschitz}}{\leq}& 2Cn\gamma_1\sum_{i=0}^{n-1}\left\| F_{\pi^k_i}(\Bar{z}_i^k) - F_{\pi^k_i}(\Hat{z}_0^k)\right\|^2 + 2 n^2 C\gamma_1 L^2 \|\Hat{z}^k_0 - z^*\|^2 \nonumber
\end{eqnarray}
Using the definition of $\Hat{z}^k_0 = z_0^k - \gamma_2 F(z^k_0)$, we get
\begin{eqnarray}
    T 
    &\leq& 2Cn\gamma_1\sum_{i=0}^{n-1}\left\| F_{\pi^k_i}(\Bar{z}_i^k) - F_{\pi^k_i}(\Hat{z}_0^k)\right\|^2  + 2 n^2 C\gamma_1 L^2 \|z_0^k - \gamma_2 F(z^k_0) - z^*\|^2 \nonumber\\
    &\stackrel{\eqref{ineq1}}{\leq}& 2Cn\gamma_1\sum_{i=0}^{n-1}\left\| F_{\pi^k_i}(\Bar{z}_i^k) - F_{\pi^k_i}(\Hat{z}_0^k)\right\|^2 + 4 n^2 C\gamma_1 L^2 \|z_0^k-z^*\|^2 + 4 n^2 C\gamma_1 \gamma_2^2 L^2 \|F(z_0^k)\|^2 \quad \quad \label{my23} 
\end{eqnarray}
Taking expectation condition on the filtration $\mathcal{F}_k$, we have 
\begin{eqnarray}
    \Expep{T} &\leq& 2C\gamma_1n \Expep{\sum_{i=0}^{n-1}\left\| F_{\pi^k_i}(\Bar{z}_i^k) - F_{\pi^k_i}(\Hat{z}_0^k)\right\|^2} + 4n^2 C\gamma_1 L^2 \|z_0^k-z^*\|^2  + 4 n^2 C\gamma_1 \gamma_2^2 L^2 \|F(z_0^k)\|^2 \quad \quad \label{my24}
\end{eqnarray}
It, now, remains to bound the sum on the right-hand side of \eqref{my24}. Using Lemma \ref{Lemma: Strongly-monotone-case1}, we get
\begin{eqnarray}
    \Expep{T} &\stackrel{\eqref{Lemma: Strongly-monotone-case1}}{\leq}& 12C\gamma_1n L_{max}^2 \underbrace{\Expep {\sum_{i=0}^{n-1} \left\|z^k_{i}-z^k_0\right\|^2}}_{T_1} +6C\gamma_1n L_{max}^2 \gamma_2^2 \underbrace{\Expep{\sum_{i=0}^{n-1} \left\|F_{\pi_i^k}(z^k_{0}) -F(z^k_{0}) \right\|^2}}_{T_2} \nonumber \\
    && + 4 n^2 C\gamma_1 L^2 \|z_0^k-z^*\|^2  + 4 n^2 C\gamma_1 \gamma_2^2 L^2 \|F(z_0^k)\|^2 \label{bv0}
\end{eqnarray}
Using Lemma \ref{Lemma: new iterates in epoch}, we can bound the term $T_1$ for $\gamma_1 \leq \frac{1}{3\sqrt{2n(n-1)}L_{max}}, \gamma_2 \leq \frac{1}{L_{max}}$ as follows
\begin{eqnarray}
    T_1 &\leq& n\gamma_1^2 \left[18n^2L^2 + A (24n^2-23n+1)\right]\norm{z_0^k - z^*}^2 + 2n\gamma_1^2(24n^2-23n+1)\sigma_*^2 \label{bv1}
\end{eqnarray}
From Proposition \ref{app: bound for const assumpt}, we can bound the term $T_2$ as follows
\begin{eqnarray}
    T_2 &\leq& n A \|z^k_0 - z^*\|^2 + 2n \sigma_*^2 \label{bv2}
\end{eqnarray}
Substituting \eqref{bv1}, \eqref{bv2} into \eqref{bv0}, we get
\begin{eqnarray}
    \Expep{T} &\leq& 12C\gamma_1^3n^2 L_{max}^2\left[18n^2L^2 + A (24n^2-23n+1)\right]\norm{z_0^k - z^*}^2 \nonumber \\ 
    && + 24C\gamma_1^3n^2 L_{max}^2 (24n^2-23n+1)\sigma_*^2 \nonumber \\ 
    && +6C\gamma_1n L_{max}^2 \gamma_2^2 \left[n A \|z^k_0 - z^*\|^2 + 2n \sigma_*^2\right] \nonumber\\
    && +4 n^2 C \gamma_1 L^2 \|z_0^k-z^*\|^2  + 4 n^2 C\gamma_1 \gamma_2^2 L^2 \|F(z_0^k)\|^2 \nonumber \\
    &=& 2n^2 \gamma_1 C \left\{2L^2+6\gamma_1^2 L_{max}^2\left[18n^2L^2 + A (24n^2-23n+1)\right]+3A L_{max}^2 \gamma_2^2\right\} \norm{z_0^k - z^*}^2 \nonumber \\
    && + 4 n^2 C\gamma_1 \gamma_2^2 L^2 \|F(z_0^k)\|^2 + 12CL_{max}^2\gamma_1n^2 \left[2\gamma_1^2 (24n^2-23n+1) + \gamma_2^2\right] \sigma_*^2 \label{le1}
\end{eqnarray}

To simplify the coefficient of the term $\norm{z_0^k - z^*}^2$ in the above inequality we can use the fact that 
$\gamma_1 \leq \frac{1}{3\sqrt{2}nL_{max}} \leq \frac{1}{nL_{max}}$ and $\gamma_2 \leq \frac{1}{L_{max}}$ and get the upper bound 
        \begin{eqnarray}
            \left\{2L^2+6\gamma_1^2 L_{max}^2\left[18n^2L^2 + A (24n^2-23n+1)\right]+3A L_{max}^2 \gamma_2^2\right\}
            &\leq& 2 L^2+6\left[18L^2 +24A\right]+3A\nonumber\\
            &\leq& 10 (11L^2 + 15A)\label{up1} 
        \end{eqnarray}
        
    Substituting \eqref{up1} into \eqref{le1}, we get 
\begin{eqnarray}
    \Expep{T} &\leq& 20n^2 \gamma_1 C(11L^2 + 15A) \norm{z_0^k - z^*}^2 + 4 n^2 C\gamma_1 \gamma_2^2 L^2 \|F(z_0^k)\|^2 \nonumber \\
    && + 12CL_{max}^2\gamma_1n^2 \left[2\gamma_1^2 (24n^2-23n+1) + \gamma_2^2\right] \sigma_*^2 \nonumber
\end{eqnarray}
\end{proof}

\subsubsection{Proof of Theorem \ref{theorem monotone my}} 
\label{theorem monotone app my}

We are now ready to prove Theorem \ref{theorem monotone my}. To do that, we are using Lemmas \ref{Lemma: my0} and \ref{Lemma: new my}.

\begin{proof}
Since the operator $F$ is monotone and each $F_i$ is $L_i$-Lipschitz, we have from Lemma \ref{Lemma: my0} that for $z^* \in \mathcal{Z}_*$ it holds 
\begin{eqnarray}
    \|z_0^{k+1} - z^*\|^2 &\leq& (1+2 \gamma_1 n L^2+2 \gamma_1 L_{max}) \|z_0^k-z^*\|^2  - n\left[2\gamma_2^2n(1 - \gamma_1 n L^2-\gamma_1 L_{max}) - 3\gamma_1 \left(\gamma_2 n + 1\right)^2\right] \|F(z_0^k)\|^2\nonumber \\ 
      && - (2 - \gamma_1 L_{max}) \|z_0^{k+1} -\Tilde{z}^k_0\|^2 + \gamma_1\left[3n\gamma_1^2(1 + L^2)+\gamma_1+\frac{2}{L_{max}}\right]\left\|\sum\limits_{i=0}^{n-1} F_{\pi^k_i}(\Bar{z}_i^k)\right\|^2\nonumber
\end{eqnarray}
Since $\gamma_1 \leq \gamma_{1, max} \leq \frac{2}{L_{max}}$, we have that $-(2 - \gamma_1 L_{max}) \leq 0$ and thus we get 
  \begin{eqnarray}
    \|z_0^{k+1} - z^*\|^2 &\leq& (1+2 \gamma_1 n L^2+2 \gamma_1 L_{max}) \|z_0^k-z^*\|^2 \nonumber \\
    && - n\left[2\gamma_2^2n(1 - \gamma_1 n L^2-\gamma_1 L_{max}) - 3\gamma_1 \left(\gamma_2 n + 1\right)^2\right] \|F(z_0^k)\|^2\nonumber \\ 
      && + \gamma_1\left[3n\gamma_1^2(1 + L^2)+\gamma_1+\frac{2}{L_{max}}\right]\left\|\sum\limits_{i=0}^{n-1} F_{\pi^k_i}(\Bar{z}_i^k)\right\|^2\label{my10}
  \end{eqnarray} 
 
  Note that for $\gamma_1 \leq \frac{1}{2(nL^2+L_{max})}$ the term $ -(1 - \gamma_1 n L^2-\gamma_1 L_{max}) \leq - \frac{1}{2}$. 
  To simplify the coefficient of the term $ \|F(z_0^k)\|^2$ in \eqref{my10}, 
  we will choose the stepsizes of the algorithm to satisfy $\gamma_1 \leq \min\left\{\frac{1}{2(nL^2+L_{max})}, \frac{\gamma_2^2n}{4\left(\frac{n}{L_{max}} + 1\right)^2}\right\}, \gamma_2 \leq \frac{1}{L_{max}}$ so that 
  \begin{eqnarray}
    - n\left[2\gamma_2^2n(1 - \gamma_1 n L^2-\gamma_1 L_{max}) - 3\gamma_1 \left(\gamma_2 n + 1\right)^2\right] &\leq& - n\left[\gamma_2^2n - 3\gamma_1 \left(\gamma_2 n + 1\right)^2\right]\nonumber \\ 
    &\stackrel{\gamma_2 \leq \frac{1}{L_{max}}}{\leq}& - n\left[\gamma_2^2n - 3\gamma_1 \left(\frac{n}{L_{max}} + 1\right)^2\right] \nonumber \\
    &\stackrel{\gamma_1 \leq \frac{\gamma_2^2n}{4\left(\frac{n}{L_{max}} + 1\right)^2}}{\leq}& -\frac{\gamma_2^2 n^2}{4} \nonumber 
  \end{eqnarray}
  
  Additionally, since $\gamma_1 \leq \gamma_{1, max} \leq \frac{1}{nL_{max}} \leq \frac{1}{L_{max}}$ we can simplify the coefficient of the term $\left\|\sum\limits_{i=0}^{n-1} F_{\pi^k_i}(\Bar{z}_i^k)\right\|^2$ in \eqref{my10}, as follows 
  \begin{eqnarray}
    3n\gamma_1^2(1 + L^2)+\gamma_1+\frac{2}{L_{max}} &\leq& \frac{3(1 + L^2)}{nL^2_{max}}+\frac{1}{L_{max}}+\frac{2}{L_{max}} \leq C \nonumber
  \end{eqnarray}
    where $C = \frac{3(1 + L^2+L_{max})}{L^2_{max}}$.
    
    Hence, we get from inequality \eqref{my10} that 
  \begin{eqnarray}
      \|z_0^{k+1} - z^*\|^2 &\stackrel{}{\leq}& (1+2 \gamma_1 n L^2+2 \gamma_1 L_{max}) \|z_0^k-z^*\|^2  -\frac{\gamma_2^2 n^2}{4} \|F(z_0^k)\|^2 + \gamma_1 C \left\|\sum\limits_{i=0}^{n-1} F_{\pi^k_i}(\Bar{z}_i^k)\right\|^2\nonumber
  \end{eqnarray}

  Taking expectation condition on the filtration $\mathcal{F}_k$, we have that 
    \begin{eqnarray}
        \Expep{\|z_0^{k+1} - z^*\|^2} &\leq& (1+2 \gamma_1 n L^2+2 \gamma_1 L_{max}) \|z_0^k-z^*\|^2  -\frac{\gamma_2^2 n^2}{4} \|F(z_0^k)\|^2 +\underbrace{\gamma_1C\Expep{\left\|\sum\limits_{i=0}^{n-1} F_{\pi^k_i}(\Bar{z}_i^k)\right\|^2}}_{T}\quad \quad \label{my244}
    \end{eqnarray}
    
    From Lemma \ref{Lemma: new my}, we have for $\gamma_1 \leq \frac{1}{3\sqrt{2}nL_{max}}$, $ \gamma_2 \leq \frac{1}{L_{max}}$ that the following holds
      \begin{eqnarray}
          T &\leq& 20n^2 \gamma_1 C(11L^2 + 15A) \norm{z_0^k - z^*}^2 + 4 n^2 C\gamma_1 \gamma_2^2 L^2 \|F(z_0^k)\|^2 \nonumber \\
        && + 12CL_{max}^2\gamma_1n^2 \left[2\gamma_1^2 (24n^2-23n+1) + \gamma_2^2\right] \sigma_*^2 \label{my11}
      \end{eqnarray}
  
      Substituting \eqref{my11} into \eqref{my244} and rearranging, we obtain 
      \begin{eqnarray}
         \Expep{\|z_0^{k+1} - z^*\|^2} &\leq& \left\{1 + 2 \gamma_1 [nL^2+L_{max}+10n^2C(11L^2 + 15A)]\right\} \|z_0^k-z^*\|^2  \nonumber \\
         &&-\left(\frac{1}{4}- 4C\gamma_1L^2\right) \gamma_2^2 n^2 \|F(z_0^k)\|^2\nonumber\\
        && + 12CL_{max}^2\gamma_1n^2 \left[2\gamma_1^2 (24n^2-23n+1) + \gamma_2^2\right] \sigma_*^2 \nonumber \\
            &\stackrel{\gamma_1 \leq \frac{1}{32CL^2}}{\leq}& \left\{1 + 2 \gamma_1 [nL^2+L_{max}+10n^2C(11L^2 + 15A)]\right\} \|z_0^k-z^*\|^2  -\frac{\gamma_2^2n^2}{8}\|F(z_0^k)\|^2\nonumber\\
        && + 12CL_{max}^2\gamma_1n^2 \left[2\gamma_1^2 (24n^2-23n+1) + \gamma_2^2\right] \sigma_*^2 \nonumber
    \end{eqnarray}
    where at the last step we used that $\gamma_1 \leq \min\left\{\frac{\gamma_2^2n}{4\left(\frac{n}{L_{max}} + 1\right)^2}, \frac{1}{3\sqrt{2}nL_{max}}, \frac{1}{2nK [nL^2+L_{max}+10n^2C(11L^2 + 15A)}\right\} \leq \frac{1}{32CL^2}$.
    Letting $\alpha = 1 + 2 \gamma_1 [nL^2+L_{max}+10n^2C(11L^2 + 15A)]$ and rearranging the terms, we get
    \begin{eqnarray}
        \frac{\gamma_2^2n^2}{8} \|F(z_0^k)\|^2 &\stackrel{}{\leq}& \alpha \|z_0^k-z^*\|^2 - \Expep{\|z_0^{k+1} - z^*\|^2} + 12CL_{max}^2\gamma_1n^2 \left[2\gamma_1^2 (24n^2-23n+1) + \gamma_2^2\right] \sigma_*^2\nonumber
    \end{eqnarray}
    
Taking expectation on both sides and using the tower property of expectation, we have 
\begin{eqnarray}
    \Expe{\|F(z_0^k)\|^2} &\stackrel{}{\leq}& \frac{8\Expe{\alpha \|z_0^k-z^*\|^2 - \|z_0^{k+1} - z^*\|^2}}{\gamma_2^2n^2} + \frac{96CL_{max}^2\gamma_1 \left[2\gamma_1^2 (24n^2-23n+1) + \gamma_2^2\right] \sigma_*^2}{\gamma_2^2} \nonumber
\end{eqnarray}

Let $G_k = \left(\frac{1}{\alpha}\right)^k, \forall k \geq 0$. 
Multiplying both sides with $G_{k+1}$ and using the fact that $G_{k+1} \cdot \alpha = G_{k}$, we get
\begin{eqnarray}
  G_{k+1} \Expe{\|F(z_0^k)\|^2} &\stackrel{}{\leq}& \frac{8\Expe{G_k \|z_0^k-z^*\|^2 - G_{k+1} \|z_0^{k+1} - z^*\|^2}}{\gamma_2^2n^2} \nonumber \\
  && + \frac{96CL_{max}^2\gamma_1 \left[2\gamma_1^2 (24n^2-23n+1) + \gamma_2^2\right] \sigma_*^2}{\gamma_2^2} G_{k+1}\nonumber  
\end{eqnarray}

Summing for $k = 0, ..., K-1$ and multiplying by $\frac{1}{\sum_{j=0}^{K-1}G_{j+1}}$ we have that 
\begin{eqnarray}
  \sum\limits_{k=0}^{K-1} \frac{G_{k+1}}{\sum_{j=0}^{K-1}G_{j+1}} \Expe{\|F(z_0^k)\|^2} &\stackrel{}{\leq}& \frac{8 \sum_{k=0}^{K-1} \Expe{G_k \|z_0^k-z^*\|^2 - G_{k+1} \|z_0^{k+1} - z^*\|^2}}{\gamma_2^2n^2\sum_{j=0}^{K-1}G_{j+1}} \nonumber \\
&& + \frac{96CL_{max}^2\gamma_1 \left[2\gamma_1^2 (24n^2-23n+1) + \gamma_2^2\right] \sigma_*^2}{\gamma_2^2}\label{rih}
\end{eqnarray}

From the telescopic sum in the right hand-side of \eqref{rih}, we get 
\begin{eqnarray}
  \sum\limits_{k=0}^{K-1} \frac{G_{k+1}}{\sum_{j=0}^{K-1}G_{j+1}} \Expe{\|F(z_0^k)\|^2} &\leq& \frac{8 \|z_0-z^*\|^2}{\gamma_2^2n^2\sum_{j=0}^{K-1}G_{j+1}} + \frac{96CL_{max}^2\gamma_1 \left[2\gamma_1^2 (24n^2-23n+1) + \gamma_2^2\right] \sigma_*^2}{\gamma_2^2}\nonumber
\end{eqnarray}

Let $w_{k} = \frac{G_{k+1}}{\sum_{j=0}^{K-1}G_{j+1}}$. We have that
\begin{eqnarray}
    \sum_{k=0}^{K-1} w_k\Expe{\| F(z_0^k)\|^2}&\leq& \frac{8 \|z_0-z^*\|^2}{\gamma_2^2n^2\sum_{j=0}^{K-1}G_{j+1}} + \frac{96CL_{max}^2\gamma_1 \left[2\gamma_1^2 (24n^2-23n+1) + \gamma_2^2\right] \sigma_*^2}{\gamma_2^2} \label{ineq_e_my13}
\end{eqnarray}
We, next, bound the sum $\sum_{j=0}^{K-1}G_{j+1}$ by using the fact that the sequence $G_k$ is decreasing as follows
\begin{eqnarray}
    \sum_{j=0}^{K-1}G_{j+1} \geq \sum_{j=0}^{K-1} G_{K} = K G_{K} = \frac{K}{\alpha^K}
  \Rightarrow \frac{1}{\sum_{j=0}^{K-1}G_{j+1}} \leq \frac{\alpha^K}{K} \label{my14}
\end{eqnarray}
Substituting \eqref{my14} in \eqref{ineq_e_my13}, we get 
\begin{eqnarray}
  \sum_{k=0}^{K-1} w_k\Expe{\| F(z_0^k)\|^2} &\leq& \frac{8 \alpha^K \|z_0-z^*\|^2}{\gamma_2^2n^2K}+ \frac{96CL_{max}^2\gamma_1 \left[2\gamma_1^2 (24n^2-23n+1) + \gamma_2^2\right] \sigma_*^2}{\gamma_2^2} \label{my15}
\end{eqnarray}
Using inequality $1+x \leq e^{x}$, we have that
\begin{eqnarray}
    \alpha^K &=& \left\{1 + 2 \gamma_1 [nL^2+L_{max}+10n^2C(11L^2 + 15A)]\right\}^K \leq  e^{2 K \gamma_1 [nL^2+L_{max}+10n^2C(11L^2 + 15A)]} \label{my16}
\end{eqnarray}
 
For $\gamma_1\leq \frac{1}{2nK [nL^2+L_{max}+10n^2C(11L^2 + 15A)]}$, we have that $2 K \gamma_1 [nL^2+L_{max}+10n^2C(11L^2 + 15A)] \leq 1$ and thus \eqref{my16} gives 
\begin{eqnarray}
     \alpha^K \leq e^{2 K \gamma_1 [nL^2+L_{max}+10n^2C(11L^2 + 15A)]} \leq e^1 \leq 3 \nonumber 
\end{eqnarray}

Substituting the bound of $\alpha^K$ into \eqref{my15}, we get that for stepsizes $\gamma_1 \leq \gamma_{1, max}, \gamma_2 \leq \frac{1}{L_{max}}$ it holds 
\begin{eqnarray}
    \sum_{k=0}^{K-1} w_k\Expe{\| F(z_0^k)\|^2} &\leq& \frac{24 \|z_0 -z^*\|^2}{\gamma_2^2n^2K} +\frac{96CL_{max}^2\gamma_1 \left[2\gamma_1^2 (24n^2-23n+1) + \gamma_2^2\right] \sigma_*^2}{\gamma_2^2} \label{res-monotone-ineq}
\end{eqnarray}
where $\gamma_{1, max} = \min\left\{\frac{\gamma_2^2n}{4\left(\frac{n}{L_{max}} + 1\right)^2}, \frac{1}{3\sqrt{2}nL_{max}}, \frac{1}{2nK [nL^2+L_{max}+10n^2C(11L^2 + 15A)}\right\}$. \\

\textbf{Proof of Inequality \eqref{2nd-pt-monotone_my}.} Let $D_2 = nL^2+L_{max}+10n^2C(11L^2 + 15A)$ for brevity. Substituting $\gamma_2 = \frac{1}{L_{max} n^{\frac{3}{4}} K^{\frac{1}{4}}}$ and the bound $\gamma_1 = \gamma_{1, max} \leq \frac{1}{2nKD_2}$ in \eqref{res-monotone-ineq}, we have that 
\begin{eqnarray}
   \sum_{k=0}^{K-1} w_k\Expe{\| F(z_0^k)\|^2} &\leq& \frac{24L_{max}^2 \|z_0 -z^*\|^2}{\sqrt{nK}} + \frac{48\sqrt{n}CL_{max}^4\sigma_*^2}{D_2 \sqrt{K}}  \left[\frac{(24n^2-23n+1)}{2D_2^2 n^2K^2}+ \frac{1}{L^2_{max} n^{\frac{3}{2}}\sqrt{K}}\right]\nonumber \\
   &=& \frac{24L_{max}^2 \|z_0 -z^*\|^2}{\sqrt{nK}} + \frac{24CL_{max}^4\sqrt{n}(24n^2-23n+1)\sigma_*^2}{D_2^3 n^2K^{\frac{5}{2}}}+ \frac{48CL_{max}^2\sigma_*^2}{D_2nK}\label{res-with-epochs} 
\end{eqnarray}
\end{proof}

\subsubsection{Iteration Complexity of SEG-RR}
\label{sec: iter_comp_seg_rr}
According to Theorem \ref{theorem monotone my}, if the stepsizes of the \ref{eq:SEG-RR} algorithm satisfy $\gamma_2 = \frac{1}{L_{max} n^{\frac{3}{4}} K^{\frac{1}{4}}}$ and the update step size is $\gamma_1 = \min\left\{\frac{\gamma_2^2n}{4\left(\frac{n}{L_{max}} + 1\right)^2}, \frac{1}{3\sqrt{2}nL_{max}}, \frac{1}{2nK [nL^2+L_{max}+10n^2C(11L^2 + 15A)}\right\}$ then it holds that
\begin{eqnarray}
    \sum_{k=0}^{K-1} w_k\Expe{\| F(z_0^k)\|^2} &\leq& \frac{24L_{max}^2 \|z_0 -z^*\|^2}{\sqrt{nK}} + \frac{24CL_{max}^4\sqrt{n}(24n^2-23n+1)\sigma_*^2}{D_2^3 n^2K^{\frac{5}{2}}}+ \frac{48CL_{max}^2\sigma_*^2}{D_2nK} \label{res-with-epoch2}
\end{eqnarray}
where $C = \frac{3(1 + L^2+L_{max})}{L^2_{max}}$ and $D_2 = nL^2+L_{max}+10n^2C(11L^2 + 15A)$.

We will prove that after some number of epochs the first term will dominate in \eqref{res-with-epoch2} and the rate of convergence will be $\mathcal{O}\left(\frac{1}{\sqrt{nK}}\right)$. To show this convergence rate, we, first, bound the term $D_2$ to simplify the expression in \eqref{res-with-epoch2} and then show that one can find a large enough constant $c > 0$ such that after a certain number of epochs the aforementioned rate is achieved. 

We, first, lower bound the term $D_2$ by the term $\Tilde{D} \triangleq L^2+L_{max}+\frac{30(1 + L^2+L_{max})}{L^2_{max}}(11L^2 + 15A)$, as follows
\begin{eqnarray}
    D_2 &=& nL^2+L_{max}+10n^2C(11L^2 + 15A) \nonumber \\ 
    &=& nL^2+L_{max}+\frac{30n^2(1 + L^2+L_{max})}{L^2_{max}} (11L^2 + 15A) \nonumber \\ 
    &\stackrel{}{\geq}& L^2+L_{max}+\frac{30(1 + L^2+L_{max})}{L^2_{max}}(11L^2 + 15A) \nonumber \\ 
    &\triangleq& \Tilde{D} \label{up_D}
\end{eqnarray}
Substituting the lower bound of $D_2$ into \eqref{res-with-epoch2}, we get 
\begin{eqnarray}
    \sum_{k=0}^{K-1} w_k\Expe{\| F(z_0^k)\|^2} &\leq& \frac{24L_{max}^2 \|z_0 -z^*\|^2}{\sqrt{nK}} + \frac{24CL_{max}^4\sqrt{n}(24n^2-23n+1)\sigma_*^2}{\Tilde{D}^3 n^2K^{\frac{5}{2}}}  + \frac{48CL_{max}^2\sigma_*^2}{\Tilde{D} nK}\label{ipe1}
\end{eqnarray}
Note that the constants $C, \Tilde{D}$ appearing in inequality \eqref{ipe1} do not depend on $n$ or $K$. Thus, one can find a large enough constant $c > 0$, i.e. $c = 24L_{max}^2 \max\left\{\|z_0 -z^*\|^2, \frac{24CL_{max}^2\sigma_*^2}{\Tilde{D}^3}, \frac{2C\sigma_*^2}{\Tilde{D}}\right\}$, such that from \eqref{ipe1} it holds
\begin{eqnarray}
    \sum_{k=0}^{K-1} w_k\Expe{\| F(z_0^k)\|^2} &\leq& \frac{c}{\sqrt{nK}} + \frac{c\sqrt{n}}{K^{\frac{5}{2}}}+ \frac{c}{nK} \leq c \left(\frac{2}{\sqrt{nK}} +\frac{\sqrt{n}}{K^{\frac{5}{2}}}\right) \nonumber
\end{eqnarray}
and find the number of epochs $K$ such that the following holds
\begin{eqnarray}
    c \left(\frac{2}{\sqrt{nK}} +\frac{\sqrt{n}}{K^{\frac{5}{2}}}\right) \leq \frac{3c}{\sqrt{nK}} &\iff& K\geq \sqrt{n} \label{com-result-monot}
\end{eqnarray}

Hence, after $K \geq \mathcal{O}\left(\sqrt{n}\right)$, we have that $\sum_{k=0}^{K-1} w_k\Expe{\| F(z_0^k)\|^2} \leq \frac{3c}{\sqrt{nK}} = \mathcal{O}\left(\frac{1}{\sqrt{nK}}\right)$.

\section{Further Convergence Guarantees}
\label{Appendix_Further}
In this section, we provide theoretical guarantees for SEG-SO and IEG as well as suggest a switching stepsize rule for SEG-RR. The use of the switching stepsize rule allows us to establish for SEG-RR a $\mathcal{O}\left(\frac{1}{k}\right)$ convergence to the exact solution $z^*$. We highlight, also, that the proposed stepsize schedule suggests when one should switch from a constant to a decreasing stepsize regime and is to the best of our knowledge the first time used in without-replacement sampling algorithms. 
\subsection{Other Variants of Without-replacement Sampling}\label{app SOandIG}
We start by showing how the proofs for SEG-RR can be modified in order to obtain convergence guarantees for two other variants of without-replacement sampling, namely the Shuffle Once (SO) sampling and the Incremental ExtraGradient (IEG).

The Shuffle Once variant samples at the first epoch of the algorithm a permutation $\pi$ of the dataset and then runs SEG using one data point at each iteration of the stochastic algorithm. The data point used in the $i$-th iteration is $\pi_i$, namely the $i$-th element of the permutation $\pi$.
Thus, the proofs for SEG-RR in all three regimes hold also for SEG-SO if we let $\pi^k = \pi$ for all $k \geq 0$. In this way, we are able to recover convergence guarantees for SEG-SO variant in strongly monotone, affine and monotone settings.

Regarding the Incremental ExtraGradient (IEG) variant, one can identify more easily the differences with random reshuffling in the pseudocode of Algorithm~\ref{algo EG IG}. Specifically, IEG does not sample any permutation of the dataset and instead regards the data samples in the order that were initially given in the dataset. Thus, the main modification in the proof of SEG-RR to get convergence guarantees for IEG is that one cannot use Lemma \ref{prop:random_reshuffling} and instead needs to use Lemma \ref{prop: bound prr epoch deviations} for bounding the distance of stochastic oracles $F_i$ from the operator $F$. Additionally, we observe that since the permutation $\pi=[n]$ (the initial order of the dataset) is fixed for all epochs $k\geq0$, there is no randomness involved in the selection of the data points at each epoch and hence any term appearing in conditional expectation in the proofs of SEG-RR will be equal to the same term without the expectation in the analysis of IEG. 

So far, we have explained how the proofs for SEG-RR in all three regimes can be modified to obtain convergence guarantees for SEG-SO and IEG. For illustration purposes, we provide in Sections \ref{app SEG_so strongly monotone}, \ref{app IEG strongly monotone} the proof for the strongly monotone case for SEG-SO and IEG, highlighting the differences with the proof of SEG-RR. Lastly, we note that our results for the switching stepsize rule in SEG-RR from Section \ref{sec: switch} can be also extended to the SEG-SO and IEG algorithms.
\subsubsection{SEG-SO}
\label{app SEG_so strongly monotone}
\vspace{+0.1cm}

\begin{corollary}
Suppose that the operator $F$ is $\mu$-strongly monotone and each $F_i, \, \forall i\in[n]$ is $L_i-$Lipschitz.
\begin{enumerate}[leftmargin=*]
\setlength{\itemsep}{0pt}
    \item Then the iterates of SEG-SO with constant step size $\gamma_2 = 2 \gamma_1$, $\gamma_1 \leq \frac{\mu}{10L_{max}^2\sqrt{10n^2+2n+54}}$ satisfy:
\begin{eqnarray}
    \Expe{\|z_0^{k} - z^*\|^2 } \leq\left(1 - \frac{\gamma_1 n\mu}{4} \right)^k {\|z_0- z^*\|^2} +\frac{96L_{max}^2}{\mu^2}\left[(25+n) \gamma_1^2+\gamma_2^2\right] \sigma_*^2 \nonumber 
\end{eqnarray}
    \item Let $K$ be the total number of epochs the SEG-SO is run.\\
    For step size $\gamma_2 = 2 \gamma_1$, $\gamma_1 = \min\left\{\frac{\mu}{10L_{max}^2\sqrt{10n^2+2n+54}}, \frac{4\log (n^{1/2}K)}{\mu nK}\right\}$, the following holds:
\begin{eqnarray}
   \Expe{\|z_0^{K} - z^*\|^2 } &=& \Tilde{\mathcal{O}}\left( e^{-\frac{\mu^2K}{L_{max}^2}} + \frac{1}{{nK^2}}\right) \nonumber
\end{eqnarray}
\end{enumerate}
\end{corollary}
\begin{proof}
Let $\pi$ be the permutation that is chosen at the start of the SEG-SO algorithm. By applying Theorem \ref{thm SC-SC4} and letting the permutation $\pi^k = \pi$ for all epochs $k \geq 0$ one can observe that the algorithm run is essentially SEG-SO. Thus, the results follow immediately.
\end{proof}

\newpage
\subsubsection{IEG}
\label{app IEG strongly monotone}
\vspace{+0.1cm}
\begin{lemma}\label{Lemma: squared-norms-epoch-iterates-IEG}
    Assume that each $F_i, i\in [n]$ is $L_i-$Lipchitz and the step size of IEG satisfy $\gamma_1\leq \frac{1}{3\sqrt{2n(n-1)}L_{max}}$, $ \gamma_2 \leq \frac{1}{\sqrt{n (n-1)} L_{max}}$.
    The iterates of the IEG algorithm satisfy the following bound 
    \begin{eqnarray}
        \frac{1}{n} \sum\limits_{j=0}^{n-1} \norm{z_{j}^k - z_0^k}^2 &\stackrel{}{\le}& [10n^2L^2 + 27(n-1)A] \gamma_1^2\norm{z_0^k - z^*}^2 + 66n(n-1)\gamma_1^2\sigma_*^2\nonumber
    \end{eqnarray}
\end{lemma}
\begin{proof}
    The proof of the Lemma \ref{Lemma: squared-norms-epoch-iterates-IEG} follows exactly the proof of Lemma \ref{Lemma: squared-norms-epoch-iterates} until inequality \eqref{eq:4} with the only difference that the expectation of any quantity is substituted with the quantity inside the expectation. Hence, from inequality \eqref{eq:4} we have that
    \begin{eqnarray}
        \norm{z_i^k - z_0^k}^2 &\leq& 6\gamma_1^2 L_{max}^2i (1+2\gamma_2^2 L_{max}^2)  nG_k + 3\gamma_1^2 (i^2+8 \gamma_2^2 i^2L_{max}^2)\norm{F(z_0^k)}^2 \nonumber \\ 
        && +24 \gamma_1^2\gamma_2^2L_{max}^2 A ni \norm{z_0^k - z^*}^2 + 48 \gamma_1^2\gamma_2^2L_{max}^2ni \sigma_*^2\nonumber \\ 
        && + 3\gamma_1^2 i^2 \norm{\frac{1}{i}\sum\limits_{j=0}^{i-1} F_{\pi_j^k} (z_0^k)-F (z_0^k)}^2\nonumber
    \end{eqnarray}
    where $G_k = \frac{1}{n} \sum\limits_{j=0}^{n-1} \norm{z_{j}^k - z_0^k}^2$.\\ 
The only change occurs in applying Proposition \ref{app: bound for const assumpt} to bound the last term, instead of Lemma \ref{prop: bound prr epoch deviations}. Applying inequality \eqref{ineq1} and Proposition \ref{app: bound for const assumpt}, we obtain
\begin{eqnarray}
    \norm{z_i^k - z_0^k}^2 &\stackrel{\eqref{ineq1}, \eqref{app: bound for const assumpt}}{\leq}&
     6\gamma_1^2 L_{max}^2i (1+2\gamma_2^2 L_{max}^2)  nG_k + 3\gamma_1^2 (i^2+8 \gamma_2^2 i^2L_{max}^2)\norm{F(z_0^k)}^2 \nonumber \\ 
        && +3(8 \gamma_2^2L_{max}^2 +1) niA \gamma_1^2 \norm{z_0^k - z^*}^2 + 6\gamma_1^2ni (8 \gamma_2^2L_{max}^2+3)\sigma_*^2\nonumber
\end{eqnarray}
Summing over $0\le i \le n-1$ and multiplying with $\frac{1}{n}$, we get: 
	\begin{eqnarray}
	G_k &\le& 3\gamma_1^2 L_{max}^2 (1+2\gamma_2^2 L_{max}^2) n(n-1) G_k +\gamma_1^2 D \norm{F(z_0^k)}^2 \nonumber \\ 
        && + \frac{3(8 \gamma_2^2L_{max}^2 +1)(n-1)A\gamma_1^2}{2}\norm{z_0^k - z^*}^2 + 3\gamma_1^2n(n-1) (8 \gamma_2^2L_{max}^2+3)\sigma_*^2 \nonumber
	\end{eqnarray}
	where we used the fact $\frac{1}{n} \sum_{i=0}^{n-1} i = \frac{n-1}{2}\nonumber$
        and let also $D = \left[\frac{(1+8\gamma_2^2 L_{max}^2)(n-1)(2n-1)}{2}\right]$ for brevity.\\
    Rearranging the terms, letting $D_1 = [1 - 3 n(n-1)(1+2\gamma_2^2 L_{max}^2)\gamma_1^2]$ and selecting the update stepsize $\gamma_1 < \frac{1}{\sqrt{3(1+2\gamma_2^2 L_{max}^2)n(n-1)}L_{max}}$, we have that
    \begin{eqnarray}
        G_k &\le& \gamma_1^2 \frac{D}{D_1} \norm{F(z_0^k)}^2 + \frac{3(8 \gamma_2^2L_{max}^2 +1)(n-1)A\gamma_1^2}{2D_1}\norm{z_0^k - z^*}^2 + \frac{3\gamma_1^2n(n-1) (8 \gamma_2^2L_{max}^2+3)\sigma_*^2}{D_1}\nonumber
    \end{eqnarray}
    Selecting $\gamma_1 \leq \frac{1}{3\sqrt{2n(n-1)}L_{max}}, \gamma_2 \leq \frac{1}{\sqrt{n(n-1)}L_{max}}$ and using inequalities \eqref{bound-for-D_1-2nd-term}, \eqref{bound-for-D-new-2nd-t}, we get
    \begin{eqnarray}
    G_k &\le& 10n^2\gamma_1^2 \norm{F(z_0^k)}^2 + 27(n-1)A\gamma_1^2\norm{z_0^k - z^*}^2 + 66n(n-1)\gamma_1^2\sigma_*^2\nonumber
    \end{eqnarray}
    Lastly, from the Lipschitz property of $F$, we obtain  
    \begin{eqnarray}
        G_k &\leq& [10n^2L^2 + 27(n-1)A] \gamma_1^2\norm{z_0^k - z^*}^2 + 66n(n-1)\gamma_1^2\sigma_*^2\nonumber
    \end{eqnarray}
\end{proof}
\newpage
\begin{corollary}
Suppose that the operator $F$ is $\mu$-strongly monotone and each $F_i, \, \forall i\in[n]$ is $L_i-$Lipschitz.
\begin{enumerate}[leftmargin=*]
\setlength{\itemsep}{0pt}
    \item Then the iterates of IEG with constant step size $\gamma_2 = 2 \gamma_1$, $\gamma_1 \leq \frac{\mu}{10L_{max}^2\sqrt{10n^2+n+29}}$ satisfy:
\begin{eqnarray}
    \|z_0^{k} - z^*\|^2 \leq\left(1 - \frac{\gamma_1 n\mu}{4} \right)^k {\|z_0- z^*\|^2} +\frac{48L_{max}^2}{\mu^2} \left[6n(n-1)\gamma_1^2+\gamma_2^2 \right]\sigma_*^2 \nonumber 
\end{eqnarray}
    \item Let $K$ be the total number of epochs the IEG is run.\\
    For step size $\gamma_2 = 2 \gamma_1$, $\gamma_1 = \min\left\{\frac{\mu}{10L_{max}^2\sqrt{10n^2+2n+29}}, \frac{4\log (n^{1/2}K)}{\mu nK}\right\}$, the following holds:
\begin{eqnarray}
  \|z_0^{K} - z^*\|^2  &=& \Tilde{\mathcal{O}}\left( e^{-\frac{K\mu^2}{L_{max}^2}} + \frac{1}{{K^2}}\right)\nonumber
\end{eqnarray}
\end{enumerate}
\end{corollary}
\begin{proof}
In IEG the data points are sampled according to the initial order in the dataset and thus $\pi = [n]$. A change to be noted in the proofs for IEG is that the algorithm does involve any stochasticity, as the permutation $\pi^k = [n]$ is fixed at each epoch $k \geq 0$. As a result, any term appearing inside expectation in the proof of SEG-RR will be deterministic in IEG and thus there is no necessity for expected values in the proofs of IEG.

By applying Theorem \ref{thm SC-SC4} and letting the permutation $\pi^k = [n]$ for all epochs $k \geq 0$ one can observe that the algorithm run is essentially IEG. 
The only difference with the proof of Theorem \ref{thm SC-SC4} is that in inequality \eqref{ineq_with_3terms}, Lemma \ref{Lemma: squared-norms-epoch-iterates-IEG} will be used instead of Lemma \ref{prop: bound prr epoch deviations} for bounding the term $T_2$. This will give the following upper bound 
\begin{eqnarray}
   \|z_0^{k+1} - z^*\|^2 &\leq& \left(1 - \frac{1}{2} \gamma_1 n \mu + \frac{U}{1 - \frac{\gamma_1 n \mu}{2}} + \frac{6nCL_{max}^2\gamma_1}{\mu}\right) \|z_{0}^k - z^*\|^2   + \frac{12nL_{max}^2\gamma_1}{\mu}\left[ 6n(n-1)\gamma_1^2+\gamma_2^2 \right]\sigma_*^2 \nonumber
\end{eqnarray}
where $C_1 = 2 \left[10n^2L^2 + 27(n-1)A\right] \gamma_1^2 +A \gamma_2^2$.
Selecting $\gamma_2 = 2 \gamma_1, \gamma_1 \leq \frac{\mu}{10L_{max}^2\sqrt{10n^2+n+29}}$, we get that \\
$\left(1 - \frac{1}{2} \gamma_1 n \mu + \frac{U}{1 - \frac{\gamma_1 n \mu}{2}} + \frac{6nCL_{max}^2\gamma_1}{\mu}\right) \leq (1 - \frac{1}{4}\gamma_1 n \mu)$ and thus
\begin{eqnarray}
 \|z_0^{k+1} - z^*\|^2 &\leq& \left(1 - \frac{1}{4} \gamma_1 n \mu\right) \|z_{0}^k - z^*\|^2  + \frac{12nL_{max}^2\gamma_1}{\mu}\left[ 6n(n-1)\gamma_1^2+\gamma_2^2 \right]\sigma_*^2 \nonumber   
\end{eqnarray}
Unrolling the recursion, we get
\begin{eqnarray}
 \|z_0^{k+1} - z^*\|^2 &\leq&  \left(1 - \frac{\gamma_1 n\mu}{4} \right)^{k+1} {\|z_0- z^*\|^2} +\frac{48L_{max}^2}{\mu^2} \left[6n(n-1)\gamma_1^2+\gamma_2^2 \right]\label{eq-ieg4}   
\end{eqnarray}
\textbf{Proof of 2nd point} Substituting the stepsize $\gamma_1 = \min\left\{\frac{\mu}{10L_{max}^2\sqrt{10n^2+2n+29}}, \frac{4\log (n^{1/2}K)}{\mu nK}\right\} \leq \frac{4\log (n^{1/2}K)}{\mu nK}$ into \eqref{eq-ieg4}, we have that 
\begin{eqnarray}
  \|z_0^{K} - z^*\|^2 &\leq& \left(1 - \frac{1}{4} \gamma_1 n \mu\right)^K \|z_{0} - z^*\|^2  + \frac{48L_{max}^2}{\mu^2} \left[6n(n-1)+4\right] \frac{16\log^2 (n^{1/2}K)}{\mu^2 n^2K^2}\sigma_*^2\nonumber \\
  &\leq& \left(1 - \frac{1}{4} \gamma_1 n \mu\right)^K \|z_{0} - z^*\|^2  + \mathcal{\Tilde{O}}\left(\frac{1}{K^2}\right)\nonumber
\end{eqnarray}
We now consider the following cases:
\paragraph{Case 1: $\frac{\mu}{10L_{max}^2\sqrt{10n^2+2n+29}}\leq \frac{\log (n^{1/2}K)}{\mu nK}$} In this case we have that $\gamma_1 = \frac{\mu}{10L_{max}^2\sqrt{10n^2+2n+54}}$, which implies that the RHS of \eqref{second_term_edited} is bounded by 
\begin{eqnarray}
    \|z_0^{K} - z^*\|^2 &\leq& e^{-\frac{\gamma_1 nK\mu}{4}} {\|z_0- z^*\|^2} +\mathcal{\Tilde{O}}\left(\frac{1}{K^2}\right)\nonumber \\ 
    &\leq& e^{-\frac{nK\mu^2}{40L_{max}^2\sqrt{10n^2+2n+29}}} {\|z_0- z^*\|^2} +\mathcal{\Tilde{O}}\left(\frac{1}{K^2}\right) \nonumber\\
    &\leq& e^{-\frac{K\mu^2}{40\sqrt{12}L_{max}^2}} {\|z_0- z^*\|^2} +\mathcal{\Tilde{O}}\left(\frac{1}{K^2}\right)\label{case1_ieg}
\end{eqnarray}

\paragraph{Case 2: $\frac{4\log (n^{1/2}K)}{\mu nK} \leq \frac{\mu}{10L_{max}^2\sqrt{10n^2+2n+29}}$}
In this case we have that $\gamma_1 = \frac{4\log (n^{1/2}K)}{\mu nK}$, which implies that the RHS of \eqref{second_term_edited} is bounded by 
\begin{eqnarray}
    \|z_0^{K} - z^*\|^2  &\leq& e^{-\frac{\gamma_1 nK\mu}{4}} {\|z_0- z^*\|^2} +\mathcal{\Tilde{O}}\left(\frac{1}{K^2}\right)\nonumber \\ 
    &\leq& \frac{1}{nK^2}{\|z_0- z^*\|^2} +\mathcal{\Tilde{O}}\left(\frac{1}{K^2}\right) \label{case2_ieg}
\end{eqnarray} 
Taking the maximum of the right-hand side of \eqref{case1_ieg} and \eqref{case2_ieg} and using the inequality $\max\{ a, b\} \leq a + b$, we obtain the desired result which holds for both cases:
\begin{equation}
  \|z_0^{K} - z^*\|^2 \leq e^{-\frac{K\mu^2}{40\sqrt{12}L_{max}^2}} {\|z_0- z^*\|^2} + \frac{1}{K^2}{\|z_0- z^*\|^2} +\mathcal{\Tilde{O}}\left(\frac{1}{K^2}\right)\nonumber  
\end{equation}
Suppressing constant and logarithmic terms, we get the final result
\begin{eqnarray}
    \|z_0^{K} - z^*\|^2 &=& \Tilde{\mathcal{O}}\left( e^{-\frac{K\mu^2}{L_{max}^2}} + \frac{1}{{K^2}}\right)\nonumber
\end{eqnarray}
\end{proof}

\subsection{SEG-RR with Switching Stepsize Rule}
\label{sec: switch}
We, next, provide theorems for the use of a switching stepsize rule in the strongly monotone and affine case that allows us to establish convergence to the exact solution $z^*$. The stepsize rule indicates the use of a constant stepsize at the start of the algorithm in order to converge to a neighborhood around the solution $z^*$ and then switch to a decreasing one with the goal of reducing the neighborhood and converging to the exact solution. 
\begin{theorem}\label{app SC-SC decreasing-steps}
    \label{SC-SC decreasing-steps}
    Suppose that the operator $F$ is $\mu$-strongly monotone, each $F_i, \, \forall i\in[n]$ is $L_i-$Lipschitz and SEG-RR is run with step size $\gamma_{2, k} = 2\gamma_{1, k}$,
    \begin{eqnarray}
      \gamma_{1, k} = 
      \begin{dcases}
           \gamma_{1, max},\text{ for } k < k^*= \ceil{ \frac{64}{\mu^2 \gamma_{1,max}^2}}\\
          \frac{4(2k+1)}{\mu(k+1)^2}, \quad \text{for } k \geq k^* \nonumber
      \end{dcases}
    \end{eqnarray}
    Then, we have that the iterates of SEG-RR satisfy
    \begin{eqnarray}
      \Expe{\|z^{K+1}_0 - z^*\|^2} &\leq& \frac{(k^*)^2 e^{-\frac{\mu n\gamma_{1,max}}{4}}}{(K+1)^2}{\|z_0- z^*\|^2} + \frac{96L_{max}^2(29+n)\sigma_*^2}{\mu^2(K+1)^2}\left(\gamma_{1,max}^2k^{*^2}+\frac{128}{\mu^2} K \right) \nonumber  
    \end{eqnarray}
    where $\gamma_{1, max} = \frac{\mu}{10L_{max}^2\sqrt{10n^2+2n+54}}$.
\end{theorem}
\begin{proof}
    Let $\gamma_{1, k}, \gamma_{2, k}$ be the step size of SEG-RR algorithm in the $k$-th epoch and fix $\gamma_{2, k} = 2\gamma_{1, k}$. Let also $k^* \in \mathbb{Z}_*$ be an epoch at which the stepsize scheme uses the decreasing stepsize $\gamma_{1,k} = \frac{4(2k+1)}{\mu (k+1)^2}$ and satisfies $\gamma_{1,k^*} \leq \gamma_{1,max}$.
    Observe that $\forall k \geq k^*: \gamma_{1,k} \leq \gamma_{1,max}$ and thus inequality \eqref{sc-sc-res-to} holds. 
    
    Substituting $\gamma_1 = \gamma_{1, k}$ and $\gamma_{2, k} = 2 \gamma_{1, k}$ into \eqref{sc-sc-res-to}, we get:
    \begin{eqnarray}
        \Expep{\|z^{k+1}_0 - z^*\|^2}&\stackrel{}{\leq}& 
    \Bigg(1 - \frac{1}{4}\gamma_{1, k} n\mu \Bigg){\|z_{0}^k - z^*\|^2}+\frac{24nL_{max}^2(29+n)}{\mu} \gamma_{1, k}^3\sigma_*^2\nonumber \\ 
        &\stackrel{}{\leq}& \Bigg(1 - \frac{1}{4}\gamma_{1, k} n\mu \Bigg){\|z_{0}^k - z^*\|^2}+\frac{24nL_{max}^2(29+n)}{\mu} \frac{4^3(2k+1)^3}{\mu^3(k+1)^6}\sigma_*^2\quad \quad \label{eq_for_dec_step_subst}
    \end{eqnarray}
        We, then, multiply both sides of \eqref{eq_for_dec_step_subst} by $(k+1)^2$ and obtain
    \begin{eqnarray}
        (k+1)^2\Expep{\|z^{k+1}_0 - z^*\|^2}\stackrel{}{\leq}k^2 \Expe{\|z_{0}^{k^*} - z^*\|^2} + \frac{24nL_{max}^2(29+n)}{\mu^4}\frac{4^3(2k+1)^3}{(k+1)^4}\sigma_*^2\nonumber
    \end{eqnarray}
    Using the inequality $\frac{(2k+1)^3}{(k+1)^4}\leq 8$, we have that:
    \begin{eqnarray}
       (k+1)^2\Expep{\|z^{k+1}_0 - z^*\|^2}\stackrel{}{\leq}k^2 {\|z_{0}^k - z^*\|^2} +\frac{48 \cdot 4^4 L_{max}^2(29+n)}{\mu^4}\sigma_*^2 \nonumber
    \end{eqnarray}
    Rearranging the terms and summing for $k=k^*, ..., K$ we are able to get the telescopic cancellation 
    \begin{eqnarray}
        \sum\limits_{k=k^*}^K \left[(k+1)^2\Expep{\|z^{k+1}_0 - z^*\|^2}- k^2{\|z_{0}^k - z^*\|^2} \right] &\stackrel{}{\leq}&\sum\limits_{k=k^*}^K \frac{48 \cdot 4^4 L_{max}^2 (29+n)\sigma_*^2}{\mu^4}\nonumber
    \end{eqnarray}
    Thus, we have that:
    \begin{eqnarray}
        \Expep{\|z^{K+1}_0 - z^*\|^2} &\stackrel{}{\leq}& \frac{(k^*)^2}{(K+1)^2}\Expep{\|z^{k^*}_0 - z^*\|^2}+\frac{48 \cdot 4^4 L_{max}^2 (29+n)}{\mu^4(K+1)^2} (K-k^*) \sigma_*^2 \quad \quad \nonumber
    \end{eqnarray}
    Taking expectation on both sides and using the tower property, we get
    \begin{eqnarray}
    \Expe{\|z^{K+1}_0 - z^*\|^2} &\stackrel{}{\leq}& \frac{(k^*)^2}{(K+1)^2}\Expe{\|z^{k^*}_0 - z^*\|^2}+\frac{48 \cdot 4^4 L_{max}^2 (29+n)}{\mu^4(K+1)^2}(K -k^*) \sigma_*^2 \quad \quad \label{eq-before-adaptivestepnew}
    \end{eqnarray}
    For $k \leq k^*$ we have that \eqref{sc-sc result-4} holds and thus combining it with \eqref{eq-before-adaptivestepnew} results to
    \begin{eqnarray}
            \Expep{\|z^{K+1}_0 - z^*\|^2} &\stackrel{}{\leq}& \frac{(k^*)^2}{(K+1)^2} \Bigg(1 - \frac{\mu n\gamma_{1,max}}{4} \Bigg)^{k^*} {\|z_0- z^*\|^2} \nonumber \\ 
            &&+\frac{96L_{max}^2(29+n)\sigma_*^2}{\mu^2(K+1)^2}\left(\gamma_{1,max}^2{k^*}^2+\frac{128}{\mu^2} (K -k^*)\right)\label{K*}
    \end{eqnarray} 
    Using the inequality $(1-x)^{x} \leq e^{-x}$, we get:
    \begin{eqnarray}
       \Expe{\|z^{K+1}_0 - z^*\|^2} &\leq& \frac{(k^*)^2 e^{-\frac{\mu n\gamma_{1,max}}{4}}}{(K+1)^2}{\|z_0- z^*\|^2} + \frac{96L_{max}^2(29+n)\sigma_*^2}{\mu^2(K+1)^2}\left(\gamma_{1,max}^2k^{*^2}+\frac{128}{\mu^2} K \right) \label{k**}
    \end{eqnarray}
    In the above, we choose $k^*$ so that it minimizes the second term in \eqref{k**} 
    and thus $k^* = \ceil{ \frac{64}{\mu^2\gamma_{1,max}^2}}$.
\end{proof}

Next, we provide convergence guarantees for a switching stepsize rule in the affine case. 

\begin{theorem}\label{app Bilinear-decreasing stepsimpl}
    \label{bilinear decreasing-steps}
Suppose that each $F_i, \, \forall i\in [n]$ is monotone, affine and $L_i-$Lipschitz. If SEG-RR is run with step size $\gamma_{2, k} = 4\gamma_{1, k}$,
    \begin{eqnarray}
      \gamma_{1, k} = 
      \begin{dcases}
           \gamma_{1, max},\text{ for } k < k^*= \ceil{ \frac{16}{{\lambda_{min}^+(Q)}^2\gamma_{1,max}^2}}\\
          \frac{2(2k+1)}{\lambda_{min}^+(Q)(k+1)^2}, \quad \text{for } k \geq k^* \nonumber
      \end{dcases}
    \end{eqnarray}
    then we have that the iterates of SEG-RR satisfy
    \begin{eqnarray}
      \Expe{\|z^{K+1}_0 - z^*\|^2} &\leq& \frac{(k^*)^2}{(K+1)^2} e^{- \frac{1}{2}\gamma_1nk^*\lambda_{min}^+(Q)} \Expe{\|z_0- z^*\|^2} \nonumber \\ 
            && + \frac{8\left(24n -23+\frac{1}{n}\right)L_{max}}{{\lambda_{min}^+(Q)}^2(K+1)^2}  \sigma_*^2 \left[\gamma_{1, max}^2k^{*^2}+\frac{32(K-k^*)}{\lambda_{min}^2(Q)} \right] \nonumber  
    \end{eqnarray}
    where $\gamma_{1, max} = \frac{\lambda_{min}^+(Q)}{2\sqrt{120} nL^2_{max}}$.
\end{theorem}
\begin{proof}
    Let $\gamma_{1,k}, \gamma_{2, k}$ be the step size of SEG-RR algorithm in the $k$-th epoch and fix $\gamma_{2, k} = 4 \gamma_{1, k}$. Let also $k^* \in \mathbb{Z}_*$ be an epoch at which the stepsize scheme uses the decreasing stepsize $\gamma_{1,k} = \frac{2(2k+1)}{\lambda_{min}^+(Q)(k+1)^2}$.
    Observe that $\forall k \geq k^*: \gamma_{1,k} \leq \gamma_{1,max}$ and thus inequality \eqref{billast} holds. \\
    Substituting $\gamma_1 = \gamma_{1, k}$ and $\gamma_2 = 4 \gamma_{1, k}$ in \eqref{billast} we get:
    \begin{eqnarray}
        \Expep{\|z^{k+1}_0 - z^*\|^2} &\stackrel{}{\leq}& \left[1 - \frac{2k+1}{(k+1)^2} \right] \|z_{0}^k - z^*\|^2\nonumber \\
        && + \frac{4 \left(24n^2 -23n+1\right)L_{max}\sigma_*^2}{\lambda_{min}^+(Q)}  \left(\frac{2(2k+1)}{\lambda_{min}^+(Q)(k+1)^2}\right)^3 \nonumber
    \end{eqnarray}
    Multiplying both sides with $(k+1)^2$ results to
    \begin{eqnarray}
        (k+1)^2\Expep{\|z^{k+1}_0 - z^*\|^2}\stackrel{}{\leq}k^2 \Expep{\|z_{0}^{k^*} - z^*\|^2} +\frac{32 \left(24n^2 -23n+1\right)L_{max}\sigma_*^2}{{\lambda_{min}^+(Q)}^4} \frac{(2k+1)^3}{(k+1)^4}\nonumber
     \end{eqnarray}
    Using the inequality $\frac{(2k+1)^3}{(k+1)^4}\leq \frac{2^3(k+1)^3}{(k+1)^4} \leq 8$ we have that:
    \begin{eqnarray}
       (k+1)^2\Expep{\|z^{k+1}_0 - z^*\|^2}\stackrel{}{\leq}k^2 {\|z_{0}^k - z^*\|^2} +\frac{256 \left(24n^2 -23n+1\right)L_{max}\sigma_*^2}{{\lambda_{min}^+(Q)}^4} \nonumber
    \end{eqnarray}
    Rearranging the terms and summing for $k=k^*, ..., K$, we are able to get the telescopic cancellation 
    \begin{eqnarray}
        \sum\limits_{k=k^*}^K \left[(k+1)^2\Expep{\|z^{k+1}_0 - z^*\|^2}- k^2{\|z_{0}^k - z^*\|^2} \right] &\stackrel{}{\leq}&\sum\limits_{k=k^*}^K \frac{256 \left(24n^2 -23n+1\right)L_{max}\sigma_*^2}{{\lambda_{min}^+(Q)}^4}\nonumber
    \end{eqnarray}
    Thus, we get that:
    \begin{eqnarray}
        \Expep{\|z^{K+1}_0 - z^*\|^2} &\stackrel{}{\leq}& \frac{(k^*)^2}{(K+1)^2}\Expep{\|z^{k^*}_0 - z^*\|^2}+\frac{256 \left(24n^2 -23n+1\right)L_{max}\sigma_*^2}{{\lambda_{min}^+(Q)}^4(K+1)^2} \nonumber
    \end{eqnarray}
    Taking expectation on both sides and using the tower property, we obtain:
    \begin{eqnarray}
    \Expe{\|z^{K+1}_0 - z^*\|^2} &\stackrel{}{\leq}& \frac{(k^*)^2}{(K+1)^2}\Expe{\|z^{k^*}_0 - z^*\|^2} +\frac{256\left(24n^2 -23n+1\right)L_{max}\sigma_*^2}{{\lambda_{min}^+(Q)}^4(K+1)^2}(K -k^*) \sigma_*^2 \quad \quad \label{eq-before-adaptivestep_bil}
    \end{eqnarray}
    For $k \leq k^*$ we have that \eqref{final result bil1} holds with $\gamma_1 = \gamma_{1, max}$
    \begin{eqnarray}
         \Expe{\|z_0^{k^*+1} - z^*\|^2 }&\leq&\left[1 - \frac{1}{2}\gamma_{1, max} n\lambda_{min}^+(Q)\right]^{k^*} {\|z_0- z^*\|^2}
   + \frac{8\left(24n -23+\frac{1}{n}\right)L_{max}}{\lambda^2_{min}(Q)} \gamma_{1,max}^2 \sigma_*^2    \nonumber
    \end{eqnarray}
    and thus combining it with \eqref{eq-before-adaptivestep_bil} we get: 
    \begin{eqnarray}
            \Expep{\|z^{K+1}_0 - z^*\|^2} &\stackrel{}{\leq}& \frac{(k^*)^2}{(K+1)^2} \left(1 - \frac{1}{2}\gamma_1n\lambda_{min}^+(Q)\right)^{k^*} \Expe{\|z_0- z^*\|^2} \nonumber \\ 
            && + \frac{8\left(24n -23+\frac{1}{n}\right)L_{max}}{\lambda^2_{min}(Q)(K+1)^2}  \sigma_*^2 \left[\gamma_{1, max}^2k^{*^2}+\frac{32(K-k^*)}{\lambda_{min}^2(Q)} \right]\nonumber \\
            &\stackrel{\eqref{ineq exponential}}{\leq}& \frac{(k^*)^2}{(K+1)^2} e^{- \frac{1}{2}\gamma_1nk^*\lambda_{min}^+(Q)} \Expe{\|z_0- z^*\|^2} \nonumber \\ 
            && + \frac{8\left(24n -23+\frac{1}{n}\right)L_{max}}{\lambda^2_{min}(Q)(K+1)^2}  \sigma_*^2 \left[\gamma_{1, max}^2k^{*^2}+\frac{32(K-k^*)}{\lambda_{min}^2(Q)} \right] \quad \quad \label{K*_bil}
    \end{eqnarray}
    Lastly, we choose $k^*$, so that it minimizes the second term in \eqref{K*_bil} and thus $k^* = \ceil{ \frac{16}{\lambda_{min}^2(Q)\gamma_{1,max}^2}}$.
\end{proof}
\newpage
\section{On Experiments}
\label{app experiments}

 In Appendix~\ref{bocasla}, we provide more details on the experiments discussed in the main paper. In Appendix \ref{app add experiments}, we run more experiments to evaluate the performance of \ref{eq:SEG-RR}. As stated in the main paper, the code for reproducing our experimental results is available at \url{https://github.com/emmanouilidisk/Stochastic-ExtraGradient-with-RR}.

\subsection{Experimental Details}
\label{bocasla}
We first describe our experimental setup. In the strongly monotone setting, we consider the following quadratic problem:
\begin{equation}
    \min_{x \in \mathbb{R}^{d}} \max_{y \in \mathbb{R}^{d}} \frac{1}{n}\sum\limits_{i=1}^n \frac{1}{2} x^\top A_i x + x^\top B_i y - \frac{1}{2} y^\top C_i y + a_i^\top x - c_i^\top y \nonumber
\end{equation}
We sample the matrices $A_i$ by first sampling an orthogonal matrix $P$ and then sampling a diagonal matrix $D_i$ with elements in the diagonal uniformly sampled from the interval $[\mu, L]$. Here, the parameters $\mu, L$ correspond to the strong monotonicity parameter and the Lipschitz parameter of the problem. We acquire the matrices $A_i$, as the product $A_i = P D_i P^{T}$. We sample the matrices $B_i, C_i$ similarly to sampling the matrices $A_i$ with the only difference that the elements of $D_i$ lie in the interval $[0, 0.1]$ and $[\mu, L]$ respectively. The vectors $a_i, c_i$ are sampled from the normal distribution $\mathcal{N}(0, I)$. In all experiments, we use $n=100, d = 100$, while we specify the values of $\mu, L$ in each experiment independently as they differ.

In the bilinear regime, we focus on the following two-player zero-sum game:
\begin{eqnarray}
    \min_{x \in \mathbb{R}^{d}} \max_{y \in \mathbb{R}^{d}} \frac{1}{n}\sum\limits_{i=1}^n x^\top B_i y + a_i^\top x - c_i^\top y \nonumber
\end{eqnarray}
We let the matrices $B_i$ be $B_i = P D_i P^T$, where $P$ is an orthogonal matrix and $D_i$ a diagonal matrix with elements in the diagonal selected uniformly at random from $[\lambda_{min}^+(Q), L_{max}]$. We specify that the parameters $\lambda_{min}^+(Q), L_{max}$ correspond to the parameters $\lambda_{min}^+(Q), L_{max}$ of Theorem \ref{thm-bil-1}. Regarding the vectors $a_i, c_i$, they are sampled from the normal distribution $\mathcal{N}(0, I)$. In all experiments, we use $n = 100$ and let $d=1$ in the two-dimensional experiments; otherwise, $d = 100$. We specify individually for each experiment the parameters $\lambda_{min}^+(Q), L_{max}$ that have been used.
\subsection{Additional Experiments}
\label{app add experiments}
In this part, we provide additional experiments to the ones presented in the main paper. 

\paragraph{SC - SC Problems.}
We initially focus on Strongly Convex - Strongly Concave (SC - SC) minimax problems and compare the different without-replacement sampling variants of SEG, namely SEG-RR, SEG-SO and IEG, with \ref{eq:S_SEG} (denoted as \textit{SEG} in the plots). 
\begin{figure}[H]
\centering
\includegraphics[width=160mm]{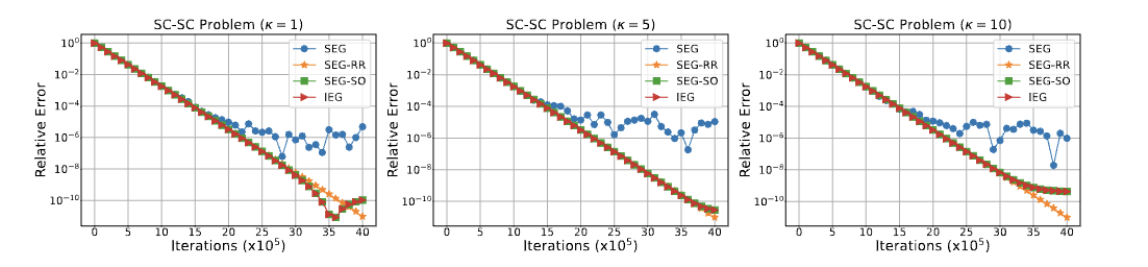}
\caption{SC-SC problem. SEG-RR, SEG-SO, SEG-IG and SEG run with step size as in Theorem \ref{thm SC-SC4} for problem with $\mu=1, n=100$ and different condition numbers.
SEG-RR achieves smaller relative error in comparison to the other without-replacement sampling methods.}
\end{figure}

Since SEG-RR seems to achieve at least as small (if not smaller) error than the other without-replacement variants it makes sense to use SEG-RR in practice. In this way, we will focus for the rest of this section on experiments comparing SEG-RR with SEG. 
We start by exploiting the behaviour of with and without-replacement sampling for the step size suggested by the analysis of SEG in \citet{gorbunov2022stochastic}. 
We observe that even for the theoretical step size SEG-RR performs better than SEG in terms of relative error. 

\begin{figure}[H]
\centering
\includegraphics[width=160mm]{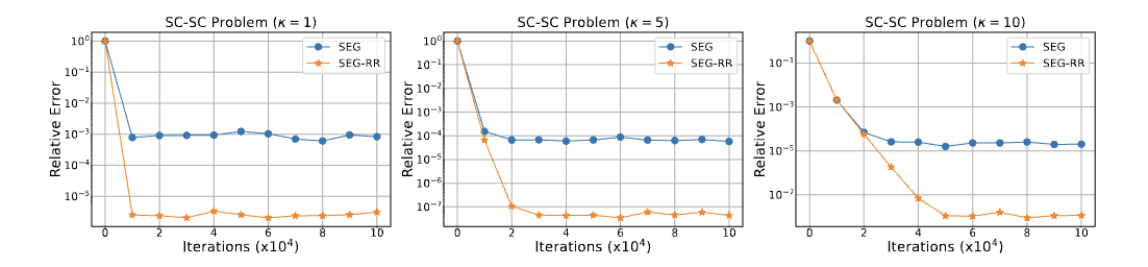}
\caption{SC-SC problem. SEG-RR vs SEG run with step size as in \citet{gorbunov2022stochastic} for problem with $\mu=1, n=100$ and different condition numbers.}
\end{figure}

We, next, compare SEG-RR with SEG. We conduct experiments for problems with different condition numbers $\kappa = \{ 1,5,10,100\}$ and different step size $\gamma_1 = \{\frac{1}{10L_{max}}, \frac{1}{100L_{max}}, \frac{1}{1000L_{max}}\}$. In this vein, we fix $\mu = 1$ and let the Lipschitz parameter vary as the condition number $\kappa = \frac{L}{\mu}$ changes. 

\begin{figure}[H]
\centering
\includegraphics[width=165mm]{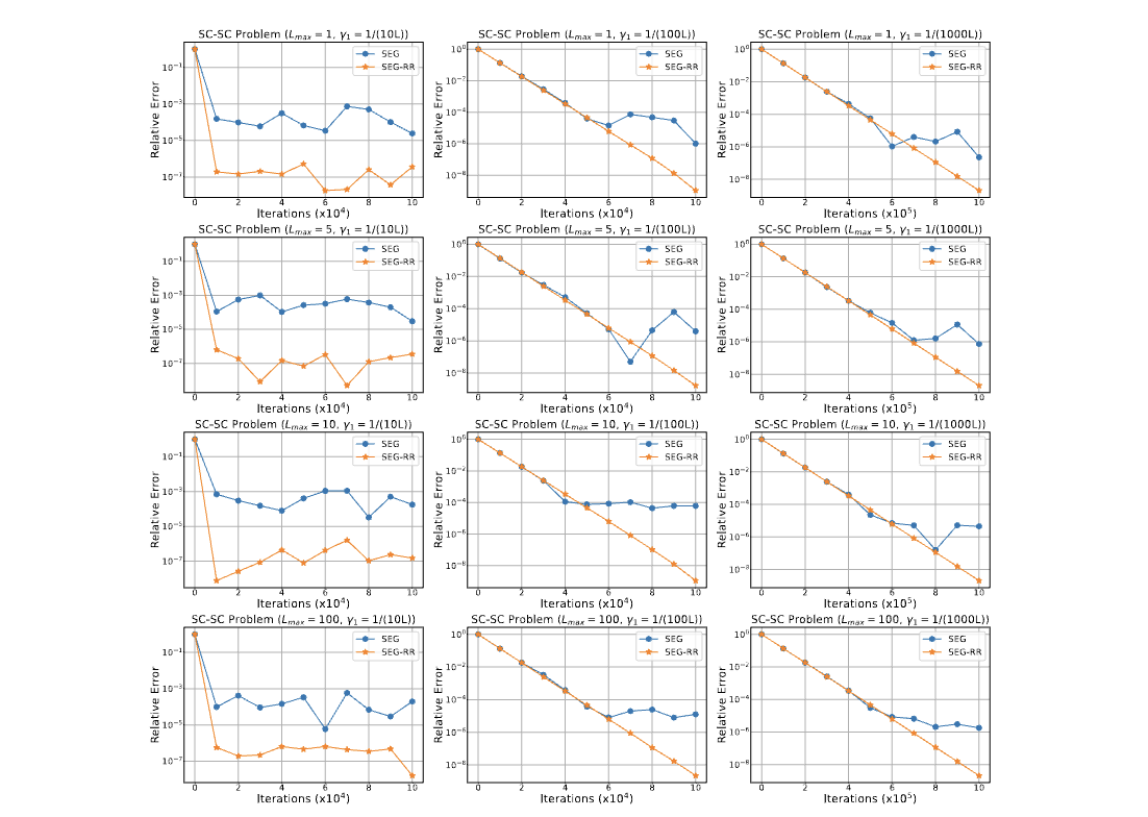}
\caption{SC - SC Problems: Comparison of SEG-RR and SEG with different step size. Each row corresponds to a problem with a different condition number $\kappa = \{ 1,5,10,100\}$, while each column corresponds to a specific stepsize $\gamma_1 = \{\frac{1}{10L}, \frac{1}{100L}, \frac{1}{1000L}\}$. In all problems $\mu = 1, n=100$.}
\end{figure}
\newpage

\paragraph{Bilinear Games.}
We, first, provide experiments comparing SEG-RR, SEG-SO, IEG with SEG. We use as step size in all algorithms the step size suggested in Theorem \ref{thm-bil-1}. We fix $\lambda_{min}^+(Q) = 1$ and let the Lipschitz constant of the problem vary. 
\begin{figure}[H]
\centering
\includegraphics[width=170mm]{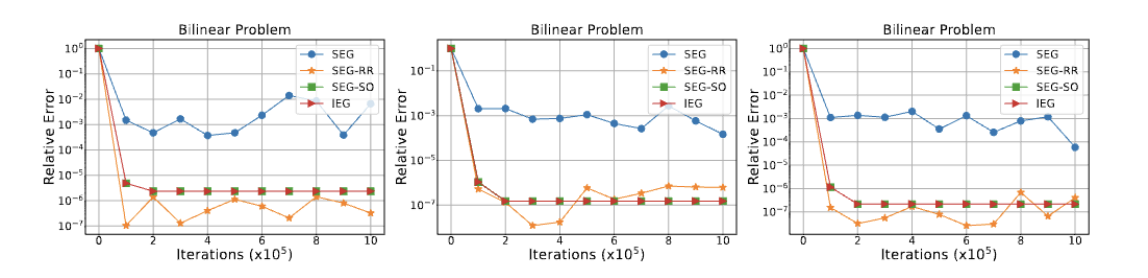}
\caption{Bilinear games. The without-replacement variants of SEG achieve better performance in terms of relative error in comparison to \ref{eq:S_SEG}. The problem parameters $\lambda_{min}^+(Q)=1, n=100$ for different $L_{max}=\{1, 5, 10\}$ in the plots from left to right accordingly and with
step size the ones in Theorem \ref{thm-bil-1}.}
\end{figure}

It is obvious that the without-replacement sampling variants of SEG converge with with smaller relative error than the uniform with-replacement variant for the same number of iterations/epochs.

We, next, provide experiments for SEG-RR and SEG for problems with different Lipschitz parameters $L_{max} = \{ 1, 5, 10\}$ using the theoretical step size $\gamma_1 = \frac{0.1}{(t+19)^{r_\eta}}, \gamma_2 = \frac{1}{(t+19)^{r_{\gamma}}}$ where $r_{\gamma} = 0, r_{\eta} =0.7$ suggested in the analysis of SEG for bilinear games in \citet{hsieh2020explore}. 

\begin{figure}[H]
\centering
\includegraphics[width=160mm]{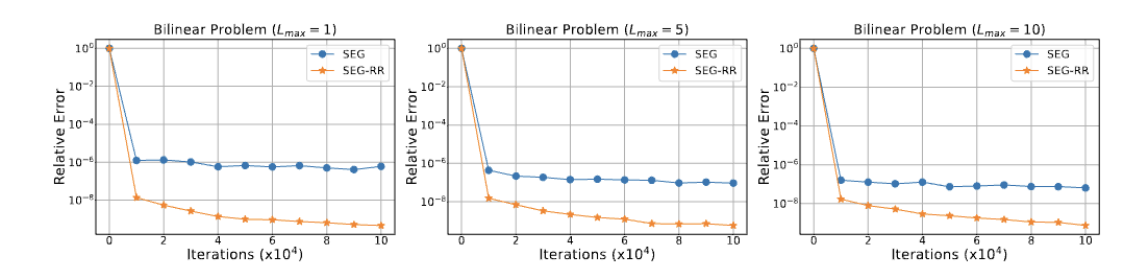}
\caption{Bilinear games with $\lambda_{min}^+(Q)=1, n=100$ for different $L_{max}$ and with
step sizes as in Theorem \ref{thm-bil-1}.}
\end{figure}
It is obvious that for the stepsizes suggested by theory SEG-RR achieves a smaller relative error for the same number of epochs/iterations in comparison to SEG. 

We, lastly, conduct experiments for step sizes larger than the theoretical ones and for a number of different problem instances to capture the performance of SEG-RR and SEG in a broad range of step sizes and problem setups. We run experiments for problems with different $L = \{1, 5, 10\}$ and for step sizes $\gamma_2 = 4 \gamma_1$ with $\gamma_1 = \{\frac{1}{10L_{max}}, \frac{1}{100L_{max}}\}$.
\newpage
\begin{figure}[H]
\centering
\includegraphics[width=180mm]{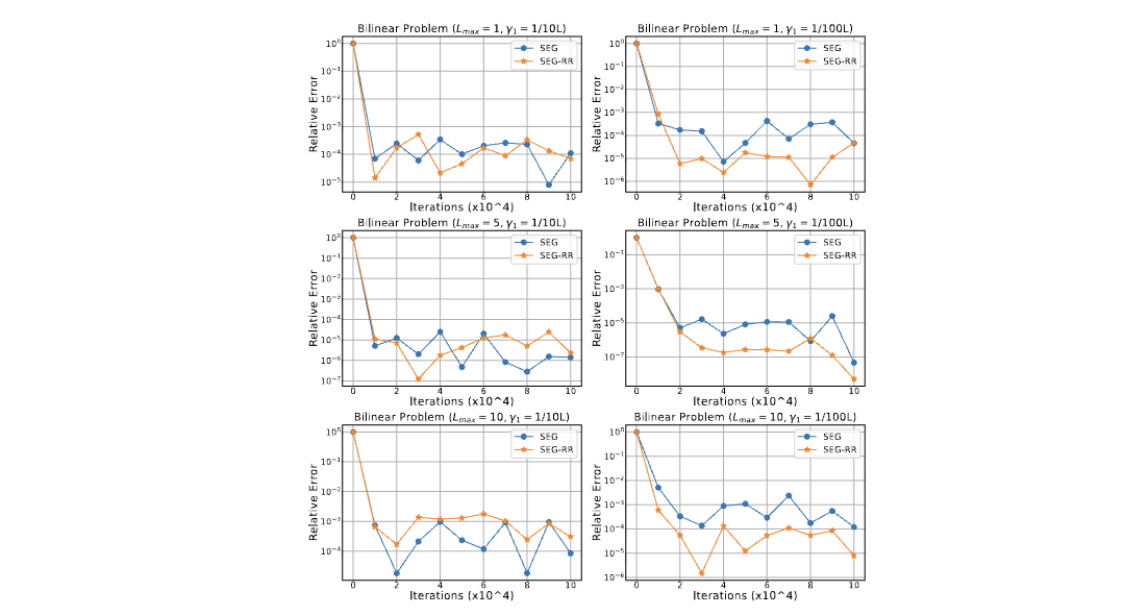}
\caption{Bilinear games for different $L_{max} = \{1, 5, 10\}$ and with
step sizes $\gamma_1  = \{\frac{1}{10L}, \frac{1}{100L}\}$. In all problems $\lambda_{min}^+(Q) =1, n=100$.}
\end{figure}

In the above plots, it is obvious that in most cases SEG-RR achieves at least as good (if not smaller) relative error than \ref{eq:S_SEG}, which advocates for the use of random reshuffling in practice. 

\end{document}